\documentclass{article}

\usepackage{mystyle}

\usepackage[english]{babel}

\usepackage[letterpaper,top=2cm,bottom=2cm,left=3cm,right=3cm,marginparwidth=1.75cm]{geometry}
\usepackage{comment}

\newtheorem{lemma}{Lemma}
\newtheorem{theorem}{Theorem}
\newtheorem{corollary}{Corollary}
\newtheorem{proposition}{Proposition}
\newtheorem{definition}{Definition}
\newtheorem{example}{Example}
\newtheorem{remark}{Remark}
\newtheorem{assumption}{Assumption}

\crefname{theorem}{Theorem}{Theorems}
\crefname{proposition}{Proposition}{Propositions}
\crefname{remark}{Remark}{Remarks}
\crefname{definition}{Definition}{Definitions}
\crefname{corollary}{Corollary}{Corollaries}
\crefname{lemma}{Lemma}{Lemmas}
\crefname{example}{Example}{Examples}
\crefname{assumption}{Assumption}{Assumptions}

\renewcommand{\d}[1]{\mathrm{d}}

\renewcommand*\d{\mathop{}\!\mathrm{d}}

\newcommand{\COT}{\mathrm{COT}}
\newcommand{\co}{\mathrm{co}}
\newcommand{\loc}{\mathrm{loc}}

\newcommand{\RN}[1]{
\textup{\uppercase\expandafter{\romannumeral#1}}
}

\newcommand{\PLeb}{\Pp^{\mathrm{Leb}}_2([0,1] \times \Theta)}
\newcommand{\Leb}{\mathrm{Leb}}
\newcommand{\TV}{\mathrm{TV}}

\title{Training Infinitely Deep and Wide Transformers}
\author{
    Raphaël Barboni \\
    Bocconi University \\
    \texttt{raphael.barboni@unibocconi.it}
    \and
    Takashi Furuya \\
    Doshisha University, RIKEN AIP\\
    \texttt{tfuruya@mail.doshisha.ac.jp}
    \and
    Maarten V. de Hoop \\
    Rice University \\
    \texttt{mdehoop@rice.edu}
    \and
    Gabriel Peyr\'e \\
    CNRS, ENS, PSL Universit\'e\\
    \texttt{gabriel.peyre@ens.fr}}
\date{}

\begin{document}

\maketitle

\begin{abstract}
Transformers have become the dominant architecture in modern machine learning, yet the theoretical understanding of their training dynamics remains limited, especially for deep architectures with many layers and attention heads.
This paper develops a rigorous mathematical framework for analyzing gradient-based training of transformers in the mean-field regime, where both the depth (number of layers) and width (number of attention heads) tend to infinity.
In fact, we extend the mean-field training theory for ResNets to the transformer setting. While ResNet training can be understood as controlling a neural ordinary differential equation (ODE), transformer training corresponds to controlling a neural partial differential equation (PDE), due to the coupling of multiple token distributions through the attention mechanism. Our mean-field model features two types of measure representations: token distributions evolving through layers (represented as measures over the token space), and attention parameters at each layer (represented as measures over the attention parameter space).
We establish well-posedness of the forward pass through infinitely deep transformers, characterizing token evolution via flow maps that satisfy ODEs in function spaces. Using adjoint sensitivity analysis, we derive an explicit formula for the conditional Wasserstein gradient of the training risk, involving adjoint variables governed by backward ODEs. We prove the existence and uniqueness of gradient flow curves in the conditional Wasserstein metric space, establishing a rigorous foundation for gradient-based transformer training. A key technical contribution is providing necessary and sufficient conditions for injectivity of the Neural Tangent Kernel (NTK) for attention mechanisms: we show that NTK injectivity is equivalent to linear independence of log-sum-exp functions modulo affine functions, a condition satisfied by diverse token distributions, including discrete distributions, uniform distributions, and Gaussian mixtures. Under this NTK injectivity assumption, we prove that gradient flow converges to global minima when the initial loss is sufficiently small, eliminating spurious local minima from the optimization landscape.
\end{abstract}

\section{Introduction}

Transformers~\cite{vaswani2017attention} have emerged as the dominant architecture in modern machine learning, powering breakthroughs in natural language processing, computer vision, and beyond. Despite their empirical success, the theoretical understanding of transformer training dynamics remains limited, particularly for deep architectures with many layers and large numbers of attention heads. Understanding when and why gradient-based optimization converges for transformers is essential for developing principled design choices and training procedures.

\subsection{Previous Works}

\paragraph{Infinite Width: Mean Field over Parameters}

The analysis of infinitely wide neural networks depends on the scaling normalization of the weights, leading to distinct mathematical representations.
The Neural Tangent Kernel (NTK) regime~\cite{jacot_neural_2021} arises when weights scale as $1/\sqrt{\text{width}}$. In this lazy training regime, the network remains close to initialization throughout optimization, and training dynamics reduce to kernel regression in a reproducing kernel Hilbert space (RKHS). The kernel is fixed by the architecture and initialization, and the problem becomes linear in the RKHS. This Euclidean structure is mathematically simpler and enables global convergence guarantees~\cite{du2019gradient,allen2019convergence}, but precludes genuine feature learning since the network's representational capacity is frozen at initialization.
The mean-field regime corresponds to a different scaling of $1/\text{width}$, where network parameters are represented as a probability distribution over the parameter space~\cite{mei2018mean,rotskoff2018parameters,chizat2018global}. Unlike the Euclidean NTK setting, this mean-field formulation is inherently non-Euclidean: training dynamics evolve on the Wasserstein manifold of probability measures and are governed by continuity equations or Wasserstein gradient flows. This richer regime allows for feature learning, as parameters can move substantially during training. Global convergence results have been established for shallow mean-field networks under appropriate regularity conditions~\cite{chizat2018global,mei2018mean}.
In the present paper, we adopt the mean-field representation to model transformers with infinitely many attention heads: each layer contains a probability distribution over attention parameters (query, key, value matrices). For our convergence analysis, although the model is formulated in the mean-field framework, we also leverage tools from NTK theory by studying the Neural Tangent Kernel of the attention mechanism and establishing conditions for its injectivity.

\paragraph{Infinite Depth: Neural ODE Models}
Understanding training dynamics of deep neural networks has been a central challenge in deep learning theory. Initial results on training deep networks~\cite{du2019gradient,allen2019convergence} established convergence guarantees in overparameterized regimes, typically under conditions analogous to the NTK regime where the network remains close to initialization. 
The introduction of residual connections~\cite{he2016deep} was crucial for training very deep networks in practice. Theoretically, residual networks admit an interpretation as discretizations of continuous dynamical systems. Chen \textit{et al}.~\cite{chen2018neural} introduced Neural ODEs, viewing ResNets in the limit of infinitely many layers as ordinary differential equations on the feature space. The continuous-time perspective depends on the scaling: with residual weights scaling as $1/\text{depth}$, one obtains a neural ODE, while $1/\sqrt{\text{depth}}$ scaling leads to a neural SDE~\cite{marion2022scaling}.
Training dynamics in the neural ODE regime, where residual blocks are themselves mean-field MLPs, have been rigorously analyzed by~\cite{barboni2024understanding}. They established well-posedness of the gradient flow via the theory of curves of maximal slope in the conditional Wasserstein metric, and proved local convergence (when the initial loss is sufficiently small) under a Polyak-Łojasiewicz condition. In our paper, we adopt the same framework for infinitely deep transformers, with a crucial distinction: the data (tokens) are also represented in the mean-field limit. Unlike ResNets that operate on finite-dimensional feature vectors via a neural ODE, transformers in our formulation process token distributions through a neural PDE that couples the evolution of multiple probability measures. 

\paragraph{Infinite Number of Tokens: Mean Field over Tokens and Neural PDE Frameworks}
Dealing with very large numbers of tokens is highly relevant for modern transformer applications: long-context language models process thousands of tokens, vision transformers handle thousands of image patches. 
In the encoder (non-causal) setting, a natural approach is to represent the token distribution as a probability measure over the token space. This perspective was pioneered by~\cite{sander2022sinkformers}, who formulated the evolution of tokens through self-attention layers as a non-linear PDE on the space of probability measures. We refer to this approach as ``mean field over tokens''. In the present work, it is crucial to note that we employ mean field in a dual sense: both over tokens (infinitely many tokens) and over parameters (infinitely many attention heads).
Building on this mean-field token framework,~\cite{geshkovski2023emergence} analyzed the emergence of clustering behavior in the PDE governing token evolution, proving that attention mechanisms drive tokens toward discrete clusters.
From an expressivity perspective,~\cite{furuya2024transformers} established that deep transformers are universal approximators for in-context mappings from probability measures over tokens to output tokens.

\paragraph{Optimal Transport and Wasserstein Gradient Flows}

Flows over measures provide a unifying framework for understanding both training dynamics and transformer computations. In mean-field neural networks, Wasserstein gradient flows~\cite{jordan1998variational,ambrosio2008gradient} describe the training dynamics: gradient descent on parameters induces a gradient flow on the measure representing the parameter distribution~\cite{chizat2018global}. For transformers, token evolution through layers can be formulated as continuity equations with velocity fields determined by attention~\cite{sander2022sinkformers,geshkovski2023emergence}, though these are not generally Wasserstein flows.
For deep architectures, the layer structure must be preserved during training: parameters at different depths should not mix. This motivates conditional optimal transport~\cite{barboni2024understanding}. We adopt this framework for transformers, with the additional complexity that tokens themselves evolve as measures, requiring a neural PDE formulation.

\paragraph{Training of Transformers}
Most existing results on transformer training focus on in-context learning, a behavior first demonstrated experimentally in \cite{garg2022can}. Subsequent works have shown that transformers are expressive enough to exhibit this phenomenon \cite{akyurek2022learning}, that shallow but optimal transformers already possess this capability \cite{mahankali2023one}, and, in more restricted settings, that the training dynamics converge to globally optimal parameters \cite{ahn2023transformers}.
The works most closely related to ours are \cite{wu2023convergence} and its refinement in \cite{song2024unraveling}, which establish convergence to a global minimum for the training of a single attention layer under an NTK-type parameter rescaling (equivalently, sufficiently small initialization loss). This line of work is further extended in \cite{qin2025convergence,gao2024global}, which study deep transformers by incorporating skip connections.
Our work adopts a similar NTK-based perspective but goes beyond these results by accommodating an arbitrarily large (possibly infinite) number of residual blocks, as well as an arbitrary number of tokens.

\subsection{Our Contributions}

This paper extends the mean-field training theory of~\cite{barboni2024understanding} from ResNets to transformers. While ResNet training corresponds to controlling a neural ODE (an ordinary differential equation on the token space), transformer training corresponds to controlling a neural PDE (a partial differential equation coupling the evolution of multiple token distributions). Our mean-field model represents transformers as both infinitely deep (parameterized by a depth variable) and infinitely wide (parameterized by a probability distribution of attention parameters), with the key distinction that there are two types of measure representations: tokens evolving through layers, and attention parameters at each layer.

\textbf{Training Mean-Field Transformers (\cref{sec:training}):}
In~\cref{subsec:mean_field_transformers}, we introduce the mean-field formulation for infinitely deep and wide transformers, where token distributions evolve via a continuity equation with mean-field attention operators.~\Cref{prop:existence_uniqueness} establishes well-posedness of the forward process, showing that token evolution is characterized by flow maps satisfying ODEs in function spaces.~\Cref{thm:upper_gradient} derives the conditional Wasserstein gradient of the training risk using adjoint sensitivity analysis, obtaining an explicit upper-gradient formula involving adjoint variables that satisfy backward ODEs.~\Cref{thm:gradient_flow_wellposed} proves existence and uniqueness of gradient flow curves for training, and~\cref{thm:maximal_slope_equivalence} establishes equivalence between the gradient flow continuity equation and the curve of maximal slope characterization.~\Cref{thm:local_convergence} then proves local linear convergence to a global minimizer when the $V$-part of the NTK is positive at initialization and the initial loss is sufficiently small.

\textbf{Injectivity of the Neural Tangent Kernel for Attention (\cref{sec:NTK_expressivity}):}
A key technical challenge is ensuring the NTK remains positive definite throughout training. In the section on analysis for injective NTK of attention, we establish necessary and sufficient conditions for injectivity of the NTK map for attention mechanisms. The main proposition (\cref{prop:equivalence_injectivity_cumulant}) shows that NTK injectivity is equivalent to linear independence of log-sum-exp (cumulant generating) functions modulo affine functions. This condition holds for diverse token distributions: discrete distributions satisfying a distinctness condition (\cref{eq:ass:dist}, verified almost surely for i.i.d. samples from absolutely continuous measures discussed in \cref{rem-Justification}), uniform distributions on cubes with distinct radii, symmetric multivariate Laplace distributions with distinct covariance matrices (\cref{lem:Vandermonde}), and two-component Gaussian mixtures with distinct separation parameters. These results are further stabilized under convolution with centered Gaussians (\cref{lem:stability-Gaussian}). The injectivity conditions provide constructive guidance for initialization and ensure convergence of gradient-based training to global minima when the initial loss is sufficiently small.

\paragraph{Notations}

For a metric space $X$, $\Pp(X)$ is the set of Borel probability measures on $X$, $\Pp_c(X)$ is the subset of compactly supported probability measures and, for $p \geq 1$, $\Pp_p(X)$ is the subset of probability measures with finite $p$-th order moment.
When $X$ is a subset of a Banach space this $p$-th order moment is denoted $\Ee_p(\rho) \eqdef \int_X \| x \|^p \d \rho(x)$, for $\rho \in \Pp_p(X)$.
We denote by $\Ww_p$ the \emph{Wasserstein distance} on $\Pp_p(X)$~\cite{villani2009optimal,santambrogio2015optimal}.
For $d \geq 1$, the set of $d$-dimensional real finite vector measures is denoted by $\Mm^d(X)$ and the total variation norm is denoted by $\|.\|_{\TV}$.
For a measure $\rho$ on $X$ and a measurable map $f : X \to Y$ between two metric spaces, $f_\# \rho$ denotes the pushforward of $\rho$ by $f$.
For $x \in X$, we denote by $\delta_x \in \Pp(X)$ the Dirac measure at $x$.
For two Banach spaces $\Bb, \Bb'$, a subset $X \subset \Bb$ and $k \geq 0$, we denote by $\Cc^k(X, \Bb')$ the space of $k$-times continuously differentiable functions $f : X \to \Bb'$.
This space is provided with the norm $\| f \|_{\Cc^k(X, \Bb')} = \sup_{0 \leq l \leq k} \sup_{x \in X} \| f^{(l)}(x) \|$.
All norm subscripts may be omitted when clear from the context.


\section{Training mean-field models of Transformers}
\label{sec:training}

We introduce in this section mean-field models of infinitely deep and arbitrarily wide Transformers (\cref{subsec:mean_field_attention,subsec:mean_field_transformers}), show how the associated training dynamics can be interpreted as a gradient flow on the space of parameter distributions (\cref{subsec:gradient}) and give conditions for convergence of this gradient flow (\cref{subsec:convergence}).
A detailed study of our convergence assumptions will be the subject of~\cref{sec:NTK_expressivity}.

\subsection{Mean-field Attention layers} \label{subsec:mean_field_attention}

We consider in this paper Attention layers where a supplementary bias parameter is added to the \emph{query matrix} and where the \emph{key matrix} is fixed to $K = \Id$.
Given a distribution $\mu$ of tokens and an input token $x \in \RR^d$, we hence consider Attention maps defined by:
\begin{align} \label{eq:attention}
     \varphi_\theta[\mu](x) \eqdef \frac{\int e^{\left< Q x + q, y \right>} V y \, d\mu(y)}{\int e^{\left< Q x + q, y \right>} \, \text{d}\mu(y)} .
\end{align}
In the above equation, the Attention layer is parameterized by the triplet $\theta = (Q, q, V)$, living in the parameter set
$\Theta \eqdef \mathbb{R}^{d \times d} \times \RR^d \times \mathbb{R}^{d \times d}$.

\begin{remark} \label{rem:attention_modifications}
    We introduce two modifications w.r.t.\@ standard parameterizations of Attention layers:
    \begin{itemize}
        \item The fact that the matrix $K$ is fixed is introduced in this section mainly for technical reasons.
        The main purpose is to ensure appropriate regularity and integrability of the Attention w.r.t. its parameters (see~\cref{lem:differential_A,lem:A_theta_param_differential,lem:A_rho_wasserstein_differential})
        for the gradient flow dynamics to be well-defined.
        Note that this simply amounts to performing a change of variables and to replacing $(Q,q)$ by $(K^\top Q, K^\top q)$.
        In particular, it does not affect the expressivity of Attention.
        \item In contrast, the supplementary bias parameter $q$ is precisely introduced to enhance the expressivity of Attention layers.
        In particular, it breaks some invariance by symmetry in~\cref{eq:attention}.
        The parameter $q$ thus plays a key role in proving the positivity of the tangent kernel in~\cref{sec:NTK_expressivity}.
    \end{itemize}
\end{remark}

In the following, we consider mean-field models of multi-head Attention layers of infinite width, which are parameterized by the distribution of their parameters.
For $\rho \in \Pp(\Theta)$ we define:
\begin{align} \label{eq:attention_meanfield}
    \Phi_\rho[\mu](x) := \int_{\Theta} \varphi_\theta[\mu](x) \, \d \rho(\theta).
\end{align}
We recover the case of multi-head Attention layers of finite width by considering an empirical distribution of parameters.
Given a finite number $H \geq 1$ of heads and a family of parameters $(\theta_h)_{1 \leq h \leq H} \in \Theta^H$, considering $\Hat{\rho} := \frac{1}{H} \sum_h \delta_{\theta_h}$, we have: 
\[
    \Phi_{\Hat{\rho}}[\mu](x) = \frac{1}{H} \sum_h \varphi_{\theta_h}[\mu](x).
\]

\subsection{Mean-field models of Transformers}
\label{subsec:mean_field_transformers}

In this work, we consider infinitely deep Transformers.
These consist of a continuum of multi-head Attention layers, each parameterized by a measure $\rho(.|s)$ for each depth parameter $s \in [0,1]$.

For a discrete distribution of tokens $\mu(s=0) := \frac{1}{n}\sum_i \delta_{x_i(s=0)}$ at depth $s=0$, the set of tokens $x_i(s)$, encoded as a measure $\mu(s) := \frac{1}{n} \sum_{i=1}^{n} \delta_{x_i(s)}$, evolves according to a set of coupled ODEs:
\begin{align} \label{eq:ODE-evol}
     \frac{\mathrm{d} x_i}{\mathrm{d} s}(s) = \Phi_{\rho(.|s)}[\mu(s)](x_i(s)). 
\end{align}
This is extended to an arbitrary (possibly infinite) number of tokens by considering the PDE over the space of tokens:
\begin{equation}\label{eq:PDE-evol}
    \partial_s \mu(s) + \text{div}( \mu(s) \Phi_{\rho(.|s)}[\mu(s)]  ) = 0.    
\end{equation}
Note that in the following, these continuity equations are to be understood in the weak or \emph{distributional} sense, that is, for every test function $\varphi \in \Cc^\infty_c([0, 1] \times \RR^d)$ it holds:
    \begin{align*}
        \int_\Xx \varphi(1) \d \mu(1) =  \int_{\RR^d} \varphi(0) \d \mu(0) + \int_0^1 \int_\Xx \left( \partial_s \varphi(s) + \nabla \varphi(s) \cdot \Phi_{\rho(.|s)}[\mu(s)] \right) \d \mu(s) \d s.
    \end{align*}
In particular, when restricted to discrete measures, one recovers the system of ODEs in~\cref{eq:ODE-evol}.
Once this PDE has been solved and $\mu(s)$ has been computed, one can define the flow-map $\Lambda_\rho[\mu(0)](s)$ (which depends only on $\mu(0)$) induced by the Transformer, which is the mapping advecting $\mu(0)$ to $\mu(s)$.
Formally, we denote by $\Lambda_\rho[\mu(s=0)](s, x) := x(s)$ the mapping obtained by integrating the ODE
\begin{align} \label{eq:ODE_evaluation}
    \forall s \in [0,1], \quad \frac{\mathrm{d} x}{\mathrm{d} s}(s) = \Phi_{\rho(.|s)}[\mu(s)](x(s)),
    \quad \text{with} \quad x(0)=x. 
\end{align}
This implies that $\mu(s) = \Lambda_\rho[\mu(0)](s)_\# \mu(0)$. 
Note that this mapping $\Lambda_\rho[\mu]$ can be defined on the whole space, although for evaluation of the training risk, it is only required to compute it on the support of the input token distribution $\mu(0)$ as well as on some input token $x(0)$. 

\paragraph{Deep residual Transformers as flow-maps}

Since the output of a deep Transformer can be seen as the output of a flow-map, in the following, we view Attention layers as mappings acting on the space of continuous functions.
Given parameters $\theta \in \Theta$, a compact subset $S \subset \RR^d$ and a compactly supported probability measure $\mu$ with $\Supp(\mu) \subset S$, we define for every continuous map $\Lambda \in \Cc^0(S, \RR^d)$:
\begin{align} \label{eq:A_theta}
    \Aa_\theta [\mu](\Lambda) : x \in \RR^d \mapsto \varphi_\theta[\Lambda_\# \mu](\Lambda(x)) \, .
\end{align}
In this way, $\Aa_\theta(\Lambda)$ is the mapping of the Attention layer after a push-forward of both token distribution $\mu$ and input token $x$ by $\Lambda$.
Similarly, for a parameter distribution $\rho \in \Pp(\Theta)$, one can define:
\begin{align} \label{eq:A_rho}
    \Aa_\rho [\mu](\Lambda) \eqdef \int_\Theta \Aa_\theta [\mu](\Lambda) \d \rho(\theta) \, ,
\end{align}
which is the mapping of a multi-head Attention layer after a push-forward of both token distribution $\mu$ and input token $x$ by $\Lambda$.
In particular, \cref{eq:A_rho} is to be understood as a Bochner integral while the definitions of $\Aa_\theta$ and $\Aa_\rho$ are justified in the proof of~\cref{lem:differential_A}.
In the following, since the token distribution $\mu$ will be explicit from the context, we use the shorthand notations $\Aa_\theta \eqdef \Aa_\theta[\mu]$ and $\Aa_\rho \eqdef \Aa_\rho[\mu]$.

Then, instead of seeing Transformers as controlling the token distribution $\mu$ through the continuity equation~\cref{eq:PDE-evol}, one can see Transformers as controlling the flow-map $\Lambda = \Lambda_\rho[\mu_0] \in \Cc^0([0,1] \times S, \RR^d)$ through the (Banach-space-valued) ODE:
\begin{align} \label{eq:flow_map_ODE}
    \frac{\d}{ \d s} \Lambda(s) = \Aa_{\rho(.|s)}(\Lambda(s)), \quad \text{with} \quad \Lambda(0) = \Id.
\end{align}
Note that, due to the absence of regularity assumption on the parameterization $\{ \rho(.|s) \}_{s \in [0,1]}$, this equation is to be understood in the sense of Carathéodory solutions:
\begin{align*}
    \forall s \in [0,1], \quad \Lambda(s) = \Id + \int_0^s \Aa_{\rho(.|r)} (\Lambda(r)) \d r .
\end{align*}

\begin{remark}
   The idea of studying evolution PDEs of the form~\cref{eq:PDE-evol} on the space of densities through the lens of ODEs on the space of flows is a standard idea which can be traced back to the work of Arnold on Euler's equation~\cite{arnold1966geometrie}.
   There are at least two advantages of this approach:
   \begin{itemize}
       \item From a conceptual perspective, it provides a convenient mathematical framework for studying the training of deep Transformers.
       This framework indeed extends the one developed for the study of deep ResNets or NODE models~\cite{barboni2024understanding,lu2020mean,ding2021global,isobe2023convergence}.
        However, whereas NODEs usually process data points through independent ODEs, the particularity of Transformers is to process tokens through coupled ODEs (\cref{eq:ODE-evol}). 
        This results in the study of non-local PDEs on the space of token distributions (\cref{eq:PDE-evol}) whose perturbations are hard to study due to the lack of a proper differentiable structure.
        In contrast, we study here deep Transformers seen as parameterized ODEs on the Banach space $\Cc^0(S, \RR^d)$ (\cref{eq:flow_map_ODE}).
        As a consequence, and as opposed to other works~\cite{song2024unraveling,gao2024global,qin2025convergence}, our results do not depend on the number of tokens used for training.

        \item From a more practical perspective, such a reparameterization of Attention layers provides them with enhanced regularity.
        It is for example known that the Lipschitz constant of Attention layers has an exponential dependence on the radius of the support of the input token distribution~\cite{castin2024smooth,geshkovski2023emergence}.
        In contrast, the Lipschitz constant of $\Aa$ behaves here quadratically w.r.t. $\Lambda \in \Cc^0(S, \RR^d)$ (see~\cref{lem:differential_A}).
   \end{itemize}
\end{remark}

\paragraph{Well-posedness of the forward process} \label{subsec:forward}

We study here the well-posedness of the mean-field models of Transformers, i.e., existence and uniqueness of solutions to~\cref{eq:PDE-evol} and~\cref{eq:flow_map_ODE}.
For the latter Banach-space-valued ODE to be well-posed, a minimal working assumption is that the associated velocity field satisfies adequate measurability and integrability conditions (see~\cref{ass:caratheodory} and more generally~\cref{sec:Banach_ODE} for well-posedness results of Banach-space-valued ODEs).
In the case of the parameterized map $\Aa$ defined in~\cref{eq:A_rho}, it is necessary to assume the family of parameter distributions $\{ \rho(.|s) \}_{s\in [0,1]}$ has appropriate measurability and integrability properties.
We define here:
\begin{align} \label{eq:PLeb}
    \PLeb \eqdef \left\{ \rho \in \Pp_2( [0,1] \times \Theta) \, \colon \, \text{$\rho$ has uniform marginal on $[0,1]$} \right\} .
\end{align}
Note that every parameterization $\rho \in \PLeb$ corresponds to a ($\d s$-almost everywhere) uniquely determined family of parameter distributions $\{ \rho(.|s) \}_{s\in [0,1]} \in \Pp_2(\Theta)^{[0,1]}$, which is obtained by taking the disintegration of $\rho$ w.r.t.\@ the Lebesgue measure on $[0,1]$~\cite[Thm.~4.2.4]{attouch2014variational}.

The following result shows that solutions to~\cref{eq:PDE-evol} are push-forwards of the initial token distribution by flow-maps that are solutions to the ODE~\cref{eq:flow_map_ODE}.
In the following, we denote by $\Cc^0_\co([0,1], \Pp_c(\RR^d))$ the set of co-compactly supported paths of probability measures, i.e., the set of continuous paths \mbox{$\mu : [0,1] \to \Pp_c(\RR^d)$} such that there exists a compact set $S \subset \RR^d$ with $\Supp(\mu(s)) \subset S$ for every $s \in [0,1]$.

\begin{proposition}[Existence and uniqueness of the forward process] \label{prop:existence_uniqueness}
    Let $\rho \in \PLeb$ and consider some initial token distribution $\mu_0 \in \Pp_c(\RR^d)$ with $\Supp(\mu_0) \subset S$ for some compact subset $S \subset \RR^d$.
    Then there exists a unique solution $\mu \in \Cc^0_\co([0,1], \Pp_c(\RR^d))$ starting from $\mu_0$ of:
    \begin{align} \label{eq:transformer_continuity}
        \partial_s \mu(s) + \div \left(  \mu(s) \Phi_{\rho(.|s)}[\mu(s)] \right) = 0
        \quad \text{with} \quad \mu(0) = \mu_0.
    \end{align}
    For every $s \in [0, 1]$, this solution admits the representation $\mu(s) = \Lambda_\rho[\mu_0](s)_\# \mu_0$, where $\Lambda_\rho[\mu_0] \in \Cc^0([0, 1] \times S, \RR^d)$ is the flow-map solution to the Cauchy problem:
    \begin{align} \label{eq:transformer_flow}
        \forall s \in [0,1], \quad \frac{\d}{\d s} \Lambda_\rho[\mu_0](s)
        =  \Aa_{\rho(.|s)}[\mu_0](\Lambda_\rho[\mu_0](s)),
        \quad \text{with} \quad \Lambda_\rho[\mu_0](0) = \Id.
    \end{align}
    In the following, when no ambiguity, we use the shorthand notation $\Lambda_\rho \eqdef \Lambda_\rho[\mu_0]$.
\end{proposition}

\begin{proof}
    The result is proved by applying the theory developed in~\cref{sec:nonlocal_transport}.
    Let $\Aa$ be defined by~\cref{eq:A_rho} and let $\left\{ \rho(.|s) \right\}_{s \in [0,1]}$ be the disintegration of $\rho$ w.r.t.\@ the Lebesgue measure on $[0,1]$.
    Then it follows from~\cref{lem:differential_A} that the map
    \begin{align}
        (s, \Lambda) \in [0,1] \times \Cc^0(S, \RR^d) \ \mapsto \ \Aa_{\rho(.|s)}(\Lambda) 
    \end{align}
    satisfies the Carathéodory condition, is locally $L^1$-Lipschitz and has $L^1$-linear growth (\cref{ass:caratheodory}).
    Therefore, by~\cref{thm:caratheodory} there exists a unique solution to the Cauchy problem:
    \begin{align*}
        \forall s \in [0,1], \quad \Lambda(s) = \Id + \int_0^s \Aa_{\rho(.|r)}(\Lambda(r)) \d r \, .
    \end{align*}
    Moreover, using the linear growth of $\Aa$ w.r.t.\@ $\Lambda$, \cref{lem:gronwall} gives the bound:
    \begin{align*}
        \forall s \in [0,1], \quad \left\| \Lambda(s) \right\|_{\Cc^0} \leq \left\| \Id \right\|_{\Cc^0} e^{\int_{[0,1] \times \Theta} \| V \| \d \rho(\theta, r)} \, .
    \end{align*}
    As in~\cref{thm:wellposed}, this shows the existence and uniqueness of a solution $\mu \in \Cc^0_\co([0,1], \Pp_c(\RR^d))$ to~\cref{eq:transformer_continuity} and the flow map $\Lambda_\rho[\mu_0]$ is the solution $\Lambda$ to the above Cauchy problem.
\end{proof}

\subsection{Training mean-field Transformers with Conditional Wasserstein gradient flows} \label{subsec:gradient}

We are interested here in the dynamics of gradient-based algorithms for the training of deep Transformer architectures. In the following, we consider a given (finite) dataset $\{ (\mu^{(j)}, x^j) \}_{1 \leq j \leq N} \in ( \Pp_c(\RR^d) \times \RR^d )^N$ constituted by $N$ pairs of token distributions $\mu^{(j)}$ and input tokens $x^j$.
For each index $j \in \{1, ..., N \}$ we will consider a loss $\ell^j : \RR^d \to \RR^d$ on the space of tokens. Then, for a parameterization $\rho \in \PLeb$, we define the training risk as:
\begin{align} \label{eq:risk}
    \Ll(\rho) \eqdef \frac{1}{N} \sum_{j = 1}^N  \ell^j(x^j(1)) ,
\end{align}
where $(x^j(s))_{s \in [0,1]}$ and $(\mu^{(j)}(s))_{s \in [0,1]}$ are, respectively, the solutions of the forward process~\cref{eq:ODE_evaluation} and~\cref{eq:PDE-evol}, with respective initializations, $x^j(0) = x^j$ and $\mu^{(j)}(0) = \mu^{(j)}$.

Equivalently, the training risk is given for $\rho \in \PLeb$ by:
\begin{align} \label{eq:risk_flowmap}
    \Ll(\rho) = \frac{1}{N} \sum_{j = 1}^N  \ell^j(\Lambda^j_\rho(1,x^j)) ,
\end{align}
where, in the above equation and in the rest of this paper, we use the shorthand $\Lambda^j_\rho = \Lambda_\rho[\mu^{(j)}]$ to denote the flow-maps given by~\cref{prop:existence_uniqueness} and solution to the ODE:
\begin{align*}
    \frac{\d}{\d s} \Lambda^j_\rho(s) = \Aa^j_{\rho(.|s)}(\Lambda^j_\rho(s)) ,
\end{align*}
with $\Aa^j_{\rho(.|s)} = \Aa_{\rho(.|s)}[\mu^{(j)}]$ defined by~\cref{eq:A_rho}.
We will also consider, for each index $j \in \{1, ..., N\}$, a compact subset $S^j \subset \RR^d$ such that $\Supp(\mu^{(j)}) \cup \{x^j\} \subset S^j$.
In particular, $\Lambda^j_\rho \in \Cc^0([0,1] \times S^j, \RR^d)$.

\begin{remark}
    For concreteness, one can consider in the following that the losses $\ell^j$ are of the form $\ell^j : x \mapsto \frac{1}{2} \| x - y^j \|^2$ for some target tokens $y^j \in \RR^d$.
    In this case, the training risk associated with some parameterization $\rho \in \PLeb$ reads:
    \begin{align*}
        \Ll(\rho) = \frac{1}{2N} \sum_{j = 1}^N  \left\| x^j(1) - y^j \right\|^2
        = \frac{1}{2N} \sum_{j = 1}^N  \left\| \Lambda^j_\rho(1, x^j) - y^j \right\|^2 .
    \end{align*}
\end{remark}

\subsubsection{Conditional Optimal Transport geometry}

To study the dynamics of gradient-based methods for the minimization of the risk $\Ll$ in the training of Transformers, we follow the approach of~\cite{barboni2024understanding}.
Using tools from the theory of gradient flows in metric spaces~\cite{ambrosio2008gradient}, it was shown in~\cite{barboni2024understanding} that gradient flows for the training of mean-field models of residual architectures can be modeled by metric gradient flows on the space of parameterizations when provided with a \emph{Conditional Optimal Transport (COT)} metric.
Such a COT metric on $\PLeb$ corresponds to a layer-wise $\Ww_2$-distance, which enforces the marginal constraint defining $\PLeb$ to be conserved throughout training.
For $\rho, \rho' \in \PLeb$ we define:
\begin{align}
    \Ww_2^\COT(\rho, \rho') \eqdef \left( \int_0^1 \Ww_2(\rho(.|s), \rho'(.|s))^2 \d s \right)^{1/2}
\end{align}
Then $\Ww_2^\COT$ provides $\PLeb$ with a complete metric space structure (see e.g.~\cite[Sec.~2]{barboni2024understanding}).

\begin{proposition}
    The distance $\Ww_2^\COT$ defines a complete metric space structure on $\PLeb$.
\end{proposition}

Of particular interest here is that on $\PLeb$ provided with the metric $\Ww_2^\COT$, absolutely continuous curves enjoy dynamical properties that are similar to those of the classical Wasserstein space.
That is, a celebrated result from~\cite{ambrosio2008gradient} is that absolutely continuous curves in $\Ww_2(\Theta)$ are characterized by solutions of continuity equations. Here, for an interval $I \subset \mathbb{R}$, a narrowly continuous curve $(\rho_t)_{t \in I}$ in $\PLeb$ is absolutely continuous w.r.t. $\Ww_2^\COT$ if and only if it is a solution to a continuity equation of the form:
\begin{align} \label{eq:continuity_equation}
    \partial_t \rho_t + \div_\theta (\rho_t v_t) = 0 \quad \text{on $I \times [0,1] \times \Theta$},
\end{align}
where $v : I \times [0,1] \times \Theta \to \Theta$ is some Borel velocity field s.t. $\| v_t \|_{L^2(\rho_t)} \in L^1(I)$ (see e.g.~\cite[Sec.~2]{barboni2024understanding} but also~\cite{peszek2023heterogeneous}). This property informally provides $\PLeb$ with a differential structure that will motivate our definition of gradient flow curves in~\cref{def:gradient_flow}. Precisely, gradient flow curves for the risk $\Ll$ are solutions to the above continuity equation for a velocity field $v_t$ that is the solution to some variational problem (\cref{thm:maximal_slope_equivalence}).

\subsubsection{First-order perturbation of the risk and adjoint equations}

We study here first-order perturbations of the training risk $\Ll$ defined in~\cref{eq:risk} along perturbations of the parameterization of the form~\cref{eq:continuity_equation}.
Such a perturbation analysis is permitted by the regularity of the Attention map, which we detail in~\cref{subsec:regularity_attention_inputs,subsec:regularity_attention_parameters}.
We start by stating a result on the first-order perturbations of the mean-field Transformers flow-maps, which we prove in~\cref{subsec:proof_flow_differentiability}.

\begin{proposition} \label{prop:flow_differentiability}
Let $(\rho_t)_{t\in (0,1)}$ be some absolutely continuous curve in $\PLeb$ with velocity field $(v_t)_{t \in (0,1)}$.
    For some initial token distribution $\mu(0) \in \Pp_c(\RR^d)$ with $\Supp(\mu(0)) \subset S$, denote by $\Lambda_t = \Lambda_{\rho_t}[\mu(0)] \in \Cc^0([0, 1] \times S, \RR^d)$ the flow-map obtained in~\cref{prop:existence_uniqueness}.
    Then the map $t \mapsto \Lambda_t$ is absolutely continuous in $\Cc^0([0,1] \times S, \RR^d)$ and, for a.e.\@ $t \in (0,1)$, its (Fréchet) differential $\delta \Lambda_t \eqdef \frac{\d}{\d t} \Lambda_t$ is the solution of the Cauchy problem:
    \begin{align*}
        \delta \Lambda_t (s, x) = & \int_0^s \D_\Lambda \Aa_{\rho(.|r)}(\Lambda_t(r)) \cdot \delta \Lambda_t (r)  \d r + \int_0^s \int_\Theta  \D_\theta \Aa_\theta(\Lambda_t(r)) \cdot v_t(r) \d \rho(\theta,r) \d r ,
    \end{align*}
    where $\D_\Lambda \Aa_{\rho(.|r)}(\Lambda)$ and $\D_\theta \Aa_\theta(\Lambda)$ are defined in~\cref{lem:differential_A} and~\cref{lem:A_theta_param_differential} respectively.
\end{proposition}

To define perturbations of the risk, we will not, in fact, directly rely on the above linear ODEs but rather on \emph{adjoint equations}, which can be seen as their dual, defined backward in time.
The associated \emph{adjoint variables} live in the dual of the space of solutions to the forward equation~\cref{eq:flow_map_ODE}, i.e., the space of finite, $d$-dimensional vector measures $\Mm^d(S) = \Cc^0(S, \RR^d)^*$.
These can be seen as limits of the quantities computed by the backpropagation algorithm when the depth of the Transformers goes to infinity. Numerical resolution of the adjoint equations is at the core of the adjoint method developed for the training of NODE models~\cite{chen2018neural}.

\begin{definition}[Adjoint equations]
    Let $\rho \in \PLeb$ be some parameter distribution and, for index $j \in \{1, ..., N\}$, let $\ell^j$ be the loss function in~\cref{eq:risk}.
    The \emph{adjoint variable} associated to input token distribution $\mu^{(j)} \in \Pp_c(\RR^d)$ and input token $x^j \in \RR^d$ is the measure-valued function $\mm^j_\rho \in \Cc^0([0,1], \Mm^d(S^j))$ solution of the (linear) Cauchy problem:
    \begin{align} \label{eq:backward_ODE}
        \forall s \in [0,1], \quad \mm^{(j)}_\rho(s) =
        \nabla \ell^j (\Lambda^j_\rho (1, x^j)) \, \delta_{x^j} 
        + \int_s^1 \D_\Lambda \Aa^j_{\rho(.|r)}(\Lambda^j_\rho(r))^* \cdot \mm^{(j)}_\rho (r) \d r \, ,
    \end{align}
    where $\D_\Lambda \Aa^j_{\rho(.|r)}(\Lambda)$ is defined in~\cref{lem:differential_A}.
\end{definition}

Note that the above adjoint equations are well-posed linear ODEs on the space of finite vector measures. The following result expresses the first-order perturbation of the training risk as a function of the solutions $\Lambda^j_\rho$ (resp. $\mm^{(j)}_\rho$) to the forward (resp. backward) ODEs.

\begin{theorem} \label{thm:upper_gradient}
    Let $(\rho_t)_{t\in (0,1)}$ be some absolutely continuous curve in $\PLeb$ with velocity field $(v_t)_{t \in (0,1)}$.
    Then the map $t \mapsto \Ll(\rho_t)$ is absolutely continuous and for a.e. $t \in (0,1)$ we have:
    \begin{align} \label{eq:derivative_risk}
        \frac{\d}{\d t} \Ll(\rho_t) = \frac{1}{N} \sum_{j=1}^N \int_0^1 \int_\Theta \left< \mm^{(j)}_{\rho_t}(s) , \D_\theta \Aa^j_\theta (\Lambda_{\rho_t}^j(s)) \cdot v_t(s) \right> \d \rho_t(\theta, s) \, .
    \end{align}
    Moreover, for $\rho \in \PLeb$, defining
    \begin{align} \label{eq:upper_gradient}
        \left| \nabla \Ll \right| (\rho) \eqdef  \left( \int_0^1 \int_\Theta \left\| \frac{1}{N} \sum_{j=1}^N  \D_\theta \Aa^j_\theta (\Lambda^j_\rho(s))^* \mm^{(j)}_\rho(s) \right\|^2 \d \rho(\theta , s)  \right)^{1/2} \, ,
    \end{align}
    then $\left| \nabla \Ll \right|$ is a (strong) upper-gradient for $\Ll$ in the sense of~\cite[Def.~1.2.1]{ambrosio2008gradient}.
\end{theorem}

\begin{proof}
    We use the shorthand $\Lambda^j_t \eqdef \Lambda^j_{\rho_t}$, $\mm^{(j)}_t \eqdef \mm^{(j)}_{\rho_t}$ and $\delta \Lambda^j_t \eqdef \frac{\d}{\d t} \Lambda^j_t$ for $t \in (0,1)$.
    By~\cref{prop:flow_differentiability} (and assuming sufficient regularity on the functionals $\ell^j$), we have that the map $t \mapsto \Ll(\rho_t)$ is absolutely continuous and for a.e. $t \in (0,1)$ it holds:
    \begin{align*}
        \frac{\d}{\d t} \Ll(\rho_t) = \frac{1}{N} \sum_{j=1}^N \left< \nabla \ell^j(\Lambda^j_t(1, x^j)), \delta \Lambda^j_t(1, x^j) \right>
        = \frac{1}{N} \sum_{j=1}^N \left< \mm^{(j)}_t (1),  \delta \Lambda^j_t(1) \right>
    \end{align*}
    Then using the definition of $\mm^{(j)}_t$ (\cref{eq:backward_ODE}) and $\delta \Lambda^j_t$ (\cref{eq:flow_map_ODE}) we have for every index $j \in \{1, ..., N \}$:
    \begin{align*}
        \left< \mm^{(j)}_t (1),  \delta \Lambda^j_t(1) \right>
        ={}& \int_0^1 \left< \mm^{(j)}_t(s) , \D_\Lambda \Aa^j_{\rho_t(.|s)}(\Lambda^j_t(s)) \cdot \delta \Lambda^j_t(s) \right> \d s \\
        & + \int_0^1 \left< \mm^{(j)}_t(s),
        \int_\Theta \D_\theta \Aa^j_\theta(\Lambda^j_t(s)) \cdot v_t(s,\theta) \d \rho_t(\theta|s) \right> \d s \\
        & - \int_0^1 \left< \D_\Lambda \Aa^j_{\rho_t(.|s)}(\Lambda^j_t(s))^* \cdot \mm^{(j)}_t(s) , \delta \Lambda^j_t(s) \right> \d s \\
        ={}& \int_0^1 \left< \mm^{(j)}_t(s),
        \int_\Theta \D_\theta \Aa^j_\theta(\Lambda^j_t(s)) \cdot v_t(s,\theta) \d \rho_t(\theta|s) \right> \d s \, .
    \end{align*}
    But then by properties of the Bochner integral (\cite[Thm.~6]{diestel1977measures}), for a.e. $s \in [0,1]$ it holds:
    \begin{align*}
        \left< \mm^{(j)}_t(s),
        \int_\Theta \D_\theta \Aa^j_\theta(\Lambda^j_t(s)) \cdot v_t(s,\theta) \d \rho_t(\theta|s) \right>
        = \int_\Theta \left< \mm^{(j)}_t(s), \D_\theta \Aa^j_\theta(\Lambda^j_t(s)) \cdot v_t(s,\theta) \right> \d \rho_t(\theta|s) \, .
    \end{align*}
    Putting these last two equations together gives the desired result for $\frac{\d}{\d t} \Ll(\rho_t)$.
    The fact that $\left| \nabla \Ll  \right|$ is an upper-gradient then follows by applying the Cauchy-Schwarz inequality.
\end{proof}

\begin{example}[Finite number of tokens]
    Let us study in more detail the case of a finite number of tokens, that is, when for every index $j \in \{1, ..., N\}$ we have an empirical token distribution $\mu^{(j)} = \frac{1}{n^j} \sum_{i=1}^{n^j} \delta_{x^j_i}$ for some $n^j \geq 1$.
    We denote by $x^j_0 \eqdef x^j$ the input token.
    We can thus study the action of the Transformer on the compact subsets $S^j = \left\{ x^j_0, x^j_1, ..., x^j_{n^j} \right\}$ and, for simplicity, assume all tokens are distinct, i.e., $\# S^j = n^j+1$.
    
    Let us fix a parameterization $\rho \in \PLeb$.
    With the above notations, the forward process through the Transformer reads:
    \begin{align*}
        \frac{\d}{\d s} x^j_i (s)
        = \Phi_{\rho(.|s)}[\mu^{(j)}(s)](x^j_i(s)) 
        = \int_\Theta
        \frac{\sum_{l=1}^{n^j} e^{\left< Q x^j_i(s) + q, x^j_l(s) \right>} x^j_l(s)}
        {\sum_{l=1}^{n^j} e^{\left< Q x^j_i(s) + q, x^j_l(s) \right>}}
        \d \rho(\theta|s),
        \quad
        \substack{1 \leq j \leq N,\\ 0 \leq i \leq n^j.}
    \end{align*}
    Here $\mu^{(j)}(s) = \frac{1}{n^j} \sum_{i=1}^{n^j} \delta_{x^j_i(s)}$.
    By construction, the flow-map $\Lambda^j_\rho$ defined in~\cref{prop:existence_uniqueness} satisfies $\Lambda^j_\rho(s, x^j_i) = x^j_i(s)$.
    In this setting, the adjoint variable $\mm^{(j)}_\rho$ can be identified for every $s \in [0,1]$ with a family of vectors $( \mm^{(j)}_i (s) )_{0 \leq i \leq n^j} \in \RR^{d \times (n^j+1)}$.

    Also, it is interesting to look at the expression for the gradient $| \nabla \Ll |$ in the case of a finite number of tokens.
    We will be interested in~\cref{subsec:convergence} in the lower-bound $| \nabla_V \Ll |$ (\cref{eq:upper_gradient_lower_bound}).
    It reads here:
    Denoting the attention weights by
    \[
        p^j_{i,l}(s,\theta) \eqdef
        \frac{e^{\left< Q x^j_i(s)+q, x^j_l(s) \right>}}
        {\sum_{r=1}^{n^j} e^{\left< Q x^j_i(s)+q, x^j_r(s) \right>}},
    \]
    we obtain
    \begin{align*}
        | \nabla_V \Ll |^2(\rho)
        &= \int_{[0,1] \times \Theta}
        \Big\| \frac{1}{N} \sum_{j=1}^N
        \D_V \Aa^j_\theta (\Lambda^j(s))^* \mm^{(j)}(s) \Big\|^2
        \d \rho(\theta , s) \\
        &= \int_{[0,1] \times \Theta} \Big\|
        \frac{1}{N} \sum_{j=1}^N \sum_{\substack{0 \leq i \leq n^j \\ 1 \leq l \leq n^j}}
        p^j_{i,l}(s,\theta)\,\mm^{(j)}_i(s) x^j_l(s)^\top
        \Big\|^2 \d \rho(\theta , s).
    \end{align*}
    It thus appears that the above expression is a quadratic function of the adjoint variables $(\mm^{(j)})_{1 \leq j \leq N}$.
    In~\cref{subsec:convergence} we will analyze the conditioning of the associated quadratic form to deduce convergence guarantees for the training of Transformers with gradient flow.
\end{example}

\subsubsection{Gradient flow equation and well-posedness}
\label{subsec:gradient_flow}

We now define a notion of gradient flow for the minimization of the training risk $\Ll$.
We show in~\cref{thm:maximal_slope_equivalence} that it corresponds to a metric gradient flow w.r.t. the COT metric $\Ww_2^\COT$ on $\PLeb$ and state the existence and uniqueness of gradient flow curves in~\cref{thm:gradient_flow_wellposed}.

\begin{definition} \label{def:gradient_flow}
    Let $\rho_0 \in \PLeb$ be some initial parameter distribution.
    We say a locally absolutely continuous curve $(\rho_t)_{t \in [0, +\infty)}$ is a gradient flow for the risk $\Ll$ starting from $\rho_0$ if $\lim_{t \to 0} \rho_t = \rho_0$ and it is a solution to the continuity equation:
    \begin{align} \label{eq:gradient_flow_continuity}
        \partial_t \rho_t - \div_\theta ( \rho_t \nabla \Ll [\rho_t]) = 0 \quad \text{on $(0, +\infty) \times [0,1] \times \Theta$} ,
    \end{align}
    where, for $\rho \in \PLeb$, the velocity field $\nabla \Ll [\rho]$ is defined by:
    \begin{align} \label{eq:gradient_field}
        \forall (s, \theta) \in [0,1] \times \Theta, \quad
        \nabla \Ll [\rho](s, \theta) = 
        \frac{1}{N} \sum_{j=1}^N  \D_\theta \Aa^j_\theta (\Lambda_\rho^j(s))^* \mm^{(j)}_\rho(s) \ \in \Theta .
    \end{align}
\end{definition}

    To study the existence and uniqueness of solutions to the gradient flow~\cref{eq:gradient_flow_continuity}, we rely on the theory of gradient flows in metric spaces, applied here to the case of the metric space $\PLeb$ provided with the Conditional Optimal Transport metric $\Ww_2^\COT$.
    Although there is no smooth differentiable structure on this space, the gradient flow curve can instead be characterized as a solution to a variational problem
(see e.g.~\cite{ambrosio2008gradient} but also~\cite{ambrosio2012user} or~\cite{santambrogio2017euclidean}).
The following result (proved in~\cref{subsec:app_maximal_slope_equivalence}) states that such a notion of \emph{curve of maximal slope} is equivalent to the notion of gradient flow defined in~\cref{def:gradient_flow}.

\begin{theorem} \label{thm:maximal_slope_equivalence}
    Let $\rho_0 \in \PLeb$ be some parameter distribution.
    Then a locally absolutely continuous curve $(\rho_t)_{t \in [0, +\infty)}$ is a gradient flow for the risk $\Ll$ starting from $\rho_0$ in the sense of~\cref{def:gradient_flow} if and only if it is a \emph{curve of maximal slope} for $\Ll$, i.e., it satisfies:
    \begin{enumerate}
        \item the map $t \in [0, +\infty) \mapsto \Ll(\rho_t)$ is non-increasing,
        \item for a.e. $t \in [0, +\infty)$ the following \emph{Evolution Dissipation Inequality} holds:
        \begin{align} \label{eq:EDI} \tag{EDI}
            \frac{\d}{\d t} \Ll(\rho_t) \leq - \frac{1}{2} \left( \left| \frac{\d}{\d t} \rho_t \right|^2 + \left| \nabla \Ll \right|^2 (\rho_t) \right) ,
        \end{align}
        with $\left| \frac{\d}{\d t} \rho_t \right|$ the \emph{metric derivative} of the curve $(\rho_t)_{t \geq 0}$ and $\left| \nabla \Ll \right|$ the upper-gradient defined in~\cref{eq:upper_gradient}.
    \end{enumerate}
\end{theorem}

We prove in~\cref{subsec:app_existence_uniqueness} the following result, stating the existence and uniqueness of gradient flow curves for the training of mean-field models of Transformers.

\begin{theorem} \label{thm:gradient_flow_wellposed}
    Let $\rho_0 \in \PLeb$ be some initial parameter distribution such that $\Supp(\rho_0) \subset [0,1] \times \left\{ (V, Q, q) \in \Theta, \, \|V \| \leq R \right\}$ for some $R \geq 0$.
    Then there exists a unique gradient flow $(\rho_t)_{t \in [0, +\infty)}$ of the risk $\Ll$ starting from $\rho_0$ in the sense of~\cref{def:gradient_flow}.
\end{theorem}

\subsection{Convergence of gradient flow curves} \label{subsec:convergence}

We turn to the convergence analysis of the previously defined gradient flow dynamics for the training of mean-field Transformers.
We explain how the geometry of the risk landscape is related to the conditioning of the \emph{Neural Tangent Kernels (NTK)}  of Attention layers.
In particular, we present a local convergence result in~\cref{thm:local_convergence}.
In~\cref{sec:NTK_expressivity}, we will study in more detail the conditioning of such NTKs and give examples of parameter initializations for which our convergence conditions are satisfied.

Concretely, we will show here that the training risk $\Ll$ satisfies a \emph{Polyak-\L{}ojasiewicz inequality} around well-chosen initializations of the parameters.

\begin{definition}[Polyak-\L{}ojasiewicz inequality]
\label{def:PL}
Let $R, m \geq 0$ be constants.
The training risk $\Ll$ satisfies the $(R,m)$-P-\L{} inequality around some parameterization $\rho_0 \in \PLeb$ when it holds:
\begin{align*}
    \forall \rho \in B(\rho_0, R), \quad \left| \nabla \Ll \right|^2(\rho) \geq m \Ll(\rho),
\end{align*}
with $\left| \nabla \Ll \right|(\rho)$ the upper-gradient defined in~\cref{eq:upper_gradient}.
\end{definition}

The P-\L{} inequality is a property that has been shown to often hold for the risk associated with the training of overparameterized neural network architectures~\cite{liu2020linearity}.
Classical mathematical results can then be applied and give quantitative guarantees for the convergence of gradient flow~\cite{hauer2019kurdyka,schiavo2023local,chatterjee2022convergence}.

\subsubsection{Neural tangent kernel and convergence under positivity assumption}

As it appears in~\cref{eq:upper_gradient}, the square gradient of the risk is a quadratic function of the adjoint variables $\mm^{(j)}$.
We explain here how good conditioning of the associated quadratic form, through the \emph{Neural Tangent Kernel (NTK)}, implies convergence guarantees for gradient flow.

\begin{definition}[Neural Tangent Kernel] \label{def:NTK}
    Let $\rho \in \PLeb$ be some parameter distribution.
    We define the Neural Tangent Kernel (NTK) of $\rho$ at time $s \in [0,1]$ as the quadratic forms $\KK[\rho](s)$ on $\Mm^d(S^1) \times ... \times \Mm^d(S^N)$ given by:
    \begin{align} \label{eq:NTK}
        \KK[\rho](s)(\mm,\mm) \eqdef  \int_{\Theta}
        \left\|
        \sum_{1 \leq j \leq N} \D_\theta \Aa_\theta^j(\Lambda^j_\rho(s))^* \mm^{(j)}
        \right\|^2
        \d \rho(\theta |s)
    \end{align}
    for all $\mm = (\mm^1, ..., \mm^N) \in \Mm^d(S^1) \times ... \times \Mm^d(S^N)$.

    Moreover, we denote by $\lambda_{\min}$ its smallest eigenvalue w.r.t. the mixed $L^2$-$\TV$-norm, i.e.,
    \begin{align} \label{eq:lambda_min}
        \lambda_{\min}(\KK[\rho](s))
        \eqdef
        \min_{\mm} \KK[\rho](s)(\mm,\mm) ,
    \end{align}
    where the minimum is taken w.r.t. $\mm \in \Mm^d(S^1) \times ... \times \Mm^d(S^N)$ such that
    \begin{align*}
        \| \mm \|^2_{L^2-\TV} \eqdef \sum_{1 \leq j \leq N} \| \mm^{(j)} \|^2 \leq 1 .
    \end{align*}
    
\end{definition}

The following result shows how a good conditioning of the NTK of Attention layers implies the local P-\L{} inequality of~\cref{def:PL} for the training of Transformers.
As a consequence, ~\cref{cor:convergence_positive_NTK} states that gradient flow converges at a linear rate when the NTKs of the Attention layers stay well-conditioned.
While hard to ensure a priori before training, such an assumption could be checked numerically during training.

\begin{proposition} \label{prop:PL}
    Let $\Ee \geq 0$.
    Then there exists a constant $C = C(\Ee) > 0$ s.t. for any parameter distribution $\rho \in \PLeb$ with $\Ee_2(\rho) \leq \Ee$ it holds
    \begin{align*}
        \left\| \nabla \Ll[\rho] \right\|^2_{L^2(\rho)} \geq C N^{-1} \left( \int_0^1 \lambda_{\min} ( \KK[\rho](s) ) \d s \right) \Ll(\rho) .
    \end{align*}
\end{proposition}

\begin{proof}
    First observe that, thanks to the forward ODE~\cref{eq:flow_map_ODE} and the linear growth of $\Aa$ w.r.t. $\theta$ and $\Lambda$, we have a constant $C_1 = C_1(\Ee_2(\rho))$ such that it holds:
    $$
    \forall j \in \left\{ 1, ..., N \right\}, \, \forall s \in [0,1], \quad
    \left\| \Lambda^j_\rho(s) \right\| \leq C_1 .
    $$
    Similarly, using the quadratic growth of $\D_\Lambda \Aa$ w.r.t. $\theta$ and $\Lambda$, we can provide an estimate on the solutions of the backward ODE~\cref{eq:backward_ODE} of the form:
    $$
    \forall j \in \left\{ 1, ..., N \right\}, \, \forall s \in [0,1], \quad
    \left\| \mm^{(j)}_\rho(s) \right\| \geq C_2^{-1} \left\| \nabla \ell^j (\Lambda^j_\rho(1,x^j)) \right\|  , 
    $$
    for some constant $C_2 = C_2(\Ee_2(\rho))$.
    Then by construction of the gradient field $\nabla \Ll[\rho]$ in~\cref{eq:gradient_field} and of the tangent kernel in~\cref{eq:NTK} we have:
    $$
    \left\| \nabla \Ll[\rho] \right\|^2_{L^2(\rho)}
    = \frac{1}{N^2 }\int_0^1 \KK[\rho](s)(\mm(s), \mm(s)) \d s ,
    $$
    where $\mm(s) = (\mm^1(s), ..., \mm^N(s)) \in \Mm^d(S^1) \times ... \times \Mm^d(S^N)$.
    Thus, using the previous estimate gives:
    \begin{align*}
        \left\| \nabla \Ll[\rho] \right\|^2_{L^2(\rho)}
        & \geq \frac{1}{N^2 } \int_0^1 \lambda_{\min}(\KK[\rho](s))  \| \mm(s) \|^2_{L^2-\TV} \d s \\
        & \geq C_2^{-1} N^{-1} \left( \int_0^1 \lambda_{\min} ( \KK[\rho](s) ) \d s \right) \left( N^{-1} \sum_{j=1}^N \left\| \nabla \ell^j (\Lambda^j_\rho(1,x^j)) \right\|^2 \right) \\
        & \geq C_2^{-1} N^{-1} \left( \int_0^1 \lambda_{\min} ( \KK[\rho](s) ) \d s \right) \Ll(\rho) ,
    \end{align*}
    where in the last inequality we used that the loss $\ell$ satisfies a P-\L{}-inequality.
\end{proof}

\begin{corollary} \label{cor:convergence_positive_NTK}
    Let $\rho_0 \in \PLeb$ be some initial parameter distribution and let $(\rho_t)_{t \in [0, +\infty)}$ be a gradient flow curve of the risk $\Ll$ starting from $\rho_0$.
    Assume that, for every $t \geq 0$, it holds $\Ee_2(\rho) \leq \Ee$ for some $\Ee \geq 0$ and there exists some $\lambda > 0$ s.t.
    \begin{align}
        \int_0^1 \lambda_{\min}(\KK[\rho_t](s)) \d s \geq \lambda > 0 .
    \end{align}
    Then $\rho_t \xrightarrow{t \to +\infty} \rho_\infty \in \PLeb$ and it holds
    \begin{align*}
        \forall t \geq 0, \quad \Ll(\rho_t) \leq e^{ - C N^{-1} \lambda t} \Ll(\rho_0) ,
    \end{align*}
    for some constant $C = C(\Ee) > 0$.
\end{corollary}

\begin{proof}
    By the assumptions in the statement and by applying~\cref{prop:PL}, it holds for any $t \geq 0$ the P-\L{}-inequality:
    \[
    \left\| \nabla \Ll[\rho_t] \right\|^2_{L^2(\rho_t)}
    \geq C N^{-1}
    \left( \int_0^1 \lambda_{\min} ( \KK[\rho_t](s) ) \d s \right)
    \Ll(\rho_t) ,
    \]
    for some constant $C = C(\Ee)$.
    Since $\left\| \nabla \Ll[\rho_t] \right\|_{L^2(\rho_t)}$ is an upper-gradient for $\Ll$, the result follows from classical results on convergence of gradient flow curves in metric spaces under a P-\L{} assumption (see e.g.~\cite{schiavo2023local,hauer2019kurdyka}).
\end{proof}

We conclude this section with a remark linking the positivity of the NTK to the expressivity of Attention layers.

\begin{remark}
\label{rmk:NTK_expressivity}
    Note that we study here the conditioning of the risk landscape by looking at the conditioning of the NTK operator $\KK$.
    In particular, \cref{prop:PL} is meaningful only when the NTK is (strictly) positive, as a consequence of which we show the risk satisfies a Polyak-\L{}ojasiewicz inequality.

    Alternatively, such a positivity assumption on the NTK can be seen, in a dual sense, as an expressivity assumption on Attention layers.
    Indeed, considering a parameter distribution $\rho \in \PLeb$ and $s \in [0,1]$, the assumption
    $$
    \lambda_{\min}(\KK[\rho](s)) > 0
    $$
    holds if and only if the space of functions
    \begin{align*}
        \Bigl\{
        (x^j)_{j=1}^N \in S^1 \times ... \times S^N
        \mapsto
        \left(
        \int_{\Theta} \D_\theta \Aa^j(\Lambda^j_\rho(s))(x^j)
        \cdot v(\theta) \d \rho(\theta|s)
        \right)_{j=1}^N \in (\RR^d)^N 
         \colon \; v \in L^2(\rho(.|s)) \Bigr\}
    \end{align*}
    is dense in $\Cc^0(S^1, \RR^d)  \times ... \times \Cc^0( S^N, \RR^d)$.
    This can be proven by following the same argument as in~\cite[Prop.~1]{micchelli2006universal}.

    However, a fundamental advantage of~\cref{prop:PL} over a study of the expressivity properties of Attention layers is the possibility to compute (at least numerically) the minimal eigenvalue $\lambda_{\min}$.
    This leads us to obtain quantitative convergence rates in~\cref{cor:convergence_positive_NTK}, when assuming a uniform lower-bound on this minimal eigenvalue.
\end{remark}

\subsubsection{Lower-bound on the NTK and local convergence}

Of particular interest in the following will be a lower-bound on the upper-gradient, which is provided by only considering the gradient w.r.t. the linear variable $V$.
Indeed, for any parameter distribution $\rho \in \PLeb$ it follows from~\cref{eq:upper_gradient} that:
\begin{align} \label{eq:upper_gradient_lower_bound}
    \left| \nabla \Ll \right| (\rho)
    \geq
    \left( \int_0^1 \int_\Theta
    \left\| \frac{1}{N} \sum_{j=1}^N
    \D_V \Aa^j_\theta (\Lambda_\rho^j(s))^* \mm^{(j)}_\rho(s) \right\|^2
    \d \rho(\theta , s)  \right)^{1/2}
    \defeq \left| \nabla_V \Ll \right| (\rho) \, .
\end{align}
Recall that $\Aa^j_\theta(\Lambda^j)$ is linear w.r.t. the parameter $V$, namely
$$
    \Aa^j_\theta(\Lambda^j)(x) = \varphi_\theta[\Lambda^j_\# \mu^{(j)}](\Lambda^j(x)) =  V \cdot \Mm^{j}_\theta (\Lambda^j)(x),
$$
where for $\Lambda^j \in \Cc^0(S^j, \RR^d)$ and parameter $\theta = (V, Q, q) \in \Theta$ we defined:
\begin{align} \label{eq:Mm_definition}
    \Mm^j_\theta (\Lambda^j)(x) \eqdef 
    \frac{\int e^{\left< Q \Lambda^j(x)+q, \Lambda^j(y)  \right>} \Lambda^j(y) \d \mu^{(j)}(y) }{ \int e^{\left< Q \Lambda^j(x)+q, \Lambda^j(y)  \right>} \d \mu^{(j)}(y) } .
\end{align}
In particular, with those notations:
$$
\sum_{j=1}^N  \D_V \Aa^j_\theta (\Lambda^j)^* \mm^{(j)} = \sum_{j=1}^N \int_{S^j} \d \mm^{(j)}(x) \otimes \Mm^j_\theta (\Lambda^j)(x) .
$$
We denote by $\KK^1$ the kernel associated with the quadratic form in~\cref{eq:upper_gradient_lower_bound}. By construction, we have $\KK \geq \KK^1$, where $\KK$ was the quadratic form defined in~\cref{def:NTK}.

\begin{definition} \label{def:NTK_V}
    Let $\rho \in \PLeb$ be some parameter distribution.
    For time $s \in [0,1]$ we define the quadratic forms $\KK^1[\rho](s)$ for all $\mm = (\mm^1, ..., \mm^N) \in \Mm^d(S^1) \times ... \times \Mm^d(S^N)$ by:
    \begin{align} \label{eq:NTK_V}
        \KK^1[\rho](s)(\mm,\mm) \eqdef  \int_{\Theta}
        \left\|
        \sum_{1 \leq j \leq N} \D_V \Aa_\theta^j(\Lambda^j_\rho(s))^* \mm^{(j)}
        \right\|^2
        \d \rho(\theta |s) .
    \end{align}
    Equivalently, we have:
    \begin{align*}
        \KK^1[\rho](s)(\mm,\mm) = \sum_{1 \leq i,j \leq N}
        \left< \mm^{(i)}, K^1_{i,j}[\rho](s) \star \mm^{(j)} \right> ,
    \end{align*}
    where $\star$ is the convolution operator and for indices $i,j \in \left\{ 1, ..., N \right\}$ the kernel function $K^1_{i,j}[\rho](s) \in \Cc^0(S^i \times S^j, \RR^{d \times d})$ is given by:
    \begin{align*}
        \forall (x, y) \in S^i \times S^j, \quad
        K^1_{i,j}[\rho](s)(x,y) \eqdef \int_\Theta \Mm^i_\theta(\Lambda^i(s))(x) \otimes  \Mm^j_\theta(\Lambda^j(s))(y) \d \rho(\theta |s) .
    \end{align*}

    As in~\cref{def:NTK}, we denote by $\lambda_{\min}$ the smallest eigenvalue w.r.t. the mixed $L^2$-$\TV$-norm, i.e.
    \begin{align} \label{eq:lambda_min_V}
        \lambda_{\min}(\KK^1[\rho](s))
        \eqdef
        \min_{\mm} \KK^1[\rho](s)(\mm,\mm) ,
    \end{align}
    where the minimum is taken w.r.t. $\mm \in \Mm^d(S^1) \times ... \times \Mm^d(S^N)$ s.t. $\| \mm \|^2_{L^2-\TV} \leq 1$.
\end{definition}

\begin{example}[Finite number of tokens]
    Of particular interest is the case of a finite number of tokens, that is, when the token distributions are empirical distributions of the form $\mu^{(j)} = \frac{1}{n^j} \sum_{i=1}^{n^j} \delta_{x^j_i}$.
    In this case, the adjoint variables can be identified as families of vectors $\mm^{(j)} = (\mm^{(j)}_i)_{0 \leq i \leq n^j} \in (\RR^d)^{n^j+1}$ (assuming all tokens are distinct) and we have:
    $$
    \sum_{j=1}^N  \D_V \Aa^j_\theta (\Lambda^j_\rho(s))^* \mm^{(j)} = \sum_{j=1}^N \sum_{i=0}^{n^j} \mm^{(j)}_i(s) \otimes \Mm^j_\theta(\Lambda^j_\rho(s))(x^j_i) .
    $$
    The kernel matrix $\KK^1$ thus corresponds to a finite-dimensional matrix of size $n \times n$, where $n = \sum_{j=1}^N (n^j+1)$ is the total number of tokens used for training.
    Its entries are defined at time $s \in [0,1]$, for indices $j, j' \in \{1, ..., N\}$, $i \in \{0, ..., n^j \}$ and $i' \in \{0, ..., n^{j'} \}$ by:
    \begin{align*}
    \KK^1_{(j,i), (j', i')}[\rho](s)
    \eqdef{}& \int_\Theta
    \left< \Mm^j_\theta(\Lambda^j_\rho(s))(x^j_i),
    \Mm^{j'}_\theta(\Lambda^{j'}_\rho(s))(x^{j'}_{i'}) \right>
    \d \rho(\theta|s) .
    \end{align*}
\end{example}

The following result shows that, provided $\KK^1$ is strictly positive at initialization and the risk at initialization is sufficiently small, gradient flow converges towards a global minimizer of the risk with a linear convergence rate.
In particular, the convergence rate is explicit w.r.t.\@ the conditioning of $\KK^1$ and the sample size $N$.
In contrast with~\cref{cor:convergence_positive_NTK}, the positivity of $\KK^1$ only needs to be checked here at initialization.

\begin{theorem} \label{thm:local_convergence}
    Let $\Ee \geq 0$.
    Then there exists a constant $C = C(\Ee) > 0$ such that, for any parameter distribution $\rho_0 \in \PLeb$ satisfying $\Ee_2(\rho_0) \leq \Ee$ and
    $$
    \Ll(\rho_0) < C N^{-3} \lambda_0^3, \quad \text{with} \quad
    \lambda_0 \eqdef \int_0^1 \lambda_{\min}(\KK^1[\rho_0](s)) \d s,
    $$
    any gradient flow curve $(\rho_t)_{t \in [0, +\infty)}$ starting from $\rho_0$ satisfies:
    $$
    \rho_t \xrightarrow{t \to +\infty}
    \rho_\infty \in \PLeb, \quad \text{and} \quad \Ll(\rho_t) \leq \Ll(\rho_0) \exp(- C N^{-1} \lambda_0 t), \, \forall t \geq 0 .
    $$
\end{theorem}

\begin{proof}
    Consider $C_1 = C(2\Ee+2)$ the constant given by~\cref{lem:NTK_lipschitz} and let $R = \min(1,\frac{\lambda_0}{2C_1N})$.
    Then by~\cref{lem:NTK_lipschitz}, for every $\rho \in \PLeb$ with $\Ww_2^\COT(\rho, \rho_0) \leq R$ it holds:
    $$
    \int_0^1 \lambda_{\min}(\KK^1[\rho](s)) \d s \geq \lambda_0 /2 ,
    $$
    and, hence, by~\cref{prop:PL}:
    $$
    \left\| \nabla \Ll[\rho] \right\|^2_{L^2(\rho)} \geq C_2 N^{-1} \lambda_0 \Ll(\rho) ,
    $$
    for some constant $C_2 = C_2(\Ee) > 0$.
    The result then follows from~\cite[Thm.~1.4]{schiavo2023local}.
\end{proof}

\begin{lemma} \label{lem:NTK_lipschitz}
    For a parameter distribution $\rho \in \PLeb$, let $\KK^1$ be defined by~\cref{eq:NTK_V}.
    Then the map $\rho \mapsto \int_0^1 \lambda_{\min}(\KK^1[\rho](s)) \d s$ is locally Lipschitz.
    Precisely, for parameter distributions $\rho, \Tilde{\rho} \in \PLeb$ and corresponding kernel matrices, we have:
    \begin{align*}
        \left| \int_0^1 \lambda_{\min}(\KK^1[\rho](s)) \d s
        - \int_0^1 \lambda_{\min}(\KK^1[\Tilde{\rho}](s)) \d s \right|
        \leq  C N \Ww_2^\COT(\rho, \Tilde{\rho}) ,
    \end{align*}
    where $C = C(\Ee)$ is a constant depending on $\Ee = \max(\Ee_2(\rho), \Ee_2(\Tilde{\rho}))$.
\end{lemma}

\begin{proof}
    Consider the map $\Mm^j : \Theta \times \Cc^0(S^j, \RR^d) \to \Cc^0(S^j, \RR^d)$ defined in~\cref{eq:Mm_definition}.
    Observe that, proceeding as the proof of~\cref{lem:differential_A,lem:A_theta_param_differential} we can find a constant $C \geq 0$ s.t. for every index $j \in \left\{1, ..., N\right\}$, every $\Lambda \in \Cc^0(S^j, \RR^d)$ and every $\theta \in \Theta$ it holds:
    \begin{align} \label{eq:Mm_differential}
        \left\| \D_\Lambda \Mm^j_\theta(\Lambda) \right\| \leq C (1+\|\Lambda\|^2)(1+\|\theta\|)
        \quad \text{and} \quad 
        \left\| \D_\theta \Mm^j_\theta(\Lambda) \right\| \leq C (1+\|\Lambda\|^3) .
    \end{align}
    We can use the above estimates to bound the quantity $\left\| K^1_{i,j}[\rho](s) - K^1_{i,j}[\Tilde{\rho}](s) \right\|$.
    Indeed, fixing indices $i,j \in \left\{ 1, ..., N \right\}$, it follows from the linear growth of $\D_\Lambda \Mm^j$ w.r.t. $\theta$ in~\cref{eq:Mm_differential} and~\cref{lem:flow_locally_lipschitz} that for every $(s, \theta) \in [0,1] \times \Theta$ and every $x \in S^j$, we have:
    $$
    \left\| \Mm^j_\theta(\Lambda^j_\rho(s))(x) - \Mm^j_\theta(\Lambda^j_{\Tilde{\rho}}(s))(x)  \right\| \leq C \Ww_2^\COT(\rho, \Tilde{\rho}) ,
    $$
    for some constant $C = C(\Ee)$. Taking an outer product, we have analogously for every $(x, y) \in S^i \times S^j$:
    \begin{align*}
    &\left\| \Mm^i_\theta(\Lambda^i_\rho(s))(x) \otimes \Mm^j_\theta(\Lambda^j_\rho(s))(y)
    - \Mm^i_\theta(\Lambda^i_{\Tilde{\rho}}(s))(x)
    \otimes \Mm^j_\theta(\Lambda^j_{\Tilde{\rho}}(s))(y)  \right\| \\
    &\hspace{20em}
    \leq C \Ww_2^\COT(\rho, \Tilde{\rho}) ,
    \end{align*}
    for some constant $C = C(\Ee)$.
    Then using the boundedness of $\D_\theta \Mm$ w.r.t. $\theta$ in~\cref{eq:Mm_differential}, we have for every $s \in [0,1]$, every $ \theta, \Tilde{\theta} \in \Theta$ and every $x \in S^i$:
    $$
    \left\| \Mm^i_\theta(\Lambda^i_\rho(s))(x) - \Mm^i_{\Tilde{\theta}}(\Lambda^i_\rho(s))(x)  \right\| \leq C \| \theta - \Tilde{\theta} \| ,
    $$
    for some constant $C = C(\Ee)$.
    Taking an outer product, we have analogously for every $(x, y) \in S^i \times S^j$:
    \begin{align*}
    &\left\| \Mm^i_\theta(\Lambda^i_\rho(s))(x) \otimes \Mm^j_\theta(\Lambda^j_\rho(s))(y)
    - \Mm^i_{\Tilde{\theta}}(\Lambda^i_\rho(s))(x)
    \otimes \Mm^j_{\Tilde{\theta}}(\Lambda^j_\rho(s))(y) \right\| \\
    &\hspace{20em}
    \leq C \| \theta - \Tilde{\theta} \| ,
    \end{align*}
    for some constant $C = C(\Ee)$.

    At time $s \in [0,1]$, consider an optimal coupling $\gamma \in \Pp_2(\Theta \times \Theta)$ between $\rho(.|s)$ and $\Tilde{\rho}(.|s)$.
    Then, using the definition of $\KK^1_{i,j}[\rho](s)$ in~\cref{def:NTK_V} and the previous estimates, we obtain for every $(x,y) \in S^i \times S^j$:
    \begin{align*}
        &\left\| K^1_{i,j}[\rho](s)(x,y)
        - K^1_{i,j}[\Tilde{\rho}](s)(x,y) \right\| \\
        &\leq
        \int_\Theta \Bigl\|
        \Mm^i_\theta(\Lambda^i_\rho(s))(x) \otimes \Mm^j_\theta(\Lambda^j_\rho(s))(y)
        - \Mm^i_\theta(\Lambda^i_{\Tilde{\rho}}(s))(x)
        \otimes \Mm^j_\theta(\Lambda^j_{\Tilde{\rho}}(s))(y)
        \Bigr\| \d \Tilde{\rho}(\theta|s) \\
        &\quad + \int_{\Theta \times \Theta} \Bigl\|
        \Mm^i_\theta(\Lambda^i_\rho(s))(x) \otimes \Mm^j_\theta(\Lambda^j_\rho(s))(y)
        - \Mm^i_{\Tilde{\theta}}(\Lambda^i_\rho(s))(x)
        \otimes \Mm^j_{\Tilde{\theta}}(\Lambda^j_\rho(s))(y)
        \Bigr\| \d \gamma(\theta, \Tilde{\theta})  \\
        &\leq C \left( \Ww_2^\COT(\rho, \Tilde{\rho}) + \Ww_2(\rho(.|s), \Tilde{\rho}(.|s)) \right) ,
    \end{align*}
    for some constant $C = C(\Ee)$.
    Let $\Delta K_{i,j}(s) \eqdef K^1_{i,j}[\rho](s) - K^1_{i,j}[\Tilde{\rho}](s)$.
    Hence, for $\mm = (\mm^{(1)}, ..., \mm^{(N)}) \in \Mm^d(S^1) \times ... \times \Mm^d(S^N)$ we obtain:
    \begin{align*}
        \left| \KK^1[\rho](s)(\mm, \mm) -   \KK^1[\Tilde{\rho}](s)(\mm, \mm)\right|
        & \leq \sum_{1 \leq i,j \leq N}
        \left| \left< \mm^{(i)}, \Delta K_{i,j}(s) \star \mm^{(j)} \right> \right| \\
        & \leq C \left( \Ww_2^\COT(\rho, \Tilde{\rho})
        + \Ww_2(\rho(.|s), \Tilde{\rho}(.|s)) \right)
        \sum_{1 \leq i,j \leq N} \| \mm^{(i)} \|_\TV \| \mm^{(j)} \|_\TV \\
        & \leq C N \left( \Ww_2^\COT(\rho, \Tilde{\rho})
        + \Ww_2(\rho(.|s), \Tilde{\rho}(.|s)) \right)
        \| \mm \|^2_{L^2-\TV} ,
    \end{align*}
    implying 
    \[
    \left| \lambda_{\min} (\KK^1[\rho](s)) - \lambda_{\min}(\KK^1[\Tilde{\rho}](s)) \right|
    \leq C N \left( \Ww_2^\COT(\rho, \Tilde{\rho})
    + \Ww_2(\rho(.|s), \Tilde{\rho}(.|s)) \right)
    \]
    for some constant $C = C(\Ee)$. The result then follows by integrating over $s \in [0,1]$.
\end{proof}

\section{Conditioning analysis of the NTK of Attention}
\label{sec:NTK_expressivity}

In the previous section, we studied the training dynamics of deep Transformers with gradient flow and linked the geometry of the risk landscape with the conditioning of the tangent kernels of Attention layers.
Concretely, we showed in~\cref{cor:convergence_positive_NTK} that if the NTKs of Attention layers stay well-conditioned during training, then the gradient flow converges towards a global minimizer of the training risk at a linear rate.
We also showed a local convergence result in~\cref{thm:local_convergence}, assuming the NTKs associated with the $V$-derivatives of Attention layers are positive at initialization. In this section, we study the conditioning of the NTK of Attention layers and explain to what extent our convergence assumptions are satisfied for generic token distributions.
Precisely, \cref{prop:equivalence_injectivity_cumulant} characterizes the positivity of the Attention's NTK in terms of linear independence of log-sum-exp (cumulant generating) functions of the token distributions.
In \cref{subsec:examples} we then discuss this property for various examples of token distributions.
Importantly, in \cref{prop:linear-indep-disc} and \cref{rem-Justification}, we show that it is always almost surely satisfied for empirical token distributions. 

\subsection{NTK positivity via Attention feature map injectivity}

In~\cref{def:NTK,def:NTK_V}, we framed the conditioning of the Attention's NTK in terms of the positivity of the associated quadratic form, as quantified by the smallest eigenvalue, $\lambda_{\min}$. In this section, we will focus on an equivalent (but more qualitative) characterization of positivity in terms of the injectivity of some linear map.

\begin{proposition} \label{prop:equivalence_positivity_injectivity}
    Let $\rho \in \PLeb$ be some parameter distribution.
    For $s \in [0,1]$ let $\KK[\rho](s)$ and $\KK^1[\rho](s)$ be defined by~\cref{def:NTK} and~\cref{def:NTK_V} respectively.
    Then for every $s \in [0,1]$, we have $\lambda_{\min}(\KK[\rho](s)) > 0$ if and only if the map
    \begin{align} \label{eq:NTK_injective_map}
        \prod_{j=1}^N \Mm^d (S^j) \ni
        (\mm^{(j)})
        \mapsto \left[
        \Theta \ni \theta \mapsto
        \sum_{j=1}^N \int_{S^j} \D_\theta \varphi_\theta [\Lambda^{j}(s)_\# \mu^{(j)}]
        (\Lambda^{j}(s, x))^* \d\mm^{(j)}(x)
        \right] \in L^2(\rho(.|s))
    \end{align}
    is injective. Similarly, we have $\lambda_{\min}(\KK^1[\rho](s)) > 0$ if and only if the map
    \begin{align} \label{eq:NTK_V_injective_map}
        \prod_{j=1}^N \Mm^d (S^j) \ni
        (\mm^{(j)})
        \mapsto \left[
        \Theta \ni \theta \mapsto
        \sum_{j=1}^N \int_{S^j} \D_V \varphi_\theta [\Lambda^{j}(s)_\# \mu^{j}]
        (\Lambda^{j}(s, x))^* \d\mm^{(j)}(x)
        \right] \in L^2(\rho(.|s))
    \end{align}
    is injective.
    In particular, in the case $\Supp(\rho) \subset [0,1] \times \{ (V, Q, q) \in \Theta \colon V = 0 \}$, then $\Lambda^j_\rho(s) = \Id$ for every index $j \in \{1, ..., N\}$ and the above conditions do not depend on $s \in [0,1]$.
\end{proposition}

\begin{proof}
    We only provide the proof for $\KK$ since the proof for $\KK^1$ is similar.
    It directly follows from~\cref{eq:NTK} that $\lambda_{\min}(\KK[\rho](s)) > 0$ is equivalent to the injectivity of the map:
    $$
    \prod_{j=1}^N \Mm^d (S^j) \ni
        (\mm^{(j)})
        \mapsto \left[
        \Theta \ni \theta \mapsto
        \sum_{j=1}^N \D_\theta \Aa^j_\theta (\Lambda^{j}(s))^* \d\mm^{(j)}
        \right] \in L^2(\rho(.|s)) .
    $$
    Recalling the definition of the map $\Aa_\theta$ in~\cref{eq:A_theta}, for every index $j \in \{1, ..., N\}$ and every $x \in S^j$:
    $$
    \Aa^j_\theta(\Lambda^j(s))(x) = \varphi_\theta[\Lambda^j(s)_\# \mu^j](\Lambda^j(s,x)) , 
    $$
    where $\varphi_\theta$ is the Attention layer as defined in~\cref{eq:attention}.
    Then taking derivatives w.r.t. $\theta$, it follows that for $\mm^{(j)} \in \Mm^d(S^j)$:
    $$
    \D _\theta \Aa^j_\theta(\Lambda^j(s))^* \mm^{(j)} = \int_{S^j} \D_\theta \varphi_\theta[\Lambda^j(s)_\# \mu^{(j)}](\Lambda^j(s,x))^* \d \mm^{(j)}(x) ,
    $$
    which concludes the proof.
\end{proof}

In what follows, we fix some parameter distribution $\rho \in \PLeb$ and study the injectivity of the maps in~\cref{eq:NTK_injective_map,eq:NTK_V_injective_map}.
We will assume that this parameter distribution is of the product form $\rho = \Leb([0,1]) \otimes \Bar{\rho}$ for some $\Bar{\rho} \in \Pp(\Theta)$.
In this setting, we will, without loss of generality, assume that the parameter space is
$$
\Theta = \Supp(\Bar{\rho} )
$$
and we will check the pointwise injectivity of the maps in~\cref{eq:NTK_injective_map,eq:NTK_V_injective_map}.
That is, we will view them as maps from $\prod_{j=1}^N \Mm^d(S^j)$ to $\left\{ (Q, q, V) \in \RR^{d \times d} \times \RR^d \times \RR^{d \times d}  \right\}^\Theta$ or $\left\{  V \in \RR^{d \times d}  \right\}^\Theta$ respectively.

We will also focus on the case $\Lambda^j(s) = \Id$ for every $s \in [0,1]$.
As noted in~\cref{prop:equivalence_positivity_injectivity}, this is, for example, the case when the parameter distribution $\rho \in \PLeb$ is concentrated on a subset of the parameters where $V=0$. This \emph{FixUp} initialization is common practice and is associated with good generalization performances for the training of residual architectures without normalization layers~\cite{zhang2018fixup}. In this setting, we have the following relation between~\cref{eq:NTK_injective_map,eq:NTK_V_injective_map}:

\begin{proposition}
The injectivity of the map
\[
(\mm^{(1)}, ..., \mm^{(N)}) \in \Mm^d(S^1) \times ... \times \Mm^d(S^N)
\mapsto \left[ \Theta \ni \theta \mapsto
\sum_{j=1}^N \int \D_V \varphi_\theta[\mu^{(j)}](x)^*  d\mm^{(j)} \right] .
\] 
implies the injectivity of the map 
\[
(\mm^{(1)}, ..., \mm^{(N)}) \in \Mm^d(S^1) \times ... \times \Mm^d(S^N)
\mapsto \left[ \Theta \ni \theta \mapsto
\sum_{j=1}^N \int \D_\theta \varphi_\theta[\mu^{(j)}](x)^*  d\mm^{(j)} \right] .
\] 
Moreover, the converse is true whenever $\Theta$ has a non-empty interior.
\end{proposition}

\begin{proof}
($\Rightarrow$) We let 
\[
\sum_{j=1}^N \int \D_\theta \varphi_\theta[\mu^{(j)}](x)^*  d\mm^{(j)} = 0, \ \forall \theta
\]
Taking the derivative in the direction $V$, we obtain
\[
\sum_{j=1}^N \int \D_V \varphi_\theta[\mu^{(j)}](x)^*  d\mm^{(j)} = 0, \ \forall \theta
\]
By the assumption, $\mm^{(1)}= \cdots = \mm^{(N)} = 0$.
\bigskip\bigskip

($\Leftarrow$) Conversely, assume that 
\[
\sum_{j=1}^N \int \D_V \varphi_\theta[\mu^{(j)}](x)^*  d\mm^{(j)} = 0, \ \forall \theta.
\]
Taking the $\ell$-th element and the $v_{h\ell}$-derivative, where $v_{h\ell}$ is $(h, \ell)$-th element of $V$, we get
\[ 
\sum_{j=1}^N  \int_{S^j} \frac{\int e^{\langle Q x + q, y \rangle} y_h d\mu^{(j)}(y)}{\int e^{\langle Q x + q, y \rangle} d\mu^{(j)}(y)} \d \mm^{(j)}_\ell(x) = 0. 
\]
Multiplying this equality $v_{h\ell}$ and taking the sum $\sum_{h=1}^d$, we find that 
\[ 
\sum_{j=1}^N \int_{S^j} \frac{\int e^{\langle Q x + q, y \rangle} (V y)_\ell d\mu^{(j)}(y)}{\int e^{\langle Q x + q, y \rangle}  d\mu^{(j)}(y)} \d \mm^{(j)}_\ell(x) = 0, \ \ell = 1,...,d.
\]
Concatenating the elements ($\ell=1,...,d$), we obtain
\[ 
\sum_{j=1}^N \int \varphi_\theta[\mu^{(j)}](x)  d\mm^{(j)} =
\sum_{j=1}^N \int_{S^j} \frac{\int e^{\langle Q x + q, y \rangle} Vy d\mu^{(j)}(y)}{\int e^{\langle Q x + q, y \rangle}  d\mu^{(j)}(y)} \d \mm^{(j)}(x) = 0, \ \forall \theta.
\]
Then taking $D_\theta$, we get
\[ 
\sum_{j=1}^N \int D_\theta \varphi_\theta[\mu^{(j)}](x)^*  d\mm^{(j)} = 0, \ \forall \theta.
\]
By the assumption, we conclude that $\mm^{(1)}= \cdots = \mm^{(N)} = 0$.
\end{proof}

\subsection{Necessary and sufficient condition for injective NTK}

In what follows, we assume for simplicity that $\Lambda^j(s) = \Id$, 
and focus on the injectivity of NTK maps of attention 
under an assumption imposed only on the token initialization $\mu^{(j)}$, 
which is independent of the flow map $\Lambda^j$. This case corresponds to
$
\Supp(\rho) \subset [0,1] \times \{ (Q, q, V) \in \Theta \colon V = 0 \},
$
as shown in \cref{prop:equivalence_positivity_injectivity}, 
and in practice corresponds to a \emph{FixUp} initialization~\cite{zhang2018fixup}.
That is, we focus on the injectivity of the map
\[
(\mm^{(1)}, \dots, \mm^{(N)}) 
\in \Mm^d(S^1) \times \cdots \times \Mm^d(S^N)
\mapsto 
\left[
\Theta \ni \theta \mapsto
\sum_{j=1}^N \int \D_V \varphi_\theta[\mu^{(j)}](x)^* \, d\mm^{(j)}
\right].
\]
We also assume that
\[
\mathcal{K}_Q \times \mathbb{R}^d \times \mathcal{K}_V \subset \Theta, 
\]
where $\mathcal{K}_Q, \mathcal{K}_V \subset \mathbb{R}^{d \times d}$ are compact sets with non-empty interiors containing the origin.
We recall that
\[
\mu^{(j)} \in \Pp_c(\mathbb{R}^d) \text{ with } \Supp(\mu^{(j)}) \subset S^{j}, \quad j =1,2,\cdots,N, 
\]
where $S^j \subset \mathbb{R}^d$ are compact sets. 


\medskip

We define the cumulant generating functions by 
\[
g_{\mu}(q) := \log\left(\int e^{\langle q, y \rangle} d\mu(y)\right), \ q \in \mathbb{R}^d.
\]
and cumulant generating functions for fixed angle $e \in \mathbb{S}^{d-1}$ by
\[
g_{\mu,e}(\tilde{q}) := \log 
\left(\int e^{\tilde{q} \langle e, y \rangle} d\mu(y)\right), \quad \tilde{q} \in \mathbb{R}.
\]
Informally, $g_\mu(q)$ and $g_{\mu,e}(\tilde{q})$ can be seen as log-sum-exp aggregations of the linear scores $\langle q, y \rangle$ and $\tilde{q}\langle e, y \rangle$, respectively, over samples $y \sim \mu$, 
and thus provide smooth approximations of their maxima.
Furthermore, we define the set of affine functions by 
\[
\mathrm{Affine} :=\{ \mathbb{R}^d \ni q \mapsto \langle a, q \rangle + b \ : \ a \in \mathbb{R}^d, b \in \mathbb{R} \}.
\]
We define the following notion:
\begin{definition}
\label{def:indep}
We say that $\{\mu^{(j)}\}_{j=1}^N$ satisfies the weak independence property
if $\{g_{\mu^{(j)}}\}_{j=1}^N$ are linearly independent modulo affine functions, i.e.,
\begin{equation}
\label{eq:linear-indep-affine}
\sum_{j=1}^N C_j g_{\mu^{(j)}} \in \mathrm{Affine}
\ \Rightarrow\ 
C_j=0 \quad \forall j=1,\dots,N.
\end{equation}
We say that $\{\mu^{(j)}\}_{j=1}^N$ satisfies the strong independence property
if there exists $e \in \mathbb{S}^{d-1}$ such that
\begin{equation}
\label{eq:linear-indep-strong}
C_0 \mathrm{id} + \sum_{j=1}^N C_j g_{\mu^{(j)},e} = 0
\ \Rightarrow\ 
C_j=0 \quad \forall j=0,\dots,N.
\end{equation}
\end{definition}

We immediately observe that the following holds.
\begin{proposition}
\label{lem:criteria-injective}
The strong independence property implies the weak independence property.
\end{proposition}

We have the following proposition.
\begin{proposition} \label{prop:equivalence_injectivity_cumulant}
If $\{\mu^{(j)}\}_{j=1}^N$ satisfies the weak independence property, then, the map 
\[
(\mm^{(1)}, ..., \mm^{(N)}) \in \Mm^d(S^1) \times ... \times \Mm^d(S^N)
\mapsto \left[ \Theta \ni \theta \mapsto
\sum_{j=1}^N \int \D_V \varphi_\theta[\mu^{(j)}](x)^*  d\mm^{(j)} \right]
\] 
is injective. Moreover, the converse is true whenever 
\begin{equation}
\label{eq:ass:two-points}
\bigcap_{j=1}^N S^j \text{ has at least two distinct points.}
\end{equation}
\end{proposition}

This proposition implies that, under the condition \eqref{eq:ass:two-points}, the weak independence property is equivalent to the injectivity of the NTK map. In \cref{subsec:examples}, we provide several examples satisfying the strong independence property, which serve as a useful criterion for verifying the injectivity.

\begin{proof}

We assume that 
\[
\sum_{j=1}^N \int \D_V \varphi_\theta[\mu^{(j)}](x)^* \d\mm^{(j)}(x) = 0, \ \forall \theta,
\]
Then we claim that $\mm^{(j)} = 0$ for all $j$. Taking $\ell$-element and $v_{h\ell}$-derivative, where $v_{h\ell}$ is $(h, \ell)$-th element of $V$,
\[ 
\sum_{j=1}^N \int \frac{\int e^{\langle Q x + q, y \rangle} y_h d\mu^{(j)}(y)}{\int e^{\langle Q x + q, y \rangle} d\mu^{(j)}(y)} \d \mm^{(j)}_\ell(x) = 0, 
\]
where $ \mm^{(j)} = ( \mm^{(j)}_1, \cdots,  \mm^{(j)}_d)$ and $y = (y_1, \cdots, y_d)$. Let $e \in \mathbb{S}^{d-1}$ be any and fixed. Multiplying $e \in \mathbb{S}^{d-1}$, 
\[ 
\sum_{j=1}^N \int \frac{\int e^{\langle Q x + q, y \rangle} \langle e, y \rangle \d\mu^{(j)}(y)}{\int e^{\langle Q x + q, y \rangle} \d\mu^{(j)}(y)} \d \mm^{(j)}_\ell(x) = 0, 
\]
Fix $\ell \in [d]$. Let for some small $R>0$ so that
$$
Q = r e c^\top, \ q = se \text{ where }, \ c \in \mathbb{S}^{d-1}, \ s \in \mathbb{R}, r \in [-R,R]
$$
so that $Q \in \mathcal{K}$. Then, this implies that 
\[
\sum_{j=1}^{N} \int  
\frac{\int e^{(r \langle c, x \rangle + s) \langle e, y \rangle}  \langle e, y \rangle d\mu^{(j)}(y)}{\int e^{(r \langle c, x \rangle + s) \langle e, y \rangle} d\mu^{(j)}(y) } d \mm^{(j)}_\ell(x) =0
\]
We denote, for each $j \in [N]$ and $e \in \mathbb{S}^{d-1}$,
\[
h_j(e) := \max_{y \in \mathrm{supp} (\mu^{(j)})} \langle e, y \rangle < \infty.
\]
Such a maximum exists because $\mathrm{supp}(\mu^{(j)})$ are compact. We rewrite as 
\[
\sum_{j=1}^{N} \int  
\frac{\int e^{(r \langle c, x \rangle + s) \langle e, y  \rangle}  \langle e, y  \rangle d\mu^{(j)}(y)}{\int e^{(r \langle c, x \rangle + s) \langle e, y  \rangle} d\mu^{(j)}(y) } d \mm^{(j)}_\ell(x)
=
\sum_{j=1}^{N} \int  
\frac{\int H_x(y) e^{s \langle e, y \rangle}  \langle e, y  \rangle d\mu^{(j)}(y)}{\int H_x(y) e^{s \langle e, y  \rangle} d\mu^{(j)}(y) } d \mm^{(j)}_\ell(x)
=
 0
\]
where $H_x(y):=e^{(r \langle c, x \rangle + s) \langle e, y  \rangle}$ is continuous and $H_x > 0$ in $\Supp(\mu^{(j)})$.
Taking the limit as $s \to \infty$ and using \cref{lem:softmax-max},
\begin{align*}
\sum_{j=1}^{N} \int h_j(e) d \mm^{(j)}_\ell(x) = 0
\end{align*}
Thus, using this fact, we see that 
\[
\sum_{j=1}^{N} \int  
\frac{\int e^{(r \langle c, x \rangle + s) (\langle e, y  \rangle - h_j(e))}  (\langle e, y  \rangle - h_j(e)) d\mu^{(j)}(y)}{\int e^{(r \langle c, x \rangle + s) (\langle e, y \rangle-h_j(e))} d\mu^{(j)}(y) } d \mm^{(j)}_\ell(x)
= 0
\]
Using 
\[
\frac{d}{dr}
F_j(r \langle c,x \rangle + s) = \langle c,x \rangle \frac{d}{ds}
F_j(r \langle c,x \rangle + s),
\]
where 
\[
F_j(s) := \frac{\int e^{s (\langle e, y  \rangle -h_j(e))}  (\langle e, y \rangle - h_j(e)) d\mu^{(j)}(y)}{\int e^{s (\langle e, y \rangle -h_j(e))} d\mu^{(j)}(y) },
\]
we obtain that
\begin{align*}
&
\frac{d}{dr}
\sum_{j=1}^{N} \int  
\frac{\int e^{(r \langle c, x \rangle + s) (\langle e, y \rangle -h_j(e))}  (\langle e, y \rangle -h_j(e)) d\mu^{(j)}(y)}{\int e^{(r \langle c, x \rangle + s) (\langle e, y \rangle -h_j(e))} d\mu^{(j)}(y) } d \mm^{(j)}_\ell(x)
\\
&
=
\sum_{j=1}^{N} \int 
\langle c, x \rangle
\frac{d}{ds}
\frac{\int e^{(r \langle c, x \rangle + s) (\langle e, y \rangle - h_j(e))}  (\langle e, y \rangle - h_j(e)) d\mu^{(j)}(y)}{\int e^{(r \langle c, x \rangle + s) (\langle e, y  \rangle -h_j(e))} d\mu^{(j)}(y) } d \mm^{(j)}_\ell(x) =0.
\end{align*}
Taking the integral $\int_{s}^\infty$, we see that, since $F_j(s) \to 0$ as $s \to \infty$,
\[
\sum_{j=1}^{N} \int 
\langle c, x \rangle
\frac{\int e^{(r \langle c, x \rangle + s) (\langle e, y \rangle - h_j(e))}  (\langle e, y \rangle -h_j(e)) d\mu^{(j)}(y)}{\int e^{(r \langle c, x \rangle + s) (\langle e, y \rangle -h_j(e))} d\mu^{(j)}(y) } d \mm^{(j)}_\ell(x) =0.
\]
Repeating the above arguments, 
\[
\sum_{j=1}^{N} \int 
(\langle c, x \rangle)^m
\frac{\int e^{(r \langle c, x \rangle + s) (\langle e, y \rangle - h_j(e))}  (\langle e, y \rangle -h_j(e)) d\mu^{(j)}(y)}{\int e^{(r \langle c, x \rangle + s) (\langle e, y \rangle -h_j(e))} d\mu^{(j)}(y) } d \mm^{(j)}_\ell(x) =0,
\]
for any $m \in \mathbb{N}_0$. Setting $r=0$ implies that 
\[
\sum_{j=1}^{N} 
C^{(j)}_\ell(m) 
\frac{\int e^{s (\langle e, y \rangle- h_j(e))}  (\langle e, y\rangle- h_j(e)) d\mu^{(j)}(y)}{\int e^{s (\langle e, y \rangle- h_j(e))} d\mu^{(j)}(y) } =0, 
\]
where 
\[
C^{(j)}_\ell(m) := \int  (\langle c, x \rangle)^m d \mm^{(j)}_\ell(x) .
\]
We obtain
\[
\sum_{j=1}^{N} C^{(j)}_\ell(m) 
\frac{d}{ds}
\log \left( \int e^{s (\langle e, y\rangle - h_j(e))} d\mu^{(j)}(y) \right) = 0,
\]
and by taking the integral $\int_0^{s}$, 
\[
\sum_{j=1}^{N} C^{(j)}_\ell(m) 
\log \left( \int e^{ s(\langle e, y\rangle - h_j(e))} d\mu^{(j)}(y) \right) = 0, 
\]
which implies that 
\[
-
\left(\sum_{j=1}^{N} C^{(j)}_\ell(m) h_j(e) \right) s  +
\sum_{j=1}^{N} C^{(j)}_\ell(m)
\log \left( \int e^{ s \langle e,  y  \rangle} d\mu^{(j)}(y) \right) = 0, 
\]
Taking $\frac{d}{ds}|_{s=0}$, we obtain
\[
-
\left(\sum_{j=1}^{N} C^{(j)}_\ell(m) h_j(e) \right) + 
\sum_{j=1}^{N} C^{(j)}_\ell(m) \int
\langle e,  y  \rangle d\mu^{(j)}(y) = 0, 
\]
which implies that 
\[
\left(\sum_{j=1}^{N} C^{(j)}_\ell(m) h_j(e) \right) = \langle e, \sum_{j=1}^{N} C^{(j)}_\ell(m) \mathbb{E}_{\mu^{(j)}}[Y] \rangle.
\]
This holds for all $s, e$; thus, denoting $q = se$, we get 
\[
\langle q, \sum_{j=1}^{N} C^{(j)}_\ell(m) \mathbb{E}_{\mu^{(j)}}[Y] \rangle  +
\sum_{j=1}^{N} C^{(j)}_\ell(m)
\log \left( \int e^{ \langle q,  y  \rangle} d\mu^{(j)}(y) \right) = 0, \ \forall q.
\]
Thus, 
\[
\sum_{j=1}^{N} C^{(j)}_\ell(m) g_{\mu^{(j)}} \in \mathrm{Affine}.
\]
By the assumption, we obtain that 
\[
C^{(j)}_\ell(m) = \int  (\langle c, x \rangle)^m d \mm^{(j)}_\ell(x) = 0
\]
for all $m \in \mathbb{N}_0$. Hence, denoting by $P_c(x) = \langle c, x \rangle$, we find that 
\[
\int z^m \, d P_{c \sharp} \mm^{(j)}_\ell (z) = 0, \quad \forall m \in \mathbb{N}.
\]
Since $P_{c \sharp} \mm^{(j)}_\ell$ is supported on a compact set, polynomials are dense in the space of continuous functions on its support. Hence, the above identity implies that
\[
P_{c \sharp} \mm^{(j)}_\ell = 0.
\]
Since this holds for all $c \in \mathbb{S}^{d-1}$, we obtain that all one-dimensional projections of $\mm^{(j)}_\ell$ vanish.
Therefore, we conclude that $\mm^{(j)}_\ell=0$ for all $j = 1,...,N$ and all $\ell \in [d]$.

\medskip

For the converse statement, we assume that there exist $a \in \mathbb{R}^d$ and $b \in \mathbb{R}$ such that 
\[
\sum_{j=1}^N C^{(j)} g_{\mu^{(j)}}(\xi) = \langle a ,\xi \rangle + b, \ \forall \xi \in \mathbb{R}^d.
\]
Then we claim that $C^{(j)} = 0$ for all $j$. Indeed, taking $\D_{\xi}$, we get
\[
\sum_{j=1}^N C^{(j)} \frac{\int e^{\langle \xi, y \rangle} y d\mu^{(j)}(y)}{\int e^{\langle \xi, y \rangle} d\mu^{(j)}(y)} = a, \ \forall \xi.
\] 
Multiplying this equality by $V \in \mathcal{K}_V$,
\[
\sum_{j=1}^N C^{(j)} \frac{\int e^{\langle \xi, y \rangle} V y d\mu^{(j)}(y)}{\int e^{\langle \xi, y \rangle} d\mu^{(j)}(y)}
= Va,
\]
which holds for all $V, \xi$. From \eqref{eq:ass:two-points}, we can choose $x_1, x_2 \in \bigcap_{j=1}^N S^j$ such that $x_1 \neq x_2$. Taking $\xi = Qx_1 + q$ and $\xi = Qx_2 + q$, respectively, we obtain
\begin{align*}
&\sum_{j=1}^N C^{(j)} \frac{\int e^{\langle Qx_1 + q, y \rangle} V y d\mu^{(j)}(y)}{\int e^{\langle Qx_1 + q, y \rangle} d\mu^{(j)}(y)} = 
\sum_{j=1}^N C^{(j)} \frac{\int e^{\langle Qx_2 + q, y \rangle} V y d\mu^{(j)}(y)}{\int e^{\langle Qx_2 + q, y \rangle} d\mu^{(j)}(y)}
= Va.
\end{align*}
Denoting by $\mm^{(j)} := (C^{(j)} (\delta_{x_1} - \delta_{x_2}) , 0, \cdots, 0) \in \Mm^d(S^j)$, we obtain that 
\begin{align*}
&
\sum_{j=1}^N \int \frac{\int e^{\langle Q x + q, y \rangle} V y d\mu^{(j)}(y)}{\int e^{\langle Q x + q, y \rangle} d\mu^{(j)}(y)} \d \mm^{(j)}(x)
= \sum_{j=1}^N \int \varphi_\theta[\mu^{(j)}](x) \d \mm^{(j)}(x) 
\\
&
=
\sum_{j=1}^N C_j \frac{\int e^{\langle Q x_1 + q, y \rangle} (V y)_1 d\mu^{(j)}(y)}{\int e^{\langle Q x_1 + q, y \rangle} d\mu^{(j)}(y)} 
-
\sum_{j=1}^N C_j \frac{\int e^{\langle Q x_2 + q, y \rangle} (V y)_1 d\mu^{(j)}(y)}{\int e^{\langle Q x_2 + q, y \rangle} d\mu^{(j)}(y)} 
=
(Va)_1 - (Va)_1 = 0,
\end{align*}
i.e.,
\[
\sum_{j=1}^N \int \varphi_\theta[\mu^{(j)}](x) \d \mm^{(j)}(x) =0, \ \forall \theta = (V,Q,q)
\]
Taking $\D_V$,
\[
\sum_{j=1}^N \int \D_V \varphi_\theta[\mu^{(j)}](x)^* \d \mm^{(j)}(x) = 0, \ \forall \theta.
\]
By the assumption of the injectivity of the map
\[
(\mm^{(1)}, ..., \mm^{(N)}) \in \Mm^d(S^1) \times ... \times \Mm^d(S^N)
\mapsto \left[ \Theta \ni \theta \mapsto
\sum_{j=1}^N \int \D_V \varphi_\theta[\mu^{(j)}](x)^* \d\mm^{(j)}(x)  \right]
\]
we obtain that $\mm^{(j)} = 0$, in particular, $C^{(j)} = 0$ for all $j = 1,2,\cdots,N$. 

\end{proof}

\bigskip

\begin{lemma}\label{lem:softmax-max}
Let $\mu \in \mathcal{P}_c(\mathbb{R}^d)$, and let $G,H:\mathbb{R}^d\to\mathbb{R}$ be continuous functions.
Assume that
\begin{enumerate}
    \item $K:=\Supp(\mu)$ is compact, and
    \item $H>0$ on $K$.
\end{enumerate}
Then
\[
\lim_{s\to\infty}
\frac{\displaystyle{\int H(y)e^{sG(y)}G(y)\,d\mu(y)}}
{\displaystyle{\int H(y)e^{sG(y)}\,d\mu(y)}}
=
\max_{y\in \Supp(\mu)} G(y).
\]
\end{lemma}

\begin{proof}
Set
\[
M:=\max_{y\in K}G(y),
\qquad K=\Supp(\mu).
\]
We define
\[
D_s:=\int H(y)e^{sG(y)}\,d\mu(y),
\qquad
N_s:=\int H(y)e^{sG(y)}G(y)\,d\mu(y).
\]
We will prove that
\[
\frac{N_s}{D_s}\longrightarrow M
\qquad (s\to\infty).
\]
First, since $G(y)\le M$ on $K$, we have
\[
N_s
=
\int H(y)e^{sG(y)}G(y)\,d\mu(y)
\le
M\int H(y)e^{sG(y)}\,d\mu(y)
=
M D_s.
\]
Hence,
\[
\frac{N_s}{D_s}\le M
\qquad\text{for all } s.
\]
It remains to prove the lower bound. Fix $\varepsilon>0$ and consider
\[
U_\varepsilon:=\{y\in K:\ G(y)>M-\varepsilon\}.
\]
Since $G$ is continuous and attains its maximum $M$ on compact set $K$, we have
\[
\mu(U_\varepsilon)>0.
\]
Moreover, since $H$ is continuous and strictly positive on the compact set $K$,
there exists
\[
m_H:=\min_{y\in K}H(y)>0.
\]

For the definitions, we obtain
\[
M D_s - N_s
=
\int H(y)e^{sG(y)}(M-G(y))\,d\mu(y).
\]
Since $M-G(y)\ge 0$, we estimate
\[
M D_s - N_s
=
\int_{K} H(y)e^{sG(y)}(M-G(y))\,d\mu(y)
\le
\varepsilon \int_{U_\varepsilon} H(y)e^{sG(y)}\,d\mu(y)
+
C \int_{K\setminus U_\varepsilon} H(y)e^{sG(y)}\,d\mu(y),
\]
where
\[
C:=\max_{y\in K}(M-G(y))<\infty.
\]
Dividing by $D_s$ gives
\[
M-\frac{N_s}{D_s}
\le
\varepsilon
+
C\,
\frac{\int_{K\setminus U_\varepsilon} H(y)e^{sG(y)}\,d\mu(y)}
{\int_K H(y)e^{sG(y)}\,d\mu(y)}.
\]
But then it suffices to show that, as $s \to \infty$,
\[
\frac{\int_{K\setminus U_\varepsilon} H(y)e^{sG(y)}\,d\mu(y)}
{\int_K H(y)e^{sG(y)}\,d\mu(y)}
\longrightarrow 0.
\]
For $y\in K\setminus U_\varepsilon$, we have $G(y)\le M-\varepsilon$, so that
\[
\int_{K\setminus U_\varepsilon} H(y)e^{sG(y)}\,d\mu(y)
\le
e^{s(M-\varepsilon)}\int_K H(y)\,d\mu(y).
\]
Suppose we choose $0<\varepsilon'<\varepsilon$. Then the set,
\[
U_{\varepsilon'}:=\{y\in K:\ G(y)>M-\varepsilon'\},
\]
also satisfies $\mu(U_{\varepsilon'})>0$, and on $U_{\varepsilon'}$ we have
$G(y)> M-\varepsilon'$. Therefore,
\[
\int_K H(y)e^{sG(y)}\,d\mu(y)
\ge
m_H e^{s(M-\varepsilon')}\mu(U_{\varepsilon'}).
\]
Combining the last two estimates, we obtain
\[
\frac{\int_{K\setminus U_\varepsilon} H(y)e^{sG(y)}\,d\mu(y)}
{\int_K H(y)e^{sG(y)}\,d\mu(y)}
\le
\frac{\int_K H\,d\mu}{m_H\,\mu(U_{\varepsilon'})}
e^{-s(\varepsilon-\varepsilon')}.
\]
Since $\varepsilon-\varepsilon'>0$, the right-hand side tends to $0$ as
$s\to\infty$. Hence,
\[
\limsup_{s\to\infty}\left(M-\frac{N_s}{D_s}\right)\le \varepsilon.
\]
Because $\varepsilon>0$ was arbitrary, we get
\[
\liminf_{s\to\infty}\frac{N_s}{D_s}\ge M.
\]
Together with the already proven upper bound, $\frac{N_s}{D_s}\le M$, this implies
\[
\lim_{s\to\infty}\frac{N_s}{D_s}=M.
\]
This completes the proof.
\end{proof}

\subsection{Examples} \label{subsec:examples}

In this section, we give examples of measures $\mu^{(j)}$ for the necessity and sufficiency condition for injective NTK, that is, the weak independence property \eqref{eq:linear-indep-affine}. 
These examples concretely illustrate under which conditions on the token initialization $\mu^{(j)}$ the convergence in \cref{thm:local_convergence} holds.

\subsubsection{Examples of discrete measures}

Here, we explore examples of discrete measures. We let
\[
\mu^{(j)} :=  \sum_{i=1}^{n_j} w_i^{(j)} \delta_{x_i^{(j)}}, \quad j \in [N],
\]
where $n_j \in \mathbb{N}$ with $n_j \geq 2$. We assume that 
\[
\sum_{i=1}^{n_j} w_i^{(j)} = 1, \quad w_i^{(j)} \in (0,1),
\]
and that, for each $j \in [N]$, 
\begin{equation*}
x^{(j)}_{p} \neq x^{(j)}_{q}, \quad p \neq q, \ p,q \in [N].
\end{equation*}


\begin{proposition}
\label{prop:linear-indep-disc}
We assume that, for all $p \neq q$, $r \neq s$, and $i \neq j$ with $p,q \in [n_i]$ and $r,s \in [n_j]$, $i,j \in \{1,...,N\}$,
\begin{equation}
\label{eq:ass:dist}
(x^{(i)}_p - x^{(i)}_q) - (x^{(j)}_r - x^{(j)}_s) \neq 0.
\end{equation} 
Then, $\{\mu^{(j)}\}_{j=1}^N$ satisfies the strong independence property.
\end{proposition}

\begin{proof}

From \eqref{eq:ass:dist}, there exists $e \in \mathbb{S}^{d-1}$ such that 
\[
\langle e, (x^{(i)}_p - x^{(i)}_q) - (x^{(j)}_r - x^{(j)}_s) \rangle \neq 0, 
\]
for all $p \neq q$, $r \neq s$, and $i \neq j$.
Assume that 
\[
C_0 q + \sum_{j=1}^{N} C_j \log \left(\sum_{i=1}^{n_j} w_i^{(j)} \exp{(q \langle e, x_i^{(j)} \rangle )} \right)
= 0, \quad \forall q  \in \mathbb{R},
\]
where  $C_j \in \mathbb{R}$, which is equivalent to 
\[
\left(C_0 + \sum_{j=1}^{N} C_j \langle e, x_{i^*_j(e)}^{(j)} \rangle \right) q + \sum_{j=1}^{N} C_j \log \left( w_{i^*_j(e)}^{(j)} + \sum_{ i\neq i^*_j(e) } w_{i}^{(j)} \exp{(-q \langle e, (x_{i^*_j(e)}^{(j)} - x_i^{(j)}) \rangle )} \right)
= 0,
\]
where 
$$
i^{*}_j(e) := \mathrm{argmax}\{ \langle e,  x^{(j)}_i  \rangle : i \in [n_j] \}
$$
which is the maximizer of $\langle e,  x^{(j)}_i \rangle$. For $i \neq i^*_j(e)$, we define
\[
y^{(j)}_r := \langle e, x^{(j)}_{i^*_j(e)} - x^{(j)}_i \rangle,
\]
where $r$ indexes the elements in $[n_j-1]$. By definition of $i^*_j(e)$, we have $y^{(j)}_r \ge 0$. 
We may further arrange $\{y^{(j)}_r\}_{r \in [n_j-1]}$ in increasing order if needed.
Then, by \eqref{eq:ass:dist}
\[
y^{(j)}_p \neq  y^{(j)}_q\quad \text{ for all } p \neq q.
\]
We denote by 
\[
\alpha_j :=\min \{ y^{(j)}_r : r \in [n_j-1] \} >0.
\]
Then
\begin{align*}
\alpha_j 
&= \min \{ \langle e,  x^{(j)}_{i^*_j(e)}-x^{(j)}_i  \rangle : i \neq i^*_j(e) \}
= \langle e,  x^{(j)}_{i^*_j(e)}-x^{(j)}_i  \rangle - \max \{ \langle e,  x^{(j)}_i  \rangle : i \neq i^*_j(e) \}
\\
&
= \langle e,  x^{(j)}_{i^*_j(e)}-x^{(j)}_{i^{**}_j(e)}  \rangle,
\end{align*}
where 
$$
i^{**}_j(e) := \mathrm{argmax}\{ \langle e,  x^{(j)}_i  \rangle : i \neq i^*_j(e) \},
$$
which is the second-largest index for $\langle e,  x^{(j)}_i \rangle$. By \eqref{eq:ass:dist}, we see that 
\[
\alpha_j  \neq \alpha_i  \text{ for all } j \neq i,
\]
and applying Lemma~\ref{lem:log-lemma-}, we find that 
\[
\left(C_0 + \sum_{j=1}^{N} C_j \langle e, x_{i^*_j(e)}^{(j)} \rangle \right) = 0, \quad C_j =  0\quad \text{ for all } j \in [N],
\]
which implies that 
\[
C_j =  0\quad \text{ for all } j =0,1,2, \ldots, N.
\]
\end{proof}

\begin{remark}
By \cref{lem:criteria-injective} and \cref{prop:linear-indep-disc}, we obtain the weak independence property. 
However, we cannot replace \eqref{eq:ass:dist} with
\begin{equation}
\label{eq:ass:dist-alt}
\quad x^{(i)}_p \neq x^{(j)}_r, \quad p \neq r, \ i \neq j;
\end{equation}
i.e., \eqref{eq:ass:dist-alt} is not sufficient for the weak independence property \eqref{eq:linear-indep-affine} to hold. Indeed, suppose that $x^{(1)}_1$, $x^{(1)}_2$, $x^{(2)}_1$, $x^{(2)}_2$, and $x^{(i)}_k + x^{(j)}_\ell$ $(k,\ell, i, j = 1,2)$ are all distinct, and that $\mu^{(1)} = \frac{1}{2} \delta_{x^{(1)}_1} + \frac{1}{2}\delta_{x^{(1)}_2}$, $\mu^{(2)} = \frac{1}{2} \delta_{x^{(2)}_1} + \frac{1}{2} \delta_{x^{(2)}_2}$, and $\mu^{(3)} = \mu^{(1)} * \mu^{(2)} = \frac{1}{4} \sum_{k,\ell=1,2}  \delta_{x^{(1)}_k + x^{(2)}_\ell}$. This case satisfies \eqref{eq:ass:dist-alt} but not \eqref{eq:ass:dist}, and
$g_{\mu^{(1)}}$, $g_{\mu^{(2)}}$, $g_{\mu^{(3)}}$ are not linearly independent, as shown in \cref{subsec:negative}.
\end{remark}

The following proposition provides some justification for the assumption \eqref{eq:ass:dist}.

\begin{proposition}
\label{rem-Justification}
If 
\[
x^{(j)}_\ell \overset{\mathrm{iid}}{\sim} \mu, \quad \ell \in [n_j], \ j \in [N],
\]
where $\mu$ is the probability measure on $\RR^d$ that is absolutely continuous w.r.t. the Lebesgue measure $\lambda_d$, then we have that 
\begin{equation}
\label{almost-surely}
\mathbb{P}[\text{\eqref{eq:ass:dist}}] = 1.
\end{equation}
Thus, \eqref{eq:ass:dist} holds almost surely.
\end{proposition}

\begin{proof}
It is sufficient to prove the following: if $y_1,\dots,y_N$ are i.i.d.\ samples in $\mathbb{R}^d$ drawn from a probability measure $\mu$ that is absolutely continuous with respect to the Lebesgue measure $\lambda_d$ on $\mathbb{R}^d$, then
\[
\mathbb{P}\left[y_a - y_b \neq y_c - y_d \quad \text{for all distinct indices } a,b,c,d \right] = 1.
\]
Indeed, we set $Y=(y_a,y_b,y_c,y_d)\in\mathbb{R}^{4d}$, whose law is absolutely continuous with respect
to $\lambda_{4d}$. 
Since 
\[
\mathrm{dim}\{ Y : y_a -y_b = y_c-y_d \} = 3d < 4d, 
\]
we find that $\mathbb{P}[y_a -y_b = y_c-y_d]=0$. Thus,
\begin{align*}
\mathbb{P}&\left[\neg(y_a - y_b \neq y_c - y_d \quad\text{for all distinct } a,b,c,d) \right]    
\\
&= 
\mathbb{P}\left[\exists\, \text{ distinct } a,b,c,d  : \ y_a-y_b = y_c-y_d\right] 
\\
&\leq \sum_{\text{distinct }a,b,c,d } \mathbb{P}[y_a -y_b = y_c-y_d] = 0,
\end{align*}
which proves the claim.
\end{proof}

\begin{lemma}
\label{lem:log-lemma-}
We assume that 
\[
C_0 q + 
\sum_{j=1}^{N} C_j \log \left(a + \sum_{r=1}^{m_j}a_r^{(j)} \exp{(- q y_r^{(j)} )} \right)
= 0\quad \text{for all $q > 0$},
\]
where  $C_j \in \mathbb{R}$, and $a, a_r^{(j)} >0$, and $y^{(j)}_r>0$ satisfying that for each $j \in [N]$
\[
y^{(j)}_p \neq y^{(j)}_q\quad \text{ for all } p \neq q
\]
and
\[
\alpha_j \neq \alpha_i\quad \text{ for all } j \neq i,
\]
where $\alpha_j = \min \{ y^{(j)}_r : r = 1,...,m_j \}$. Then we have, 
\[
C_j = 0\quad \text{ for all } j =0,1,2,...,N.
\]
\end{lemma}

\begin{proof}
Without loss of generality, we can assume that $\{\alpha_j\}_{j \in [N]}$ is in increasing order, i.e.,
\[
0<\alpha_1 < \alpha_2 < \alpha_3 < \cdots < \alpha_N.
\]
The assumption implies that, by the Taylor expansion of $\log(1+x)$, 
\begin{align*} 
0 & = C_0 q + \sum_{j=1}^{N} C_j \log 
\left( a + \sum_{r=1}^{m_j} a_r^{(j)} \exp{(-q y_r^{(j)})} \right) 
\\
&
= C_0 q + \sum_{j=1}^{N} C_j  \log a 
+ \sum_{j=1}^{N} C_j \sum_{r=1}^{m_j} \frac{a_r^{(j)}}{a} \exp{(- q y_r^{(j)} )} 
+ O(\exp{(- 2 q \alpha_1 )})
\\ 
& 
= C_0 q + \sum_{j=1}^{N} C_j \log a + 
C_1 \frac{a_{r^*}^{(1)}}{a} \exp{(- q \alpha_1)}  + C_1 \sum_{r\neq r^* }\frac{a_r^{(1)}}{a} \exp{(- q y_r^{(1)} )}
+ o(\exp{(- q \alpha_1 )})
\\
& 
=
C_0 q + \sum_{j=1}^{N} C_j \log a +
C_1 \frac{a_{r^*}^{(1)}}{a} \exp{(- q \alpha_1)} + o(\exp{(- q \alpha_1)}), 
\end{align*}
because $\alpha_1 = \min_{r} y_r^{(1)}$. We also have denoted by $r^*:=\mathrm{argmin} \{ y^{(1)}_r : r = 1,...,m_1 \}$.
Multiplying by $\frac{1}{q}$ and taking $q \to \infty$, we find that
\[
C_0 =0.
\]
Taking the limit as $q \to \infty$ in the above equation once $C_0 = 0$ is determined, 
\[
\sum_{j=1}^N C_j = 0.
\]
Multiplying by $\exp{(q \alpha_1)}$ and taking the limit as $q \to \infty$, we get
\[
C_1 = 0.
\]
In a similar way,
\begin{align*}
0 & =\sum_{j=2}^{N} C_j \log \left(a + \sum_{r=1}^{m_j} a_{r}^{(j)} \exp{(- q y_r^{(j)} )} \right)
\\ 
& 
= 
\sum_{j=2}^{N} C_j \sum_{r=1}^{m_j} \frac{a_r^{(j)}}{a} \exp{(- q y_r^{(j)} )}
+ \mathcal{O}(\exp{(- 2 q \alpha_2 )})
\\
&  
= 
C_2 \frac{a_{r^{**}}^{(2)}}{a} \exp{(- q \alpha_2)} + o(\exp{(- q \alpha_2 )} ).
\end{align*}
Multiplying by $\exp{(q \alpha_2)}$ and taking the limit as $q \to \infty$, we get
\[
C_2 = 0.
\]
We repeat this process until $C_N=0$ is obtained.
\end{proof}

\subsubsection{Examples of measures with densities}

We explore examples of measure with densities. We first prove the following lemma:

\begin{proposition}
\label{lem:Vandermonde}
Let $N \in \mathbb{N}$ and let $f_1,\dots,f_N$ be defined on a neighborhood of $0$ with the following expansions 
\[
f_j(q) = \sum_{k \ge 1} \alpha_k\, s_j^{\,k}\, q^{2k},
    \qquad j = 1,\dots,N,
\]
where $\alpha_k \in \RR$ and $s_j \in \RR \setminus \{0\}$. Assume that
\begin{itemize}
    \item $\alpha_k \neq 0$ for $k = 1,\dots,N$;
    \item $s_i \neq s_j$ if $i \neq j$.
\end{itemize}
Then the family
\[
    \{\operatorname{Id}, f_1,\dots,f_N\}
\]
is linearly independent.
\end{proposition}

\begin{proof}
Suppose that
\[
c_0 q + \sum_{j=1}^N c_j f_j(q) = 0.
\]
Since each $f_j$ is even, the left-hand side is even, hence $c_0 = 0$. Differentiating $2k$ times at $0$ for $k = 1,\dots,N$ gives
\[
0 = \sum_{j=1}^N c_j f_j^{(2k)}(0)
  = (2k)! \,\alpha_k \sum_{j=1}^N c_j s_j^{\,k}.
\]
Because $\alpha_k \neq 0$, we have
\[
\sum_{j=1}^N c_j s_j^{\,k} = 0,
\qquad k = 1,\dots,N.
\]
Let $V = (s_j^{\,k})_{1\le k,j\le N}$, which is a Vandermonde matrix. Since the $s_j$ are pairwise distinct, $V$ is invertible; hence, $c_j = 0$ for all $j \in [N]$.
\end{proof}

\medskip

\begin{example}[Uniform distributions]
Let $0 < a_1 < \dots < a_N$ and let $\mu^{(j)}$ be the law of $X_j \sim \operatorname{Unif}([-a_j,a_j]^d)$. For any $e \in \mathbb{S}^{d-1}$, 
\[
    \{\operatorname{Id}, g_{\mu^{(1)},e},\dots,g_{\mu^{(N)},e}\}
\]
is linearly independent. Thus, $\{\mu^{(j)}\}_{j=1}^N$ satisfies the strong independence property.
\end{example}

\begin{proof}
Let $e = (e_1,\dots,e_d) \in \mathbb{S}^{d-1}$. We have
\[
  g_{\mu^{(j)},e}(q)
  = \log \left[\int e^{q\langle e,y\rangle}d\mu^{(j)} \right]
  = \sum_{i=1}^d \log\left(
    \frac{\sinh(a_j e_i q)}{a_j e_i q}
  \right).
\]
We observe that, in one dimension, the cumulant generating function,
\[
  g_{\operatorname{Unif}([-a,a])}(q) = \log\left(\frac{\sinh(aq)}{aq}\right),
\]
admits an even power series expansion
\[
  g_{\operatorname{Unif}([-a,a])}(q)
  = \sum_{k\ge 1} \gamma_k\, a^{2k} q^{2k},
\]
with real coefficients $\gamma_k \neq 0$ for all $k$. Using this fact, we see that 
\[
  g_{\mu^{(j)},e}(q)
  = \sum_{i=1}^d \sum_{k\ge 1} \gamma_k (a_j e_i)^{2k} q^{2k}
  = \sum_{k\ge 1} \gamma_k
      \Bigl( a^{2k}_j \sum_{i=1}^d e_i^{2k} \Bigr) q^{2k}.
\]
Thus $g_{\mu^{(j)},e}$ has the form
\[
g_{\mu^{(j)},e}(q) = \sum_{k\ge 1} \gamma_k\, s_{j}^{\,k}\, q^{2k}, \qquad s_j := a^{2k}_j \sum_{i=1}^d e_i^{2k},
\]
where $s_j \neq 0$ are pairwise distinct. By applying \cref{lem:Vandermonde}, we obtain the claim.
\end{proof} 

\medskip

\begin{example}[Symmetric multivariate Laplace]
Let $\mu^{(j)}$ be the law $X_j \sim \operatorname{Lap}(0,\Sigma_j)$, $j=1,\dots,N$, where $\Sigma_j$ are positive semidefinite $d \times d$ matrices. Fix $e \in \mathbb{S}^{d-1}$ and set
\[
s_j := \tfrac12 e^\top \Sigma_j e.
\]
Assume that $s_1,\dots,s_N$ are pairwise distinct and nonzero. Then,
\[
\{\operatorname{Id}, g_{\mu^{(1)},e},\dots,g_{\mu^{(N)},e}\}
\]
is linearly independent. Thus, $\{\mu^{(j)}\}_{j=1}^N$ satisfies the strong independence property.
\end{example}

\begin{proof}
Let $\Sigma$ be a positive semidefinite $d\times d$ matrix.
A centered symmetric multivariate Laplace distribution
$\operatorname{Lap}(0,\Sigma)$ can be defined via its characteristic function
\[
  \varphi_X(t)
  := \mathbb{E}_{X \sim \operatorname{Lap}(0,\Sigma)}\bigl[e^{i\langle t,X\rangle}\bigr]
  = \frac{1}{1 + \tfrac12 t^\top \Sigma t},
  \qquad t \in \RR^d.
\]
Let $X \sim \operatorname{Lap}(0,\Sigma)$ with law $\mu$, and fix $e \in \mathbb{S}^{d-1}$. Then the one-dimensional projection
\[
  Y := \langle e,X\rangle
\]
has characteristic function
\[
  \varphi_Y(q)
  = \mathbb{E}_Y\bigl[e^{iqY}\bigr]
  = \varphi_X(qe)
  = \frac{1}{1 + \tfrac12 q^2 e^\top \Sigma e},
  \qquad q \in \RR.
\]
Thus the moment generating function is given by 
\[
M_Y(q) = \varphi_Y(-i q)
  = \mathbb{E}_Y\bigl[e^{qY}\bigr]
  = \frac{1}{1 - \tfrac12 q^2 e^\top \Sigma e}.
\]
This implies that 
\[
g_{\mu,e}(q) 
= -\log\Bigl(1 - s q^2\Bigr), \qquad s := \tfrac12 e^\top \Sigma e.
\]
Expanding the logarithm yields
\[
  g_{\mu,e}(q)
  = \sum_{k\ge 1} \frac{1}{k} s^{k} q^{2k}.
\]
Thus, we have for each $j$
\[
g_{\mu^{(j)},e}(q)
= \sum_{k\ge 1} \frac{1}{k} s_j^{\,k} q^{2k},
\]
where $s_1,\dots,s_N$ are pairwise distinct by assumption. Hence, by \cref{lem:Vandermonde}, we obtain the claim. 
\end{proof}

\medskip

We now show that linear independence of the cumulant functions is stable under convolution with centered Gaussian measures.

\begin{proposition}[Stability under centered Gaussian convolution]
\label{lem:stability-Gaussian}
Let $\mu^{(1)},\dots,\mu^{(N)}$ be probability measures on $\RR^d$ and fix $e \in \mathbb{S}^{d-1}$.
Suppose that for each $j$
\[
g_{\mu^{(j)},e}(q) = \sum_{k\ge 1} \alpha_k s_j^{\,k} q^{2k},
\]
where $\alpha_k \neq 0$ and $s_1,\dots,s_N$ are pairwise distinct and nonzero. 
Let $\nu^{(j)} = \mathcal{N}(0,\Sigma_j)$ be centered Gaussian measures on $\RR^d$ with covariances $\Sigma_j$ and define
\[
\widetilde{\mu}^{(j)} := \mu^{(j)} * \nu^{(j)}.
\]
Then
\[
\{\operatorname{Id}, g_{\widetilde{\mu}^{(1)},e},\dots,g_{\widetilde{\mu}^{(N)},e}\}
\]
is linearly independent. Thus, $\{\widetilde{\mu}^{(j)}\}_{j=1}^N$ satisfies the strong independence property.
\end{proposition}

\begin{proof}
By the property of the convolution, we have
\[
g_{\widetilde{\mu}^{(j)},e}(q)
= g_{\mu^{(j)},e}(q) + \frac{1}{2}\, (e^\top \Sigma_j e)\, q^2.
\]
Suppose that
\[
c_0\, q + \sum_{j=1}^N c_j g_{\widetilde{\mu}^{(j)},e}(q) = 0.
\]
As the functions $g_{\widetilde{\mu}^{(j)},e}$ are even, 
we see that $c_0 = 0$. For $k \ge 2$, the quadratic terms vanish, so that
\[
0 = \sum_{j=1}^N c_j g_{\mu^{(j)},e}^{(2k)}(0)
  = (2k)! \,\alpha_k \sum_{j=1}^N c_j s_j^{\,k},
\qquad k = 2,\dots,N+1.
\]
Since $\alpha_k \neq 0$ for $k=2,\dots,N+1$, we obtain
\[
\sum_{j=1}^N c_j s_j^{\,k} = 0,
    \qquad k = 2,\dots,N+1.
\]
This is again a Vandermonde system in the $s_j$ (just with powers starting at $2$), and thus $c_j = 0$ for all $j$.
\end{proof}

\medskip

\begin{example}[Two-component multivariate Gaussian mixtures]
Fix $e \in \mathbb{S}^{d-1}$ and let $a_1,\dots,a_N > 0$ be pairwise distinct.
Let
\[
\mu^{(j)} := \frac{1}{2}\,\mathcal{N}(a_j e, \Sigma_j)
     + \frac{1}{2}\,\mathcal{N}(-a_j e, \Sigma_j),
    \qquad j=1,\dots,N,
\]
where each $\Sigma_j$ is a symmetric positive semidefinite covariance matrix. Then
\[
\{\operatorname{Id}, g_{\mu^{(1)},e},\dots,g_{\mu^{(N)},e}\}
\]
is linearly independent. Thus, $\{\mu^{(j)}\}_{j=1}^N$ satisfies the strong independence property.
\end{example}

\begin{proof}
We can view $\mu^{(j)}$ as the convolution of the symmetric two-point measure,
\[
  \nu^{(j)} := \tfrac12(\delta_{a_j e} + \delta_{-a_j e})
\]
with the centered Gaussian $\mathcal{N}(0,\Sigma_j)$, i.e., 
\[
  \mu^{(j)} = \nu^{(j)} * \mathcal{N}(0,\Sigma_j).
\]
For the discrete measure $\nu^{(j)}$, the projection $\langle e,X\rangle$ takes values
$\pm a_j$ with equal probability, so its moment generating function is
\[
  \mathbb{E}\bigl[e^{q\langle e,X\rangle}\bigr]
  = \frac{1}{2}\bigl(e^{q a_j} + e^{-q a_j}\bigr)
  = \cosh(a_j q),
\]
and, hence,
\[
  g_{\nu^{(j)},e}(q) = \log\cosh(a_j q).
\]
Expanding around $q = 0$ yields
\[
  g_{\nu^{(j)},e}(q)
  = \sum_{k \ge 1} \beta_k a_j^{2k} q^{2k}
  = \sum_{k \ge 1} \beta_k s_j^{k} q^{2k}, \quad \text{ where } s_j = a_j^2,
\]
where $\beta_k \in \RR$ are nonzero for all $k$. Thus the family $\{g_{\nu^{(1)},e},\dots,g_{\nu^{(N)},e}\}$ satisfies the assumptions
of Lemma~\ref{lem:stability-Gaussian} with $s_j := a_j^2$.
Consequently,
\[
  \{\operatorname{Id}, g_{\mu^{(1)},e},\dots,g_{\mu^{(N)},e}\}
\]
is linearly independent.
\end{proof}

\subsubsection{Negative examples}
\label{subsec:negative}

We provide examples of $\{\mu^{(j)}\}$ that do not satisfy the weak independence property \eqref{eq:linear-indep-affine}. 
These examples illustrate situations in which token initialization leads to the obstruction of convergence in \cref{thm:local_convergence}.


\begin{example}
The following are examples of measures $\mu^{(j)}$ not satisfying \eqref{eq:linear-indep-affine}, i.e., their functions $g_{\mu^{(j)}}$ are not linearly independent in the quotient space modulo affine functions:
\begin{itemize}
    \item $\mu^{(1)}$, $\mu^{(2)}$, $\mu^{(3)} = \mu^{(1)} \ast \mu^{(2)}$; then $g_{\mu^{(1)} \ast \mu^{(2)}} = g_{\mu^{(1)}} + g_{\mu^{(2)}}$.
    
	    \item $\mu^{(1)}$, $\mu^{(2)} = (S_a)_\sharp \mu^{(1)}$ where $S_a(x) = x- a$; then we have, $g_{(S_a)_\sharp \mu^{(1)}} = g_{\mu^{(1)}} - \langle \cdot, a \rangle$.
    \item $\mu^{(j)} = \delta_{x_j}$ are single Dirac; then $g_{\mu^{(j)}}(q) = \langle x_j, q \rangle$.
	    \item $\mu^{(j)} = \mathcal{N}(m^{(j)}, C)$ are Gaussian measures with means $m^{(j)}$ and covariance $C$ independent of $j$; then $g_{\mu^{(j)}}(q) = \langle m^{(j)}, q \rangle + \frac{1}{2} \langle q, C q \rangle$. 
	    \item $\mu^{(j)} = \mathcal{N}(m^{(j)}, C^{(j)})$ are Gaussian measures with means $m^{(j)}$ and covariances $C^{(j)}$ when $N \geq d + \frac{d(d+1)}{2}$. 
\end{itemize}
\end{example}





\section{Conclusion}

This paper provides a mathematical foundation for gradient-based training of transformers in the joint infinite-depth and infinite-width limit. We show that, unlike ResNets whose training follows a neural ODE on features, transformers are governed by a neural PDE that couples token distributions through attention, a structure arising from the simultaneous mean-field limits over tokens and heads. Using tools from optimal transport and Wasserstein gradient flows, we establish well-posedness of both the forward dynamics and training evolution, and characterize NTK injectivity for attention via linear independence of log-sum-exp functions. Although our convergence result is local, it applies to realistic deep, multi-head softmax attention on continuous token distributions and reveals an optimization landscape without spurious local minima when initialization is sufficiently close to optimal.
Several directions remain for future work. Our analysis is restricted to encoder-style (non-causal) attention; extending the framework to decoder architectures with causal masking poses substantial technical challenges, since the triangular structure of causal attention breaks the symmetries used in our proofs. In addition, practical transformers typically scale the embedding dimension with the number of heads. Capturing this coupled scaling within a mean-field limit remains an open problem, as our current formulation keeps the head dimension fixed while letting the number of heads grow to infinity.

\section*{Acknowledgements}

The work of Gabriel Peyr\'e was supported by the European Research Council (ERC project WOLF) and the French government under the management of Agence Nationale de la Recherche as part of the ``France 2030'' program, reference ANR-23-IACL-0008 (PRAIRIE-PSAI).
The work of Takashi Furuya is supported by JSPS KAKENHI Grant Number JP24K16949, 25H01453, JST CREST JPMJCR24Q5, JST ASPIRE JPMJAP2329.
Maarten de Hoop gratefully acknowledges
the support of the Department of Energy, BES program under grant DE-SC0020345, and the Simons Foundation under the MATH+X Program.
The work of Raphaël Barboni was supported by the European Union (ERC), through the Starting grant ‘PrSc-HDBayLe’,
project number 101076564.

\bibliographystyle{plain}
\bibliography{ref}

\newpage

\appendix

\section{\texorpdfstring{Proofs of~\cref{sec:training}}{Proofs of training results}}

This section is devoted to proving the results in~\cref{sec:training}.
Recall that we are considering Attention layers of the form~\cref{eq:attention}, parameterized by triplets $\theta = (Q, q, V)$ in the parameter set $\Theta = \RR^{d \times d} \times \RR^d \times \RR^{d \times d}$.
For such parameter $\theta = (Q, q, V) \in \Theta$, for a token distribution $\mu \in \Pp_c(\RR^d)$ and an input token $x \in \RR^d$ we define:
\begin{align} \label{eq:normalization_mean}
    N_{(Q,q)}[\mu](x) \eqdef \int_{\RR^d} e^{\left< Q x + q, y \right>} \d \mu(y) ,
    \quad \text{and} \quad M_{(Q,q)}[\mu](x) \eqdef \int_{\RR^d} e^{\left< Q x + q, y \right>} y \d \mu(y) .
\end{align}
With these notations, the Attention layer of~\cref{eq:attention} reads:
\begin{align*}
    \varphi_\theta[\mu](x) = N_{(Q,q)}[\mu](x)^{-1} \, V \cdot M_{(Q,q)}[\mu](x) . 
\end{align*}
For a parameter distribution $\rho \in \Pp_2(\Theta)$, the (mean-field) multi-head Attention defined in~\cref{eq:attention_meanfield} reads:
\begin{align*}
    \Phi_\rho[\mu](x) = \int_\Theta \varphi_\theta[\mu](x) \d \rho(\theta) . 
\end{align*}

Recall also that the point of view adopted in this paper is to see Attention layers as applications on the space of flow-maps $\Lambda \in \Cc^0(S, \RR^d)$, where $S$ is a compact subset of $\RR^d$ containing the support of the token distribution $\mu$ as well as the input token.
In~\cref{eq:A_theta}, we have defined for parameter $\theta \in \Theta$:
\begin{align} \label{eq:A_theta_appendix}
    \Aa_\theta :
    \left\{
    \begin{array}{rcl}
         \Cc^0(S, \RR^d) & \to & \Cc^0(S, \RR^d)   \\
         \Lambda & \mapsto & \varphi_\theta [\Lambda_\# \mu](\Lambda) 
    \end{array}
    \right. .
\end{align}
Additionally, in~\cref{eq:A_rho} we defined for parameter distribution $\rho \in \PLeb$:
\begin{align} \label{eq:A_rho_appendix}
    \Aa_\rho :
    \left\{
    \begin{array}{rcl}
         \Cc^0(S, \RR^d) & \to & \Cc^0(S, \RR^d)   \\
         \Lambda & \mapsto & \Phi_\rho [\Lambda_\# \mu](\Lambda) = \int_\Theta \Aa_\theta(\Lambda) \d \rho(\theta)
    \end{array}
    \right. .
\end{align}

\subsection{Differentiability of Attention w.r.t. inputs}
\label{subsec:regularity_attention_inputs}

We start by studying the regularity of the Attention map with respect to its input.

\begin{lemma} \label{lem:differential_A}
    Let $\mu \in \Pp_c(\RR^d)$ and $S \subset \RR^d$ be some compact set s.t. $\Supp(\mu) \subset S$.
    Then, for a parameter $\theta \in \Theta$, the map $\Aa_\theta$ in~\cref{eq:A_theta_appendix} is of class $\Cc^\infty$ w.r.t. $\Lambda$.
    In particular, for $\Lambda, h \in \Cc^0(S, \RR^d)$ we have:
    \begin{align} \label{eq:differential_A}
        \left( \D_\Lambda \Aa_\theta(\Lambda) \cdot h \right)(x)
        ={}& N_{(Q,q)}[\Lambda_\#\mu](\Lambda(x))^{-2}
        \iint_{\RR^d \times \RR^d}
        e^{\left< Q \Lambda(x) +q,  \Lambda(y)+\Lambda(y') \right>} \notag\\
        &\quad \times V \Lambda(y) ( \Lambda(y) - \Lambda(y')) Q h(x)
        \, \d \mu ^{\otimes 2}(y,y') \notag\\
        &+ N_{(Q,q)}[\Lambda_\#\mu](\Lambda(x))^{-2}
        \iint_{\RR^d \times \RR^d}
        e^{\left< Q \Lambda(x) +q, \Lambda(y)+\Lambda(y') \right>} \notag\\
        &\quad \times
        V \left(I + (\Lambda(y)-\Lambda(y')) (\Lambda(x) Q + q )^\top \right) h(y)
        \, \d \mu^{\otimes 2}(y,y') \; .
    \end{align}
    
    Moreover, for a parameter distribution $\rho \in \Pp_2(\Theta)$, the mapping $\Aa_\rho$ defined in~\cref{eq:A_rho_appendix} is of class $\Cc^1$ and its differential is given for any $\Lambda, h \in \Cc^0(S, \RR^d)$ by:
    \begin{align*}
        \D\Aa_\rho(\Lambda) \cdot h = \int_\Theta \D \Aa_\theta(\Lambda) \cdot h \d \rho(\theta).
    \end{align*}
\end{lemma}

\begin{proof}
    \proofpart{Regularity of $\Aa_\theta$}
    Fix some parameter $\theta = (Q, q, V) \in \Theta$.
    For maps $\Lambda, \Gamma \in \Cc^0(S, \RR^d)$ define:
    $$
    \Nn(\Lambda, \Gamma)(x) \eqdef N_\theta[\Lambda_\# \mu](\Gamma(x)) = \int_{\RR^d} e^{\left< Q \Gamma(x) +q, \Lambda(y) \right>} \d \mu(y) ,
    $$
    and
    $$
    \Mm(\Lambda, \Gamma)(x) \eqdef M_\theta[\Lambda_\# \mu](\Gamma(x)) = \int_{\RR^d} e^{\left< Q \Gamma(x) +q, \Lambda(y) \right>} \Lambda(y) \d \mu(y) .
    $$
    Then both $\Nn(\Lambda, \Gamma) \in \Cc^0(S, \RR)$ and $\Mm(\Lambda, \Gamma) \in \Cc^0(S, \RR^d)$ are well-defined.
    One can check that the maps
    $$
    \Nn : \Cc^0(S, \RR^d) \times \Cc^0(S, \RR^d) \to \Cc^0(S, \RR)
    $$
    and
    $$
    \Mm : \Cc^0(S, \RR^d) \times \Cc^0(S, \RR^d) \to \Cc^0(S, \RR^d)
    $$
    are of class $\Cc^\infty$ and that $\Nn$ is uniformly bounded from below on bounded sets.
    The result then follows since by~\cref{eq:A_theta_appendix}, for $\Lambda \in \Cc^0(S, \RR^d)$,
    $$
    \Aa_\theta(\Lambda) = \Nn(\Lambda, \Lambda)^{-1} \left( V \cdot \Mm(\Lambda, \Lambda) \right) .
    $$

    \proofpart{Distribution of parameter}

    For $\rho \in \Pp_2(\Theta)$ and $\Lambda \in \Cc^0(S, \RR^d)$, we have that the map $\theta \mapsto \Aa_\theta(\Lambda)$ is Bochner measurable since the space $\Cc^0(S, \RR^d)$ is separable.
    Also, we have $\Aa_\theta(\Lambda) \leq C \| V \|  \| \Lambda \|_{\Cc^0}$ for some constant $C$.
    Hence $\| \Aa_\theta(\Lambda) \|$ is $\rho$-integrable for $\rho \in \Pp_2(\Theta)$, implying that $\Aa_\rho(\Lambda) \in \Cc^0(S, \RR^d)$ is well-defined as a Bochner integral in~\cref{eq:A_rho_appendix}.
    
    Moreover, from~\cref{eq:differential_A} there exists a constant $C$ s.t. for $\Lambda \in \Cc^0(S, \RR^d)$ it holds:
    $$
    \| \D_\Lambda \Aa_\theta (\Lambda) \| \leq C(1+ \|\theta\|^2)(1+\|\Lambda\|^2) .
    $$
    Hence, by properties of the Bochner integral, we have for $\rho \in \Pp_2(\Theta)$ that $\Aa_\rho$ is of class $\Cc^1$ with, for any $\Lambda, h \in \Cc^0(S, \RR^d)$,
    $$
    \D\Aa_\rho(\Lambda) \cdot h = \int_\Theta \D \Aa_\theta(\Lambda) \cdot h \d \rho(\theta).
    $$ 
\end{proof}

\subsection{Differentiability of Attention w.r.t. parameters}
\label{subsec:regularity_attention_parameters}

We now study the regularity of the Attention map with respect to its parameter.

\begin{lemma} \label{lem:A_theta_param_differential}
    Let $S \subset \RR^d$ be some compact subset and $\mu \in \Pp_c(\RR^d)$ be some token distribution s.t. $\Supp(\mu) \subset S$.
    For $\theta \in \Theta$ and $\Lambda \in \Cc^0(S, \RR^d)$ let $\Aa_\theta(\Lambda)$ be defined by~\cref{eq:A_theta}.
    Then the map $\theta \mapsto \Aa_\theta(\Lambda)$ is of class $\Cc^\infty$ w.r.t. $\theta \in \Theta$.
    In particular, for $\Lambda \in \Cc^0(S, \RR^d)$, $\theta, \theta' \in \Theta$ and $x \in S$,
    \begin{align*}
        (\D_V \Aa_\theta (\Lambda) \cdot V')(x)
        &= N_{(Q,q)}[\Lambda_\#\mu](\Lambda(x))^{-1}
        \int_{\RR^d} e^{\left< Q \Lambda(x) + q, \Lambda(y) \right>}
        V' \Lambda(y) \d \mu(y) \, , \\
        (\D_Q \Aa_\theta (\Lambda) \cdot Q')(x)
        &= N_{(Q,q)}[\Lambda_\#\mu](\Lambda(x))^{-2}
        \iint_{\RR^d \times \RR^d}
        e^{\left< Q \Lambda(x) + q, \Lambda(y)+\Lambda(y') \right>} \\
        &\quad \times
        \left< Q' \Lambda(x), \Lambda(y)-\Lambda(y') \right>
        V \Lambda(y) \d \mu^{\otimes 2}(y, y') \, , \\
        (\D_q \Aa_\theta (\Lambda) \cdot q')(x)
        &= N_{(Q,q)}[\Lambda_\#\mu](\Lambda(x))^{-2}
        \iint_{\RR^d \times \RR^d}
        e^{\left< Q \Lambda(x) + q, \Lambda(y)+\Lambda(y') \right>} \\
        &\quad \times
        \left< q', \Lambda(y)-\Lambda(y') \right>
        V \Lambda(y) \d \mu^{\otimes 2}(y, y') \, .
    \end{align*}
\end{lemma}

\begin{proof}
    For simplicity, we only prove the formula for the first-order partial derivatives w.r.t.\@ $Q \in \RR^{d \times d}$.

    Fix the parameters $V, q$ and consider $\Lambda \in \Cc^0(S, \RR^d)$ and $x \in S$.
    Then the map $Q \mapsto \varphi_\theta [\Lambda_\# \mu](\Lambda(x))$ is continuously differentiable w.r.t. $Q$ and for every $Q, Q' \in \RR^{d \times d}$ it holds:
    \begin{multline*}
        \D_Q \left( \varphi_\theta [\Lambda_\# \mu](\Lambda(x)) \right) \cdot Q' \\
        = N_{(Q,q)}[\Lambda_\#\mu](\Lambda(x))^{-2} \iint_{\RR^d \times \RR^d} e^{\left< Q \Lambda(x) + q,  \Lambda(y) + \Lambda(y') \right>}  \left< Q' \Lambda(x), (\Lambda(y)-\Lambda(y')) \right> V \Lambda(y) \d \mu^{\otimes 2}(y,y')  \, .
    \end{multline*}
    Moreover, one can check that the differential $\D_Q \left( \varphi_\theta [\Lambda_\# \mu](\Lambda(x)) \right)$ is locally Lipschitz in $Q$ with a Lipschitz constant independent of $x$ (but depending on $\Lambda$ and $(V, q)$).
    Thus, for every $Q, Q' \in \RR^{d \times d}$ (with say $\| Q - Q' \| \leq 1$) we have:
    \begin{multline*}
        \left\| \varphi_{\theta'} [\Lambda_\# \mu](\Lambda(x)) - \varphi_\theta [\Lambda_\# \mu](\Lambda(x)) -  \D_Q \left( \varphi_\theta [\Lambda_\# \mu](\Lambda(x)) \right) \cdot (Q'-Q) \right\| \\
        \leq \| Q' - Q \| \int_0^1 \left\| \D_Q \left( \varphi_{\theta + u (\theta'-\theta)} [\Lambda_\# \mu](\Lambda(x)) \right) -  \D_Q \left( \varphi_\theta [\Lambda_\# \mu](\Lambda(x)) \right) \right\| \d u \leq L \| Q' - Q \|^2 \, ,
    \end{multline*}
    for some constant $L$ independent of $x$.
    Taking the supremum w.r.t.\@ $x \in S$ this leads to:
    \begin{align*}
        \left\| \Aa_{\theta'}(\Lambda) - \Aa_\theta(\Lambda) -  \D_Q \Aa_\theta (\Lambda) \cdot (Q'-Q) \right\|_{\Cc^0} \leq L \| Q'-Q \|^2 \, ,
    \end{align*}
    where $\D_Q \Aa_\theta(\Lambda)$ is as in the statement.
    This proves the desired result.
\end{proof}

\begin{lemma} \label{lem:A_rho_wasserstein_differential}
    Let $(\rho_t)_{t\in (0,1)}$ be some absolutely continuous curve in $\Pp_2(\Theta)$ with velocity field $(v_t)_{t \in (0,1)}$.
    Let $\mu \in \Pp_c(\RR^d)$ be some initial token distribution and $S \subset \RR^d$ be some compact set s.t. $\Supp(\mu) \subset S$.
    If $\Lambda \in \Cc^0(S, \RR^d)$, then for every $t_0 < t_1 \in (0,1)$ it holds:
    \begin{align*}
        \Aa_{\rho_{t_0}}(\Lambda) = \Aa_{\rho_{t_1}}(\Lambda) + \int_{t_0}^{t_1} \int_\Theta \D_\theta \Aa_{\theta}(\Lambda) \cdot v_t \d \rho_t \d t
    \end{align*}
    where, for $\Lambda \in \Cc^0(S, \RR^d)$, $\D_\theta \Aa_\theta(\Lambda) : \Theta \to \Cc^0(S,\RR^d)$ is the pointwise differential of $\Aa_\theta(\Lambda)$ w.r.t. $\theta$, i.e.:
    \begin{align*}
        \forall x \in S, \quad (\D_\theta \Aa_\theta(\Lambda) \cdot \delta \theta)(x) = \D_\theta \Aa_\theta(\Lambda)(x) \cdot \delta \theta.
    \end{align*}
\end{lemma}

\begin{proof}
    Let $\Lambda \in \Cc^0(S, \RR^d)$ and $x \in S$.
    Then, by~\cref{lem:A_theta_param_differential} we have that the map $\theta \mapsto \Aa_\theta(\Lambda)(x) =  \varphi_\theta[\Lambda_\# \mu](\Lambda(x))$ is differentiable, and for every $\theta \in \Theta$:
    \begin{align*}
        \left\| \D_\theta \Aa_\theta(\Lambda)(x) \right\| \leq C (1+\| \Lambda\|^3)(1 + \| \theta \|)
    \end{align*}
    for some absolute constant $C \geq 0$.
    Thus the map $t \in (0,1) \mapsto \Aa_{\rho_t}(\Lambda)(x)$ is absolutely continuous and for $t_0 < t_1 \in (0,1)$ it holds:
    \begin{align*}
        \Aa_{\rho_{t_0}}(\Lambda)(x) & = \Aa_{\rho_{t_1}}(\Lambda)(x) + \int_{t_0}^{t_1} \int_\Theta \D_\theta \Aa_{\theta}(\Lambda)(x) \cdot v_t \d \rho_t \d t \\
        & = \Aa_{\rho_{t_1}}(\Lambda)(x) + \int_{t_0}^{t_1} \int_\Theta (\D_\theta \Aa_{\theta}(\Lambda) \cdot v_t)(x) \d \rho_t \d t
    \end{align*}
    where we used the definition of $\D_\theta \Aa_{\theta}(\Lambda)$ to obtain the second equality.
    But then, using the bound on $\left\| \D_\theta \Aa_\theta(\Lambda) \right\|$, we have for every $t \in (0,1)$ and every $\theta \in \Theta$:
    $$
    \left\| \D_\theta \Aa_{\theta}(\Lambda) \cdot v_t(\theta) \right\|_{\Cc^0(S, \RR^d)} \leq C (1+\| \Lambda\|^3) (1 + \| \theta \|) \| v_t(\theta) \| .
    $$
    Hence the map $(t, \theta) \mapsto \D_\theta \Aa_{\theta}(\Lambda) \cdot v_t(\theta)$ is Bochner integrable (against $\d \pmb{\rho} = \int_{t_0}^{t_1} \d \rho_t \d t \in \Pp_2([0,1] \times \Theta)$) and by the previous equation:
    \begin{align*}
        \Aa_{\rho_{t_0}}(\Lambda) = \Aa_{\rho_{t_1}}(\Lambda) + \int_{t_0}^{t_1} \int_\Theta \D_\theta \Aa_{\theta}(\Lambda) \cdot v_t \d \rho_t \d t \, \in \Cc^0(S, \RR^d) ,
    \end{align*}
    which is the desired result.
\end{proof}

\begin{corollary} \label{cor:A_rho_lipschitz}
    Let $(\rho_t)_{t\in (0,1)}$ be some absolutely continuous curve in $\Pp_2(\Theta)$.
    Let $\mu \in \Pp_c(\RR^d)$ be some initial token distribution and $S \subset \RR^d$ be some compact set s.t.\@ $\Supp(\mu) \subset S$.
    Then there exists an absolute constant $C$ s.t.\@ for every $\Lambda \in \Cc^0(S, \RR^d)$ it holds:
    \begin{align*}
        \left\| \Aa_{\rho_1}(\Lambda) - \Aa_{\rho_0}(\Lambda) \right\|_{\Cc^0} \leq C ( 1 + \| \Lambda \|_{\Cc^0}^3 ) \left( 1 + \int_0^1 \Ee_2(\rho_t) \d t \right)^{1/2} \Ww_2(\rho_0, \rho_1) \, .
    \end{align*}
\end{corollary}

\subsection{\texorpdfstring{Proof of~\cref{prop:flow_differentiability}}{Proof of flow differentiability}} \label{subsec:proof_flow_differentiability}

We present here a proof of~\cref{prop:flow_differentiability} which shows the differentiability of the flow-maps $\Lambda$.
The result will actually follow by first proving the simpler~\cref{prop:flow_differentiability_appendix}, which is a result of differentiability of the flow-map along perturbations of the parameterization $\rho \in \PLeb$ of the form $\eps \mapsto (\Id + \eps v)_\# \rho$ for some $v \in L^2(\rho)$.
Indeed, those are actually the leading order perturbations along absolutely continuous curves in $\PLeb$ (see~\cite[Lem.~2.2]{barboni2024understanding}).

\begin{proposition} \label{prop:flow_differentiability_appendix}
    Let $\rho \in \PLeb$ be some parameter distribution, $v \in L^2(\rho, \Theta)$ be some vector-field and, for $t \in (-1, 1)$, let $\rho_t = (\Id + t v)_\# \rho$.
    For some initial token distribution $\mu_0 \in \Pp_c(\RR^d)$ and compact set $S \subset \RR^d$ s.t. $\Supp(\mu_0) \subset  S$, denote by $\Lambda_t = \Lambda_{\rho_t}[\mu_0] \in \Cc^0([0, 1] \times S, \RR^d)$ the flow-map obtained in~\cref{prop:existence_uniqueness} and for simplicity write $\Lambda = \Lambda_0$.
    Then the map $t \mapsto \Lambda_t \in \Cc^0([0, 1] \times S, \RR^d)$ is differentiable at $t = 0$ and its derivative $\delta \Lambda \eqdef \left. \frac{\d}{\d t} \Lambda_t \right|_{t=0} \in \Cc^0([0, 1] \times  S, \RR^d)$ is given by the unique solution of the Cauchy problem:
    \begin{align} \label{eq:delta_lambda}
        \delta \Lambda (s) = & \int_0^s \D_\Lambda \Aa_{\rho(r)}(\Lambda(r)) \cdot \delta \Lambda (r)  \d r + \int_0^s \int_\Theta  \D_\theta \Aa_\theta(\Lambda(r)) \cdot v(r) \d \rho(r) \d r
    \end{align}
    where for $\rho \in \Pp(\Theta)$, $\Aa_\rho : \Cc^0(S, \RR^d) \to \Cc^0(S, \RR^d)$ is the mapping defined in~\cref{eq:A_rho} and $\D_\Lambda \Aa_{\rho}(\Lambda)$ and $\D_\theta \Aa_\theta(\Lambda)$ are defined in~\cref{lem:differential_A} and~\cref{lem:A_theta_param_differential} respectively.
\end{proposition}

Thus, we obtain~\cref{prop:flow_differentiability} as a direct consequence of~\cref{prop:flow_differentiability_appendix}.

\begin{proof}[Proof of~\cref{prop:flow_differentiability}]
    Let $(\rho_t)_{t \in (0,1)}$ be as in the statement.
    Using~\cite[Lem.~2.2]{barboni2024understanding}, for $\d t$-a.e. $t \in [0,1]$:
    $$
    \Ww_2(\rho_{t+h}, \Tilde{\rho}_{t,h}) = o(|h|) ,
    $$
    where, for $h \in \RR$, we defined $\Tilde{\rho}_{t,h} = (\Id + h v_t)_\# \rho_t \in \PLeb$.
    Let us also define $\Tilde{\Lambda}_{t,h} \in \Cc^0([0,1] \times S, \RR^d)$ as the associated flow-map.
    Then using~\cref{lem:flow_locally_lipschitz} we have:
    $$
    \left\| \Lambda_{t+h} - \Tilde{\Lambda}_{t,h} \right\|_{\Cc^0([0,1] \times S, \RR^d)} = o(|h|)
    $$
    and hence using~\cref{prop:flow_differentiability_appendix} for $h \mapsto \Tilde{\rho}_{t,h}$:
    \begin{align*}
        \lim_{h \to 0} \left\| \frac{\Lambda_{t+h} - \Lambda_t}{h} - \delta \Lambda_t \right\|_{\Cc^0([0,1] \times S, \RR^d)} = 0 ,
    \end{align*}
    which is the desired result.
\end{proof}

Let us now prove~\cref{prop:flow_differentiability_appendix}.
We first study the Lipschitz regularity of the Attention map w.r.t. the weight distribution and conclude that Transformers outputs are locally Lipschitz w.r.t. their parameterization.

\begin{lemma} \label{lem:flow_locally_lipschitz}
    Let $\rho, \rho' \in \PLeb$ be some parameter distributions.
    Let $\mu_0 \in \Pp_c(\RR^d)$ be some initial token distribution and $\Lambda, \Lambda' \in \Cc^0([0,1] \times S,  \RR^d)$ be the flow-maps defined by~\cref{prop:existence_uniqueness} for parameterizations $\rho$ and $\rho'$ respectively. Then 
    \begin{align*}
        \sup_{s \in [0,1]} \| \Lambda(s) - \Lambda'(s) \|_{\Cc^0(S, \RR^d)} \leq  C \Ww_2^\COT(\rho, \rho')
    \end{align*}
    where the constant $C = C(\Ee)$ in particular depends on $\Ee = \max(\Ee_2(\rho), \Ee_2(\rho'))$.
    
    In particular, if $\mu, \mu' \in \Cc_{\co}([0, 1], \Pp_c(\RR^d))$ are the evolutions of the token distribution under the Transformer's forward flow for the respective parameterizations $\rho, \rho'$, then it holds:
    \begin{align*}
        \sup_{s \in [0,1]} \Ww_2(\mu(s), \mu'(s)) \leq C \Ww_2^\COT(\rho, \rho') \, .
    \end{align*}
\end{lemma}

\begin{proof}
    It suffices to prove the first inequality and the second directly follows as $\mu(s) = \Lambda(s)_\# \mu_0$ (resp. $\mu'(s) = \Lambda'(s)_\# \mu_0$) for any $s \in [0,1]$.
    For $s \in [0,1]$, let $\Aa_{\rho(.|s)}$ and $\Aa_{\rho'(.|s)}$ be defined by~\cref{eq:A_rho} s.t.\@ the flow-maps $\Lambda, \Lambda'$ given by~\cref{prop:existence_uniqueness} are respectively solutions to the Cauchy problems:
    \begin{align*}
        \forall s \in [0,1], \quad \Lambda(s) = \Id + \int_0^s \Aa_{\rho(.|r)}(\Lambda(r)) \d r, \quad \Lambda'(s) = \Id + \int_0^s \Aa_{\rho'(.|r)}(\Lambda'(r)) \d r \, .
    \end{align*}
    By the proof of~\cref{prop:existence_uniqueness}, there exists a constant $R = R(\Ee) > 0$ s.t.\@ for every $r \in [0,1]$ it holds $\| \Lambda(r) \|, \| \Lambda'(r) \| \leq R$ and in particular $\Supp(\Lambda(r)_\# \mu_0) \subset B(0,R)$ (and accordingly for $\Lambda'$).
    Then for $r \in [0,1]$, using~\cref{lem:differential_A,cor:A_rho_lipschitz}:
    \begin{align*}
        \| \Aa_{\rho(.|r)}(\Lambda(r)) - \Aa_{\rho'(.|r)}(\Lambda'(r)) \|_{\Cc^0}
        \leq & \left\| \Aa_{\rho(.|r)}(\Lambda(r)) - \Aa_{\rho'(.|r)}(\Lambda(r)) \right\|_{\Cc^0} \\
        & + \left\| \Aa_{\rho'(.|r)}(\Lambda(r)) - \Aa_{\rho'(.|r)}(\Lambda'(r)) \right\|_{\Cc^0} \\
        \leq & C (1+R^3) \left( 1 + \Ee_2(\rho(.|r)) + \Ee_2(\rho'(.|r)) \right)^{1/2} \Ww_2(\rho(.|r), \rho'(.|r)) \\
        & + C (1+R^2) \left( 1+ \Ee_2(\rho'(.|r)) \right) \| \Lambda(r) - \Lambda'(r) \|_{\Cc^0} \, .
    \end{align*}
    Plugging this into the ODEs defining $\Lambda$ and $\Lambda'$ and using Grönwall's lemma (\cref{lem:gronwall}) then gives for every $s \in [0,1]$:
    \begin{align*}
        \| \Lambda(s) - \Lambda'(s) \| \leq
        e^{\int_0^s C (1+R^2) \left( 1+ \Ee_2(\rho'(.|r)) \right) \d r}
        \left( \int_0^1 (1 + \Ee_2(\rho(.|r)) + \Ee_2(\rho'(.|r)) ) \d r \right)^{1/2} \Ww_2^{\COT} (\rho, \rho')
    \end{align*}
\end{proof}

We now give a proof of~\cref{prop:flow_differentiability_appendix}.

\begin{proof}[Proof of~\cref{prop:flow_differentiability_appendix}]
    In the following consider $R > 0$ s.t. for every $t \in (-1, 1)$ it holds $\| \Lambda_t \|_{\Cc^0} \leq R$.
    First note that the existence of a unique $\Gamma_0 \in \Cc^0([0,1] \times S, \RR^d)$ solution to the Cauchy problem~\cref{eq:delta_lambda} follows from a direct application of~\cref{thm:caratheodory}.
    Thus, it suffices to show $\Gamma_0$ is indeed the derivative of $\Lambda_t$ at $t = 0$.
    In this purpose, consider, for $t \neq 0$, the normalized increment:
    \begin{align*}
        \Gamma_t \eqdef \frac{1}{t} (\Lambda_t - \Lambda_0) \in \Cc^0([0, 1] \times S, \RR^d) .
    \end{align*}
    Then we have by definition of $\Lambda_t$ and $\Lambda_0$ that for every $s \in [0, 1] \times S$:
    \begin{align*}
        \Gamma_t(s) = & \frac{1}{t} \int_0^s \left( \Aa_{\rho_t(.|r)}(\Lambda_t(r)) - \Aa_{\rho_0(.|r)}(\Lambda_0(r)) \right) \d r \\
        = & \frac{1}{t} \int_0^s \left( \Aa_{\rho_t(.|r)} (\Lambda_t(r))(x) - \Aa_{\rho(.|r)}(\Lambda_t(r))(x) \right) \d r \\
        & + \frac{1}{t} \int_0^s \left( \Aa_{\rho_0(.|r)}(\Lambda_t(r)) - \Aa_{\rho_0(.|r)}(\Lambda_0(r)) \right) \d r \\
        = & \int_0^s \left( \int_0^1 \D_\Lambda \Aa_{\rho_0(.|r)}(\Lambda_0(r)+ u t \Gamma_t(r)) \cdot \Gamma_t(r) \d u \right)  \d r \\
        & +  \int_0^s \left( \int_0^1 \int_\Theta \D_\theta \Aa_{\theta + u t v(r,\theta)}(\Lambda_t(r)) \cdot v(r,\theta) \d \rho(\theta|r) \d u \right) \d r \, ,
    \end{align*}
    where we used~\cref{lem:differential_A} and~\cref{lem:A_rho_wasserstein_differential} respectively to write for a.e. $r \in [0,1]$:
    \begin{align*}
        \frac{1}{t} \left( \Aa_{\rho_0(.|r)}(\Lambda_t(r)) - \Aa_{\rho_0(.|r)}(\Lambda_0(r)) \right) & = \int_0^1 \D_\Lambda \Aa_{\rho_0(.|r)}(\Lambda_0(r)+ u t \Gamma_t(r)) \cdot \Gamma_t(r) \d u \\
        & = \left( \int_0^1 \D_\Lambda \Aa_{\rho_0(.|r)}(\Lambda_0(r)+ u t \Gamma_t(r)) \d u \right) \cdot \Gamma_t(r) \, ,
    \end{align*}
    the second equality coming from the fact that the map $u \mapsto \D_\Lambda \Aa_{\rho_0(.|r)}(\Lambda_0(r)+ u t \Gamma_t(r))$ is continuous and hence Bochner-integrable, and
    \begin{align*}
        \frac{1}{t} \left( \Aa_{\rho_t(.|r)}(\Lambda_t(r)) - \Aa_{\rho_0(.|r)}(\Lambda_t(r)) \right) & = \int_0^1 \int_\Theta \D_\theta \Aa_{\theta + u t v(r,\theta)}(\Lambda_t(r)) \cdot v(r,\theta) \d \rho(\theta|r) \d u \, .
    \end{align*}
    Hence $\Gamma_t$ is solution of the Cauchy problem $\Gamma_t(s) = \int_0^s \left( A_t(r) \cdot \Gamma_t(r) + b_t (r) \right) \d r$ where we defined for $\d r$-a.e. $r \in [0,1]$:
    \begin{align*}
        A_t(r) & \eqdef \int_0^1 \D_\Lambda \Aa_{\rho_0(.|r)}(\Lambda_0(r)+ u t \Gamma_t(r)) \d u  \in \Ll(\Cc^0(S,\RR^d), \Cc^0(S, \RR^d)) \\
        b_t(r) & \eqdef \int_0^1 \int_\Theta \D_\theta \Aa_{\theta + u t v(r,\theta)}(\Lambda_t(r)) \cdot v(r,\theta) \d \rho(\theta|r) \d u \in \Cc^0(S, \RR^d) \, .
    \end{align*}
    Let us further define:
    \begin{align*}
        A_0(r) & \eqdef D_\Lambda \Aa_{\rho_0(.|r)}(\Lambda_0(r)) \in \Ll(\Cc^0(S,\RR^d), \Cc^0(S, \RR^d)) \, , \\
        b_0(r) & \eqdef \int_\Theta \D_\theta \Aa_{\theta}(\Lambda_0(r)) \cdot v(r,\theta) \d \rho(\theta|r) \in \Cc^0(S, \RR^d) \, .
    \end{align*}
    such that $\Gamma_0$ is by definition solution of the Cauchy problem $\Gamma_0(s) = \int_0^s \left( A_0(r) \cdot \Gamma_0(r) + b_0 (r) \right) \d r$.
    Then for every $s \in [0,1]$ it holds:
    \begin{align*}
        \left\| \Gamma_t(s) - \Gamma_0(s) \right\|_{\Cc^0}
        &\leq \int_0^s
        \left\| A_t(r) \cdot \Gamma_t(r) - A_0(r) \cdot \Gamma_0(r) \right\|_{\Cc^0} \d r \\
        &\quad + \int_0^s \left\| b_t(r) - b_0(r) \right\|_{\Cc^0} \d r \\
        &\leq \int_0^s
        \Bigl( \left\| A_t(r) \right\|_{\Ll(\Cc^0, \Cc^0)}
        \left\| \Gamma_t(r) - \Gamma_0(r) \right\|_{\Cc^0} \\
        &\quad + \left\|  A_t(r) - A_0(r) \right\|_{\Ll(\Cc^0, \Cc^0)}
        \left\| \Gamma_0(r) \right\|_{\Cc^0}
        + \left\| b_t(r) - b_0(r) \right\|_{\Cc^0} \Bigr) \d r ,
    \end{align*}
    leading, by an application of Grönwall's lemma, to:
    \begin{align*}
        \left\| \Gamma_t(s) - \Gamma_0(s) \right\|_{\Cc^0}
        &\leq
        e^{\int_0^1 \left\| A_t(r) \right\|_{\Ll(\Cc^0, \Cc^0)} \d r} \\
        &\quad \times
        \int_0^1
        \left( \left\|  A_t(r) - A_0(r) \right\|_{\Ll(\Cc^0, \Cc^0)}
        + \left\| b_t(r) - b_0(r) \right\|_{\Cc^0} \right) \d r.
    \end{align*}
    Now, by~\cref{lem:flow_locally_lipschitz} we have $\left\| \Gamma_t(r) \right\|_{\Cc^0} \leq \| v \|_{L^2(\rho)}$ for every $t \neq 0$ and every $r \in [0,1]$.
    Also, the maps $\Lambda \in \Cc^0(S, \RR^d) \mapsto \D_\Lambda \Aa_{\rho}(\Lambda)$ and $(\theta, \Lambda) \mapsto \D_\theta \Aa_\theta(\Lambda)$ are continuous.
    Thus for a.e. $r \in [0,1]$ it holds $A_t(r) \to A_0(r)$ and $b_t(r) \to b_0(r)$ and hence by application of Lebesgue's theorem:
    \begin{align*}
        \sup_{s \in [0,1]} \left\| \Gamma_t - \Gamma_0 \right\|_{\Cc^0} \leq e^{\int_0^1 \left\| A_t(r) \right\|_{\Ll(\Cc^0, \Cc^0)} \d r} \int_0^1 \left( \left\|  A_t(r) - A_0(r) \right\|_{\Ll(\Cc^0, \Cc^0)} + \left\| b_t(r) - b_0(r) \right\|_{\Cc^0} \right) \d r \xrightarrow[t \to 0]{} 0 \, ,
    \end{align*}
    which concludes the proof.
\end{proof}

\subsection{Equivalence between gradient flow and curve of maximal slope} \label{subsec:app_maximal_slope_equivalence}

The purpose of this section is to provide a proof for~\cref{thm:maximal_slope_equivalence}.

\begin{proof}
    \proofpart{gradient flow $\Rightarrow$ curve of maximal slope}
    Let $\rho_0 \in \PLeb$ be some initial parameter distribution and let $(\rho_t)_{t\geq 0}$ be some gradient flow curve starting from $\rho_0$ in the sense of~\cref{def:gradient_flow}.
    Then $(\rho_t)_{t\geq 0}$ is a locally absolutely continuous curve in $\Ww_2^\COT$ satisfying the continuity equation~\cref{eq:gradient_flow_continuity} with the velocity field $\nabla \Ll[\rho_t]$.
    Hence, using~\cite[Prop.~2.4]{barboni2024understanding}, it holds for a.e. $t \geq 0$:
    $$
    \left| \frac{\d}{\d t} \rho_t \right| \leq \left\| \nabla \Ll[\rho_t] \right\|_{L^2(\rho_t)} .
    $$
    Also, using~\cref{thm:upper_gradient}, the map $t \mapsto \Ll(\rho_t)$ is locally absolutely continuous and by the definition of $\nabla \Ll$ in~\cref{eq:gradient_field} we have for a.e. $t \geq 0$:
    $$
    - \frac{\d}{\d t} \Ll(\rho_t) = \left\| \nabla \Ll(\rho_t) \right\|^2_{L^2(\rho_t)} .
    $$
    Since by construction $\left| \nabla \Ll \right|(\rho) = \left\| \nabla \Ll(\rho) \right\|_{L^2(\rho)}$ (\cref{eq:upper_gradient}), we obtain~\cref{eq:EDI} by combining the two previous equations.

    \proofpart{curve of maximal slope $\Rightarrow$ gradient flow}
    Let $\rho_0 \in \PLeb$ be some initial parameter distribution and let $(\rho_t)_{t\geq 0}$ be a curve of maximal slope starting from $\rho_0$ (in the sense given in~\cref{thm:maximal_slope_equivalence}).
    Then by~\cite[Prop.~2.4]{barboni2024understanding}, there exists a Borel velocity field $v : (0, +\infty) \times [0,1] \times \Theta \to \Theta$ s.t. $\rho$ satisfies the continuity equation
    $$
    \partial_t \rho_t + \div_\theta (\rho_t v_t) = 0, \quad \text{on $(0, +\infty) \times [0,1] \times \Theta$},
    $$
    and s.t. $\left| \frac{\d}{\d t} \rho_t \right| \geq \left\| v_t \right\|_{L^2(\rho_t)}$ for a.e. $t \geq 0$.
    Hence, using~\cref{thm:upper_gradient}, the map $t \mapsto \Ll(\rho_t)$ is locally absolutely continuous and it holds for a.e. $t \geq 0$:
    $$
    \frac{\d}{\d t} \Ll(\rho_t) = \left< v_t, \nabla \Ll(\rho_t) \right>_{L^2(\rho_t)} .
    $$
    Using~\cref{eq:EDI} as well as Young's inequality then gives that $v_t = \nabla \Ll(\rho_t)$ for a.e. $t \geq 0$.
\end{proof}

\subsection{Well-posedness of the gradient flow equation}
\label{subsec:app_existence_uniqueness}

This section is devoted to the proof of~\cref{thm:gradient_flow_wellposed}, that is well-posedness of the gradient flow equation for the training of mean-field models of Transformers as presented in~\cref{sec:training}.
We will rely on classical results from the theory of gradient flows in metric spaces (see~\cite{ambrosio2008gradient} and~\cite{santambrogio2017euclidean}) showing that gradient flow curves can be obtained as limits of proximal schemes.
Historically, such a notion of a \emph{Minimizing Movements} scheme dates back to De Giorgi~\cite{de1993new}.

To apply this methodology to the case of the training risk $\Ll$ of mean-field Transformers a few regularity assumptions need to be satisfied (see~\cite[Section 2.1]{ambrosio2008gradient}).
For this purpose we first consider a modification of the Attention architecture.
For parameters $(Q, q, V) \in \Theta$ we consider Attention layers of the form:
\begin{align} \label{eq:attention_modified}
    \varphi_{\theta} [\mu](x) = \Pi(V) \cdot \frac{\int e^{\left< Q x + q, y \right>} y \, d\mu(y)}{\int e^{\left< Q x + q, y \right>} \, \text{d}\mu(y)} ,
\end{align}
where $\Pi : \RR^{d \times d} \to \RR^{d \times d}$ is any smooth bounded function.

Note that all the regularity results on the Attention mapping detailed in~\cref{subsec:regularity_attention_inputs,subsec:regularity_attention_parameters} carry over to the case of this modified Attention.
In particular, \cref{thm:maximal_slope_equivalence,thm:upper_gradient} still hold.
However, thanks to the use of the function $\Pi$, bounding the parameters $V$, we obtain a better dependence of the derivatives of the Attention map w.r.t. its parameters.
Namely, considering a compact set $S \subset \RR^d$ with $\Supp(\mu) \subset S$ and $\Lambda \in \Cc^0(S, \RR^d)$ with $\| \Lambda \|_{\Cc^0} \leq R$, we have that there exists a constant $C = C(S, R) \geq 0$ such that it holds for any $\theta \in \Theta$:
\begin{equation} \label{eq:modified_attention_growth}
\begin{aligned}
    \left\| \D_\Lambda \Aa_\theta(\Lambda) \right\| & \leq C(1+\|\theta\|), \\
    \left\| \D^2_{\Lambda,\Lambda} \Aa_\theta(\Lambda) \right\| & \leq C(1+\|\theta\|^2), \\
    \left\| \D^2_{\theta,\theta} \Aa_\theta(\Lambda) \right\| & \leq C, \\
    \left\| \D^2_{\Lambda,\theta} \Aa_\theta(\Lambda) \right\| & \leq C (1+\|\theta\|) .
\end{aligned}
\end{equation}
These growth estimates are crucial to obtain existence (\cref{thm:gradient_flow_existence}) and uniqueness (\cref{thm:gradient_flow_uniqueness}) of gradient flow curves for general initial parameter distributions.
In turn, we are able to extend the result for the case of a normal Attention (i.e., $\Pi= \Id$) when the initial parameter distribution has a $V$-marginal which is compactly supported (\cref{thm:wellposedness_compact}).

\subsubsection{Existence of gradient flow curves}

We start by stating the existence of gradient flow curves for the minimization of the risk in~\cref{eq:risk}.
These are obtained as the limit of a \emph{Minimizing Movements} scheme~\cite{de1993new}.

\begin{theorem} \label{thm:gradient_flow_existence}
    Let the Attention map be defined by~\cref{eq:attention_modified} and let the training risk $\Ll : \PLeb \to \RR$ be defined by~\cref{eq:risk}.
    For any initialization $\rho_0 \in \PLeb$ there exists a gradient flow curve $(\rho_t)_{t \geq 0}$ starting from $\rho_0$.
\end{theorem}

\begin{proof}
    The result will follow from the successive application of~\cite[Thm.~2.2.3 and Thm.~2.3.3]{ambrosio2008gradient}.
    Let us denote by $\tau$ the topology of narrow convergence on $\PLeb$
    (i.e., the weak-* topology in duality with $\Cc^0([0,1] \times \Theta)$).
    In particular, $\Ww_2^\COT$-bounded sets are $\tau$-sequentially compact.
    In order to prove the result we need to show that the risk $\Ll$ and its upper gradient $\left| \nabla \Ll \right|$ are $\tau$-lower-semicontinuous on $\Ww_2^\COT$-bounded subsets.

    \proofpart{$\Ll$ is $\tau$-continuous on $\Ww_2^\COT$-bounded subsets.}
    Let $(\rho_n)_{n \geq 0}$ be a $\Ww_2^\COT$-bounded sequence in $\PLeb$ s.t.\@ $\rho_n \xrightarrow{\tau} \rho \in \PLeb$.
    For an input token distribution $\mu \in \Pp_c(\RR^d)$, let us write $\Lambda_n = \Lambda_{\rho_n}[\mu]$.
    Using Ascoli's theorem, we will proceed by showing that $\Lambda_n \to \Lambda = \Lambda_\rho[\mu]$ in $\Cc([0,1] \times S, \RR^d)$ for any compact set s.t. $\Supp(\mu) \subset S$.
    
    From the $\Ww_2^\COT$-boundedness it follows that the sequence $(\Lambda_n)_{n \geq 0}$ is bounded.
    Using the linear growth of $\Aa_\theta$ w.r.t. its parameters we get that there exists a constant $C \geq 0$ s.t. for any $n \geq 0$ and any $s_1 \leq  s_2 \in [0,1]$:
    \begin{align*}
        \left\| \Lambda_n(s_1) - \Lambda_n(s_2) \right\|
        \leq \int_{s_1}^{s_2} \int_{\Theta} \left\| \Aa_{\theta}(\Lambda_n(r)) \right\| \d \rho(\theta|r) \d r
        \leq \int_{s_1}^{s_2} \int_{\Theta} C (1+ \| \theta \|) \d \rho(\theta|r) \d r .
    \end{align*}
    Then, using the fact that the sequence $(\rho_n)_{n \geq 0}$ has uniformly integrable first-order moments, we get that for every $\eps > 0$, there exists $k \geq 0$ s.t.
    \begin{align*}
        \left\| \Lambda_n(s_1) - \Lambda_n(s_2) \right\|
        \leq \eps + C(1+k) |s_2 - s_1| .
    \end{align*}
    Thus, the curves $(\Lambda_n)_{n \geq 0}$ are equicontinuous and, by Arzel\`a--Ascoli's theorem, we have (up to a subsequence) that $\Lambda_n \to \Bar{\Lambda} \in \Cc([0, 1] \times S, \RR^d)$.
    
    Let us show that $\Bar{\Lambda} = \Lambda = \Lambda_\rho[\mu]$.
    For $n \geq 0$, $s \in [0,1]$ and $x \in S$ we have:
    \begin{align*}
        \Lambda_n(s,x) = x & + \int_{[0,1] \times \Theta} \ind_{r \leq s} \Aa_\theta (\Lambda_n(r))(x) \d \rho_n(r, \theta) \\
        = x & + \int_{[0,1] \times \Theta} \ind_{r \leq s} \Aa_\theta (\Bar{\Lambda}(r))(x) \d \rho(r,\theta) \\
        & + \int_{[0,1] \times \Theta} \ind_{r \leq s} \left( \Aa_\theta (\Lambda_n(r)) - \Aa_\theta (\Bar{\Lambda}(r)) \right)(x) \d \rho_n(r,\theta) \tag{L1} \label{eq:gradient_flow_L1} \\
        & + \int_{[0,1] \times \Theta} \ind_{r \leq s} \Aa_\theta (\Bar{\Lambda}(r))(x) \d (\rho_n - \rho)(r,\theta) \tag{L2} \label{eq:gradient_flow_L2} .
    \end{align*}
    Using that the sequence $(\rho_n)_{n \geq 0}$ has uniformly integrable first-order moments and that $\Aa_\theta(\Lambda_n(r))(x) \to \Aa_\theta(\Bar{\Lambda}(r))(x)$ uniformly on compact subsets, we get that~\cref{eq:gradient_flow_L1} vanishes.
    Approximating the function $(r, \theta) \mapsto \ind_{r \leq s} \Aa_\theta (\Bar{\Lambda}(r))(x)$ on a compact set by a continuous and bounded function, we similarly show that~\cref{eq:gradient_flow_L2} vanishes.
    Thus taking the limit $n \to +\infty$ we have shown for every $s \in [0,1]$:
    \begin{align*}
        \Bar{\Lambda}(s) = \Id & + \int_{[0,1] \times \Theta} \ind_{r \leq s} \Aa_\theta (\Bar{\Lambda}(r)) \d \rho(r, \theta) ,
    \end{align*}
    and hence $\Bar{\Lambda} = \Lambda = \Lambda_\rho[\mu]$.

    \proofpart{The adjoint variables $\mm$ are $\tau$-continuous on $\Ww_2^\COT$-bounded subsets.}
    Let $(\rho_n)_{n \geq 0}$ be a $\Ww_2^\COT$-bounded sequence in $\PLeb$ s.t.\@ $\rho_n \xrightarrow{\tau} \rho \in \PLeb$.
    For an input token distribution $\mu \in \Pp_c(\RR^d)$ and an input token $x \in \RR^d$, let us write $\Lambda_n = \Lambda_{\rho_n}[\mu]$ and $\mm_n = \mm_{\rho_n}[\mu, x]$ for the associated adjoint variable.
    Similarly as we have shown $\Lambda_n \to \Lambda = \Lambda_\rho[\mu]$ in $\Cc([0,1] \times S, \RR^d)$ for any compact set s.t. $\Supp(\mu) \subset S$, we prove $\mm_n \to \mm = \mm_{\rho}[\mu, x] \in \Cc^0([0,1], \Mm(S)^d)$.

    Indeed, as before, from the $\Ww_2^\COT$-boundedness it follows that the sequence $(\mm_n)_{n \geq 0}$ is bounded (in total variation).
    Then, using the linear growth of $\D_\Lambda \Aa_\theta$ w.r.t. $\theta$ in~\cref{eq:modified_attention_growth} and using that the sequence $(\rho_n)_{n \geq 0}$ has uniformly integrable first-order moments we get that, for every $\eps > 0$, there exists $C \geq 0$ such that for any $n \geq 0$ and any $s_1 \leq  s_2 \in [0,1]$:
    \begin{align*}
        \left\| \mm_n(s_1) - \mm_n(s_2) \right\|
        \leq \eps + C |s_2 - s_1| .
    \end{align*}
    Thus, the curves $(\mm_n)_{n \geq 0}$ are equicontinuous and, by Arzel\`a--Ascoli's theorem, we have (up to a subsequence) that $\mm_n \to \Bar{\mm} \in \Cc([0, 1], \Mm(S)^d)$ (here convergence takes place w.r.t. the narrow topology of $\Mm(S)^d$, which is metrizable, for which bounded sets in total variation are relatively compact and hence for which Ascoli's theorem can be applied).

    We will show that $\Bar{\mm} = \mm = \mm_\rho[\mu, x]$ is the adjoint variable associated to $\rho$, $\mu$ and $x$.
    To this end, considering a test function $\Gamma \in \Cc^0(S, \RR^d)$, we will show that for $s \in [0,1]$:
    \begin{align*}
        \left< \Bar{\mm}(s), \Gamma \right> = \left< \Bar{\mm}(1), \Gamma \right> + \int_s^1 \int_\Theta \left< \Bar{\mm}(r), \D_\Lambda \Aa_\theta(\Lambda(r)) \cdot \Gamma \right> \d \rho(\theta | r) \d r .
    \end{align*}
    Indeed, for $n \geq 0$, we have by the definition of $\mm_n$ (\cref{eq:backward_ODE}):
    \begin{align*}
        \left< \mm_n(s), \Gamma \right> = \left< \mm_n(1), \Gamma \right> & + \int_s^1 \int_\Theta \left< \mm_n(r), \D_\Lambda \Aa_\theta(\Lambda_n(r)) \cdot \Gamma \right> \d \rho_n(\theta | r) \d r \\
        = \left< \mm_n(1), \Gamma \right> & + \int_s^1 \int_\Theta \left< \Bar{\mm}(r), \D_\Lambda \Aa_\theta(\Lambda(r)) \cdot \Gamma \right> \d \rho(\theta | r) \d r \\
        & + \int_s^1 \int_\Theta \left< \mm_n(r), \D_\Lambda \Aa_\theta(\Lambda(r)) \cdot \Gamma \right> \d (\rho_n - \rho)(\theta | r) \d r \tag{L3} \label{eq:gradient_flow_L3} \\
        & + \int_s^1 \int_\Theta \left< \mm_n(r), \left( \D_\Lambda \Aa_\theta(\Lambda_n(r))  - \D_\Lambda \Aa_\theta(\Lambda(r)) \right) \cdot \Gamma \right> \d \rho_n(\theta | r) \d r \tag{L4} \label{eq:gradient_flow_L4} \\
        & + \int_s^1 \int_\Theta \left< \mm_n(r) - \Bar{\mm}(r), \D_\Lambda \Aa_\theta(\Lambda(r)) \cdot \Gamma \right> \d \rho(\theta | r) \d r , \tag{L5} \label{eq:gradient_flow_L5}
    \end{align*}
    and to prove the result it suffices to show that~\cref{eq:gradient_flow_L3,eq:gradient_flow_L4,eq:gradient_flow_L5} vanish when $n \to +\infty$.
    As before, using that the sequence $(\rho_n)_{n \geq 0}$ has uniformly integrable first-order moments and that $\D_\Lambda \Aa_\theta$ has linear growth w.r.t. $\theta$ (\cref{eq:modified_attention_growth}), we can restrict the integrand to $\theta$ in a compact set $K \subset \Theta$.
    First, to show that~\cref{eq:gradient_flow_L5} vanishes, observe that the integrand is bounded and vanishes for every $r \in [0, 1]$.
    For~\cref{eq:gradient_flow_L4},
    observe that the map $\Lambda \mapsto \D_\Lambda \Aa_\theta(\Lambda)$ is Lipschitz (locally w.r.t. $\Lambda, \theta$). Together with the fact that $\mm_n(r)$ is bounded in total variation this gives that the integrand in~\cref{eq:gradient_flow_L4} converges uniformly to $0$ on $[0,1] \times K$.
    Finally, to show that~\cref{eq:gradient_flow_L3} vanishes, observe that the sequence of maps
    $$
    (r, \theta) \mapsto \left< \mm_n(r), \D_\Lambda \Aa_\theta(\Lambda(r)) \cdot \Gamma \right>
    $$
    converge pointwise to the map
    $$
    (r, \theta) \mapsto \left< \Bar{\mm}(r), \D_\Lambda \Aa_\theta(\Lambda(r)) \cdot \Gamma \right>
    $$
    and are uniformly Lipschitz on $[0,1] \times K$ (this follows from the fact that $\mm_n(r)$ are solutions to ODEs and that the map $\Aa_\theta$ is smooth w.r.t. $\theta$).
    Hence this sequence converges uniformly and~\cref{eq:gradient_flow_L3} vanishes since $\rho_n \to \rho$ narrowly.

    \proofpart{$| \nabla \Ll |$ is $\tau$-continuous on $\Ww_2^\COT$-bounded subsets.}
    Finally, we show that the upper-gradient $| \nabla \Ll |$ defined in~\cref{thm:upper_gradient} is $\tau$-continuous on $\Ww_2^\COT$-bounded subsets.
    For every index $i \in \{1, ..., N \}$, for the token distribution $\mu^i$ and the input token $x^i$, with the same notation as previously, we have shown that $\Lambda^i_n \to \Lambda^i \in \Cc^0([0,1] \times S, \RR^d)$ (w.r.t.\@ uniform convergence) and $\mm^i_n \to \mm^i \in \Cc^0([0,1], \Mm(S)^d)$ (w.r.t.\@ narrow convergence).
    Let
    \[
        G^i_n(r,\theta) \eqdef \D_\theta \Aa^i_\theta (\Lambda^i_n(r))^* \mm^i_n(r),
        \qquad
        G^i(r,\theta) \eqdef \D_\theta \Aa^i_\theta (\Lambda^i(r))^* \mm^i(r).
    \]
    For indices $i, j \in \{ 1, ..., N \}$, note that the sequence of maps
    $(r, \theta) \mapsto \left< G^i_n(r,\theta),G^j_n(r,\theta) \right>$
    converges pointwise to
    $(r, \theta) \mapsto \left< G^i(r,\theta),G^j(r,\theta) \right>$
    and is locally Lipschitz (uniformly over $n$).
    Hence it converges locally uniformly.
    Moreover, using~\cref{lem:A_theta_param_differential}, this sequence can be bounded (uniformly over $n$) by a function of quadratic growth w.r.t. $\theta$.
    Therefore, using that the sequence $(\rho_n)_{n \geq 0}$ has uniformly integrable second-order moment and converges narrowly to $\rho$, we have:
    \begin{align*}
        \lim_{n \to \infty} \| \nabla \Ll [\rho_n] \|^2_{L^2(\rho_n)}
        &= \lim_{n \to \infty} \frac{1}{N^2} \sum_{1 \leq i, j \leq N}
        \int_{[0,1] \times \Theta}
        \left< G^i_n(r,\theta), G^j_n(r,\theta) \right>
        \d \rho_n(r,\theta) \\
        & = \frac{1}{N^2} \sum_{1 \leq i, j \leq N}
        \int_{[0,1] \times \Theta}
        \left< G^i(r,\theta), G^j(r,\theta) \right>
        \d \rho(r,\theta) \\
        & = \| \nabla \Ll[\rho] \|^2_{L^2(\rho)} ,
    \end{align*}
    which is the desired property.

    \proofpart{Upper-gradient coincides with the local slope}
    Applying \cite[Thm.~2.2.3 and Thm.~2.3.3]{ambrosio2008gradient}, we have shown the existence of gradient flow curves for the ``local slope'' defined in~\cref{eq:local_slope}.
    To conclude the result, we hence need to show this notion coincides with the upper-gradient defined in~\cref{eq:upper_gradient}.
    This is the content of the following~\cref{lem:local_slope}.
    
\end{proof}

\begin{lemma} \label{lem:local_slope}
    Let $\rho \in \PLeb$ be some parameter distribution.
    Then $\left| \nabla \Ll \right| (\rho)$ is the \emph{local slope} of the risk $\Ll$ at $\rho$, that is,
    \begin{align} \label{eq:local_slope}
        \left| \nabla \Ll \right| (\rho) = \limsup_{\rho' \to \rho} \frac{\left( \Ll(\rho) - \Ll(\rho') \right)_+}{\Ww_2^\COT(\rho, \rho')} .
    \end{align}
\end{lemma}

\begin{proof}
    First note that, by definition of $\left| \nabla \Ll \right| (\rho)$ in~\cref{eq:upper_gradient} and $\nabla \Ll(\rho)$ in~\cref{eq:gradient_field}, we have:
    $$
    \left| \nabla \Ll \right| (\rho) = \left\| \nabla \Ll(\rho) \right\|_{L^2(\rho)} .
    $$
    
    Let us first prove one side of~\cref{eq:local_slope} by considering $\rho_t \eqdef (\Id +t(0,\nabla \Ll (\rho)))_\# \rho \in \PLeb$, for $t \in \RR$.
    Then, using~\cref{prop:flow_differentiability} and performing the same calculations as in the proof of~\cref{thm:upper_gradient}, we have that the map $t \mapsto \Ll(\rho_t)$ is differentiable at $t=0$ and
    $$
    \Ll(\rho_t) = \Ll(\rho) + t \left\| \nabla \Ll(\rho) \right\|^2_{L^2(\rho)} + o(t).
    $$
    Also by construction, $\Ww_2^\COT(\rho, \rho_t) \leq t \left\| \nabla \Ll(\rho) \right\|_{L^2(\rho)}$, and hence
    $$
    \limsup_{\rho' \to \rho} \frac{\left( \Ll(\rho) - \Ll(\rho') \right)_+}{\Ww_2^\COT(\rho, \rho')} \geq \limsup_{t \to 0} \frac{\left( \Ll(\rho) - \Ll(\rho_t) \right)_+}{\Ww_2^\COT(\rho, \rho_t)} \geq \left\| \nabla \Ll(\rho) \right\|_{L^2(\rho)} .
    $$

    For the converse inequality, let us denote by $\left| \partial \Ll \right|(\rho)$ the right-hand side of~\cref{eq:local_slope}.
    Then, for $\eps > 0$, using the definition of $\left| \partial \Ll \right|(\rho)$ and the $\Ww_2^\COT$-continuity of the map $\rho \mapsto \left\| \nabla \Ll(\rho) \right\|_{L^2(\rho)}$ (in the proof of~\cref{thm:gradient_flow_existence}, we showed it is narrowly continuous on $\Ww_2^\COT$-bounded subsets, this implies $\Ww_2^\COT$-continuity), we can find $\rho' \in \PLeb$ s.t.
    $$
    \frac{(\Ll(\rho)-\Ll(\rho'))_+}{\Ww_2^\COT(\rho, \rho')} \geq \left| \partial \Ll \right|(\rho)-\eps
    $$
    and for every $\rho'' \in \PLeb$ with $\Ww_2^\COT(\rho, \rho'') \leq \Ww_2^\COT(\rho, \rho')$ it holds:
    $$
    \left| \left\| \nabla \Ll(\rho) \right\|_{L^2(\rho)} - \left\| \nabla \Ll(\rho'') \right\|_{L^2(\rho'')}  \right| \leq \eps.
    $$
    Considering a constant speed geodesic $(\rho_t)_{t \in [0,1]}$ between $\rho = \rho_0$ and $\rho' = \rho_1$, we have by~\cref{thm:upper_gradient}:
    $$
    \Ll(\rho') \geq \Ll(\rho) - \Ww_2^\COT(\rho, \rho') \int_0^1 \left\| \nabla \Ll(\rho_t) \right\|_{L^2(\rho_t)} \d t  \geq \Ll(\rho) - \Ww_2^\COT(\rho, \rho') \left( \left\| \nabla \Ll(\rho) \right\|_{L^2(\rho)} + \eps \right) .
    $$
    By construction of $\rho'$, this shows $\left| \partial \Ll \right|(\rho) \leq  \left\| \nabla \Ll(\rho) \right\|_{L^2(\rho)} + 2 \eps$.
\end{proof}

\subsubsection{Uniqueness of gradient flow curves}

We now prove the uniqueness of the gradient flow curves, whose existence was previously shown.

\begin{theorem} \label{thm:gradient_flow_uniqueness}
    Let the Attention map be defined by~\cref{eq:attention_modified} and let the training risk $\Ll : \PLeb \to \RR$ be defined by~\cref{eq:risk}.
    For any initialization $\rho_0 \in \PLeb$, the gradient flow curve $(\rho_t)_{t \geq 0}$ starting from $\rho_0$ is unique.
\end{theorem}

\begin{proof}
    Let $(\rho_t)_{t \geq 0}$ and $(\rho_t')_{t \geq 0}$ be two gradient flow curves starting from $\rho_0$ in the sense of \cref{def:gradient_flow}.
    Then using~\cite[Lem.~2.3]{barboni2024understanding}, we have for (a.e.) $t > 0$:
    $$
    \frac{\d}{ \d t} \Ww_2^\COT(\rho_t, \rho_t')^2 = 2 \int_0^1 \int_{\Theta \times \Theta} \left< \theta - \theta', \nabla \Ll[\rho_t](s, \theta) - \nabla \Ll[\rho_t'](s, \theta') \right> \d \gamma(\theta, \theta'|s) \d s ,
    $$
    where $\gamma \in \Gamma^\Leb_o(\rho_t, \rho_t')$ is any optimal (conditional) coupling between $\rho_t$ and $\rho_t'$.
    Considering a time horizon $T \geq 0$ and defining $\Ee \eqdef \sup_{t \in [0,T]} \left\{ \Ee_2(\rho_t) + \Ee_2(\rho_t') \right\}$, we have by~\cref{lem:gradient_field_lipschitz} that there exists a constant $C = C(\Ee)$ s.t. it holds for $t \in [0,T]$, $s \in [0,1]$ and $\theta, \theta' \in \Theta$:
    $$
    \left\| \nabla \Ll[\rho_t](s, \theta) -  \nabla \Ll[\rho_t'](s, \theta') \right\| \leq C  \left( \| \theta - \theta'\| +  (1+\|\theta\|)\Ww_2^\COT(\rho_t, \rho_t') \right) .
    $$
    Combining this with the previous equation, using Cauchy-Schwarz and a Grönwall argument gives for every $t \in [0,T]$:
    $$
    \Ww_2^\COT(\rho_t, \rho_t') \leq e^{C t} \Ww_2^\COT(\rho_0, \rho_0') = 0,
    $$
    which leads to the desired result.
\end{proof}

The above result relies on the two following lemmas, showing local Lipschitz regularity of the adjoint variables and of the gradient field w.r.t. the parameterization.

\begin{lemma} \label{lem:adjoint_lipschitz}
    Fix an input token distribution $\mu \in \Pp_c(\RR^d)$ and an input token $x \in \RR^d$.
    Let $S$ be some compact set s.t.\@ $\Supp(\mu) \cup \{x\} \subset S$.
    Then the associated adjoint-variable map:
    \begin{align*}
        \rho \in \PLeb \mapsto \mm_\rho \in \Cc^0([0,1], \Mm(S)^d)
    \end{align*}
    is locally Lipschitz w.r.t. $\rho$.
\end{lemma}

\begin{proof}
    Consider an input token distribution $\mu \in \Pp_c(\RR^d)$ and an input token $x \in \RR^d$ as in the statement.
    Let $\rho, \rho' \in \PLeb$ be two parameter distributions and define $\Ee = \Ee_2(\rho) + \Ee_2(\rho')$.
    Denote respectively by $\Lambda$, $\Lambda'$, $\mm$ and $\mm'$ the flow maps and adjoint variables associated with the parameterizations $\rho$ and $\rho'$.
    Recall that, by definition, those respectively satisfy
    $$
    \left\{
    \begin{array}{rcl}
        \frac{\d}{\d s} \Lambda(s) & = & \Aa_{\rho(.|s)}(\Lambda(s))  \\
        \Lambda(0) & = & \Id 
    \end{array}
    \right.
    , \quad 
    \left\{
    \begin{array}{rcl}
        \frac{\d}{\d s} \Lambda'(s) & = & \Aa_{\rho'(.|s)}(\Lambda'(s))  \\
        \Lambda'(0) & = & \Id 
    \end{array}
    \right. ,
    $$
    and
    $$
    \left\{
    \begin{array}{rcl}
        \frac{\d}{\d s} \mm(s) & = & -\D_\Lambda \Aa_{\rho(.|s)}(\Lambda(s))^* \mm(s)  \\
        \mm(1) & = & \nabla \ell(\Lambda(1,x)) \, \delta_x
    \end{array}
    \right.
    , \quad 
    \left\{
    \begin{array}{rcl}
        \frac{\d}{\d s} \mm'(s) & = & -\D_\Lambda \Aa_{\rho'(.|s)}(\Lambda'(s))^* \mm'(s)  \\
        \mm'(1) & = & \nabla \ell(\Lambda'(1,x)) \, \delta_x
    \end{array}
    \right. ,
    $$
    with $\ell$ the loss associated to the problem.

    Using the smoothness of $\ell$ and the local Lipschitz regularity of the flow map w.r.t. the parameterization (\cref{lem:flow_locally_lipschitz}), we already have that there exists a constant $C = C(\Ee)$ s.t.
    $$
        \left\| \mm(1) - \mm'(1) \right\| \leq C \Ww_2^\COT(\rho, \rho') .
    $$
	    Then, for (a.e.) $s \in [0,1]$, set
	    \[
	    \begin{aligned}
	        \RN{1} &\eqdef \left\| \D_\Lambda \Aa_{\rho(.|s)}(\Lambda(s)) \right\|, \\
	        \RN{2} &\eqdef
	        \left\| \D_\Lambda \Aa_{\rho'(.|s)}(\Lambda(s))
	        - \D_\Lambda \Aa_{\rho(.|s)}(\Lambda(s)) \right\|, \\
	        \RN{3} &\eqdef
	        \left\| \D_\Lambda \Aa_{\rho'(.|s)}(\Lambda'(s))
	        - \D_\Lambda \Aa_{\rho'(.|s)}(\Lambda(s)) \right\|.
	    \end{aligned}
	    \]
	    We have
	    \begin{align*}
	        \left\| \frac{\d}{ \d s} (\mm - \mm')(s) \right\|
	        &\leq \RN{1} \left\| (\mm - \mm')(s) \right\|
	        + \left(\RN{2}+\RN{3}\right) \left\| \mm'(s) \right\| .
	    \end{align*}
	    We will thus proceed by estimating $\RN{1}$, $\RN{2}$ and $\RN{3}$.

    For $\RN{1}$, using the linear growth of $\D_\Lambda \Aa_\theta$ w.r.t. $\theta$ in~\cref{eq:modified_attention_growth}, we directly obtain
   $$
   \RN{1}
   \leq
    \int_\Theta \left\| \D_\Lambda \Aa_\theta(\Lambda(s)) \right\| \d \rho(\theta|s)
   \leq
   C \left( 1+ \Ee_2(\rho(.|s)) \right)^{1/2},
   $$
    for some constant $C = C(\Ee)$.
    
    For $\RN{2}$, consider an optimal coupling $\gamma \in \Gamma(\rho(.|s), \rho'(.|s))$ between $\rho(.|s)$ and $\rho'(.|s)$.
    Then by the definition of $\Aa_\rho$ and properties of the Bochner integral, we have:
    $$
    \RN{2} \, \leq \int_{\Theta \times \Theta} \left\| \D_\Lambda \Aa_\theta (\Lambda(s)) - \D_\Lambda \Aa_{\theta'} (\Lambda(s))  \right\| \d \gamma(\theta, \theta') .
    $$
    But then, thanks to the linear growth of $\D^2_{\Lambda, \theta} \Aa_\theta(\Lambda)$ in~\cref{eq:differential_A}, we see that we can find a constant $C = C(\Ee)$ s.t. it holds, for every $\theta, \theta' \in \Theta$,
    $$
    \left\| \D_\Lambda \Aa_\theta (\Lambda(s)) - \D_\Lambda \Aa_{\theta'} (\Lambda(s))  \right\| \leq C \| \theta - \theta' \| (1+ \| \theta \| + \| \theta' \|) .
    $$
    Integrating w.r.t. $\gamma$ and using Cauchy-Schwarz then gives:
    $$
    \RN{2} \, \leq C  \Ww_2(\rho(.|s), \rho'(.|s)) \left( 1+ \Ee_2(\rho(.|s))+ \Ee_2(\rho'(.|s)) \right)^{1/2} .
    $$

    For $\RN{3}$, observe in~\cref{eq:modified_attention_growth} that $\D^2_{\Lambda, \Lambda} \Aa_\theta(\Lambda)$ has quadratic growth w.r.t.\@ $\theta$.
    Hence, there exists a constant $C = C(\Ee)$ s.t. for every $\theta \in \Theta$
    $$
    \left\| \D_\Lambda \Aa_\theta (\Lambda'(s)) - \D_\Lambda \Aa_\theta (\Lambda(s)) \right\| \leq C \left\| \Lambda(s) - \Lambda'(s) \right\| \, (1+ \| \theta \|^2) .
    $$
    Integrating this w.r.t. $\rho'(.|s)$ and using the local Lipschitz regularity of the flow map (\cref{lem:flow_locally_lipschitz}) then gives
    $$
    \RN{3} \, \leq C  \Ww_2^\COT(\rho, \rho')  \left( 1+\Ee_2(\rho'(.|s)) \right) .
    $$

    Using the previous estimates on $\RN{1}$, $\RN{2}$ and $\RN{3}$ gives:
    $$
    \left\| \frac{\d}{ \d s} (\mm - \mm')(s) \right\| \leq
    C \left(
    \begin{array}{c}
        (1+\Ee_2(\rho(.|s)))^{1/2}
    \left\| (\mm - \mm')(s) \right\|  \\[0.5em]
         + \left( 1+ \Ee_2(\rho(.|s)) + \Ee_2(\rho'(.|s)) \right)^{1/2} \Ww_2(\rho(.|s), \rho'(.|s)) \\[0.5em]
         + (1+\Ee_2(\rho'(.|s)))\Ww_2^\COT(\rho, \rho')
    \end{array}
    \right) .
    $$
    for some constant $C = C(\Ee)$.
    Hence we obtain the desired result using the initial condition on $\mm(1)$ and $\mm'(1)$ as well as a Grönwall argument.
\end{proof}

\begin{lemma} \label{lem:gradient_field_lipschitz}
    Let $\nabla \Ll[\rho]$ be the velocity field defined in~\cref{eq:gradient_field}.
    Then, for $\rho, \rho' \in \PLeb$, there exists a constant $C = C( \Ee_2(\rho)+ \Ee_2(\rho'))$ s.t. for $s \in [0,1]$ and every $\theta, \theta' \in \Theta$ it holds:
    $$
    \left\| \nabla \Ll[\rho](s, \theta) -  \nabla \Ll[\rho'](s, \theta') \right\| \leq C  \left( \| \theta - \theta'\| +  (1+\|\theta \|) \Ww_2^\COT(\rho, \rho') \right) . 
    $$
\end{lemma}

\begin{proof}
    Consider $\rho, \rho' \in \PLeb$ and $\Ee = \Ee_2(\rho) + \Ee_2(\rho')$.
    We have for $s \in [0,1]$ and $\theta, \theta' \in \Theta$:
    $$
    \left\| \nabla \Ll[\rho](s, \theta) -  \nabla \Ll[\rho'](s, \theta') \right\|
    \leq \frac{1}{N} \sum_{i=1}^N
    \left\|
    \D_\theta \Aa^i_\theta (\Lambda^i_\rho(s))^* \mm^i_\rho(s) - \D_\theta \Aa^i_{\theta'} (\Lambda^i_{\rho'}(s))^* \mm^i_{\rho'}(s)
    \right\| .
    $$
    For this reason we now focus on a single index $i \in \left\{1, ..., N \right\}$ and drop the dependence w.r.t. $i$.
    We also denote by $\Lambda$, $\Lambda'$, $\mm$ and $\mm'$ the flow-maps and adjoint variables associated to $\rho$ and $\rho'$ respectively.
	    With this notation we obtain, for $s \in [0,1]$ and every $\theta, \theta' \in \Theta$,
	    \begin{align*}
	        \left\| \D_\theta \Aa_\theta (\Lambda(s))^* \mm(s)
	        - \D_\theta \Aa_{\theta'} (\Lambda'(s))^* \mm'(s) \right\|
	        \leq \RN{1}+\left(\RN{2}+\RN{3}\right)\left\| \mm'(s) \right\|,
	    \end{align*}
	    where
	    \[
	    \begin{aligned}
	        \RN{1} &\eqdef
	        \left\| \D_\theta \Aa_\theta(\Lambda(s)) \right\|
	        \left\| (\mm - \mm')(s) \right\|, \\
	        \RN{2} &\eqdef
	        \left\| \D_\theta \Aa_{\theta'}(\Lambda(s))
	        - \D_\theta \Aa_\theta(\Lambda(s)) \right\|, \\
	        \RN{3} &\eqdef
	        \left\| \D_\theta \Aa_{\theta'}(\Lambda'(s))
	        - \D_\theta \Aa_{\theta'}(\Lambda(s)) \right\|.
	    \end{aligned}
	    \]
    and we proceed by estimating $\RN{1}$, $\RN{2}$ and $\RN{3}$.

    For $\RN{1}$, using the boundedness of $\D_\theta \Aa$ w.r.t. $\theta$ as well as the local Lipschitz regularity of the adjoint variable map (\cref{lem:adjoint_lipschitz}) we directly obtain a constant $C = C(\Ee)$ s.t.
    $$
    \RN{1} \leq C  \Ww_2^\COT(\rho, \rho') . 
    $$
    For $\RN{2}$, observe in~\cref{eq:modified_attention_growth} that $\D_{\theta} \Aa_\theta(\Lambda)$ is Lipschitz w.r.t. $\theta \in \Theta$, hence
    $$
    \RN{2} \leq C  \| \theta - \theta' \| .
    $$
    For $\RN{3}$, using the linear growth of $\D^2_{\Lambda, \theta} \Aa_\theta(\Lambda)$ w.r.t. $\theta$ (\cref{eq:modified_attention_growth}) as well as local Lipschitz regularity of the flow (\cref{lem:flow_locally_lipschitz}), we have
    $$
    \RN{3} \leq C (1+\|\theta'\|) \| \Lambda(s) - \Lambda'(s) \| \leq C (1+\|\theta'\|) \Ww_2^\COT(\rho, \rho') .
    $$
    Finally, combining the estimates on $\RN{1}$, $\RN{2}$ and $\RN{3}$ we obtain:
    \begin{align*}
    &\left\| \D_\theta \Aa_\theta (\Lambda(s))^* \mm(s)
    - \D_\theta \Aa_{\theta'} (\Lambda'(s))^* \mm'(s) \right\| \\
    &\hspace{6em}
    \leq C \left( (1+\|\theta'\|) \Ww_2^\COT(\rho, \rho')
    + \left\| \theta-\theta' \right\| \right) ,
    \end{align*}
    for some constant $C = C(\Ee)$.
    This is the desired property.
\end{proof}

\subsubsection{Extension to the case of unbounded parameters}

We justify here the existence and uniqueness of gradient flow curves for the minimization of the training risk in the case of Attention layers of the form~\cref{eq:attention}.
In contrast with the modified Attention defined in~\cref{eq:attention_modified}, we thus get rid of the function $\Pi$ previously used to bound the parameter $V$.
This is achievable at the price of considering an initial parameter distribution whose marginal w.r.t. $V$ is compactly supported.

\begin{theorem} \label{thm:wellposedness_compact}
    Let Attention layers be defined by~\cref{eq:attention} and let the training risk $\Ll : \PLeb \to \RR$ be defined by~\cref{eq:risk}.
    Let $\rho_0 \in \PLeb$ be some initial parameter distribution such that $\Supp(\rho_0) \subset [0,1] \times \left\{ (V, Q, q) \in \Theta, \| V \| \leq R_0 \right\}$ for some $R_0 \geq 0$. 
    Then there exists a unique gradient flow curve $(\rho_t)_{t \geq 0}$ starting from $\rho_0$ in the sense of~\cref{def:gradient_flow}.
\end{theorem}

The proof of~\cref{thm:wellposedness_compact} relies on two intermediary results.
The first one shows that, for any time $T \geq 0$, we can provide estimates on the size of the support of the $V$-marginal of gradient flow curves.

\begin{lemma} \label{lem:energy_support_estimate}
    Consider the modified Attention of~\cref{eq:attention_modified} with $\Pi$ s.t.\@ $\| \Pi(V) \| \leq \| V \|$ for every $V \in \RR^{d \times d}$ (note that this in particular encompasses the normal Attention of~\cref{eq:attention}).
    Let $\rho_0$ be as in~\cref{thm:wellposedness_compact} and consider $T \geq 0$.
    Then there exists $\Ee_T, R_T \geq 0$ such that, if $(\rho_t)_{t \geq 0}$ is a gradient flow curve starting from $\rho_0$ in the sense of~\cref{def:gradient_flow}, it holds
    \begin{align*}
        \forall t \in [0, T], \quad
        \Ee_2(\rho_t) \leq \Ee_T, \quad \text{and} \quad
        \Supp(\rho_t) \subset [0,1] \times \left\{ (V, Q, q) \in \Theta, \| V \| \leq R_T \right\}.
    \end{align*}
\end{lemma}

\begin{proof}
    Let us start by proving the estimate on the second-order moment.
    This is actually a general property of gradient flows of bounded functionals that the parameters cannot blow up in finite time.
    Using the gradient flow equation~\cref{eq:gradient_flow_continuity} we have for a.e. $t \in [0,T]$:
    \begin{align*}
        \frac{\d}{\d t} \Ee_2(\rho_t) & = - 2 \int_{[0,1] \times \Theta} \left< \theta, \nabla \Ll[\rho_t](s,\theta) \right> \d \rho_t(s, \theta) .
    \end{align*}
    Thus
    \begin{align*}
        \frac{\d}{\d t} \Ee_2(\rho_t)
        \leq  2 \sqrt{\Ee_2(\rho_t)} \left\| \nabla \Ll[\rho_t] \right\|_{L^2(\rho_t)}
    \end{align*}
    and using a Grönwall argument and Cauchy-Schwarz gives for $t \in [0,T]$:
    \begin{align} \label{eq:moment_estimate}
        \Ee_2(\rho_t)
        \leq \left( \sqrt{\Ee_2(\rho_0)} + \left( T \int_0^T \left\| \nabla \Ll[\rho_{t'}] \right\|^2_{L^2(\rho_{t'})} \d t' \right)^{1/2} \right)^2
        \leq \left( \sqrt{\Ee_2(\rho_0)} + \sqrt{ T \Ll(\rho_0) } \right)^2 \defeq \Ee_T .
    \end{align}

    It remains to estimate the support of the $V$-marginal of $\rho_t$.

    The gradient field has linear growth and is locally Lipschitz in $\theta$.

    The Cauchy problem is
    \begin{align*}
        \dot X_t =
        \begin{pmatrix}
            0 \\
            -\nabla \Ll[\rho_t](X_t)
        \end{pmatrix}
    \end{align*}
    \noindent Here $t \in [0,T]$ and $X_0(s, \theta) = (s, \theta)$.
    It is then a consequence of classical results in the representation of solutions to continuity equations (see e.g.~\cite{ambrosio2008transport}) that, for every $t \in [0,T]$, we have
    $\rho_t = (X_t)_\# \rho_0$.
    Let us denote by $V_t(s,\theta)$ the $V$-component of $X_t(s, \theta)$.
    Then, taking the $V$-component in the above Cauchy problem, we have for a.e. $t \in [0, T]$:
    $$
    \frac{\d}{\d t} V_t(s, \theta) = - \frac{1}{N} \sum_{i=1}^N \D_V \Aa_{X_t(s, \theta)}^i (\Lambda^i_{\rho_t}(s))^* \mm^i_{\rho_t}(s) .
    $$
    In particular, leveraging the linearity of $\Aa_\theta$ w.r.t. the parameter $V$ and its boundedness w.r.t. the parameters $(Q,q)$ we have:
    $$
    \left\| \frac{\d}{\d t} V_t(s, \theta) \right\| \leq \frac{1}{N} \sum_{i=1}^N \| \Lambda^i_{\rho_t}(s) \| \| \mm^i_{\rho_t}(s) \|  .
    $$
    Using the forward ODE~\cref{eq:transformer_flow} and the backward ODE and considering that $\Aa_\theta$ has quadratic growth w.r.t. $\theta$  and linear growth w.r.t. $\Lambda$ we have:
    $$
    \forall i \in \{1, ..., N\}, \quad \forall (s,t) \in [0,1] \times [0,T], \quad
    \left\| \Lambda^i_{\rho_t}(s) \right\| \leq C_0 e^{C_0 \Ee(\rho_t)} \leq C_0 e^{C_0 \Ee_T} ,
    $$
    for some universal constant $C_0 \geq 0$.
    Similarly, using the backward ODE~\cref{eq:backward_ODE} and considering that $\D_\Lambda \Aa_\theta$ has quadratic growth w.r.t. $\theta$ and $\Lambda$ we obtain:
    $$
    \forall i \in \{1, ..., N\}, \quad \forall (s,t) \in [0,1] \times [0,T], \quad
    \left\| \mm^i_{\rho_t}(s) \right\| \leq C_1 e^{C_1 e^{C_1 \Ee_T} },
    $$
    for some universal constant $C_1 \geq 0$.
    Thus, if $\theta = (V, Q, q) \in \Theta$ is such that $\| V \| \leq R_0$, then using the two previous estimates as well as~\cref{eq:moment_estimate}, we have for any $t \in [0,T]$:
    \begin{align} \label{eq:support_estimate}
        \| V_t(s, \theta) \| \leq R_T \eqdef R_0 + T C_2(T) ,
    \end{align}
    where $C_2: \RR \to \RR$ is some continuous function.
\end{proof}

We can now proceed to the proof of~\cref{thm:wellposedness_compact}.

\begin{proof}[Proof of~\cref{thm:wellposedness_compact}]
    Let $T \geq 0$.
    We prove existence and uniqueness of solutions to the gradient flow equation up to time $T$.

    \proofpart{Uniqueness}

    From~\cref{lem:energy_support_estimate}, we can find a $R_T \geq 0$ such that, if $(\rho_t)_{t \in [0, +\infty)}$, $(\rho_t')_{t \in [0, +\infty)}$ are two gradient flow curves starting from $\rho_0$, then it holds for every $t \in [0,T]$.
    $$
    \forall t \in [0, T], \quad \Supp(\rho_t), \Supp(\rho_t') \subset [0,1] \times \left\{ (V, Q, q) \in \Theta, \| V \| \leq R_T \right\} .
    $$
    Denote by $\Tilde{\phi}$ the modified Attention map of~\cref{eq:attention_modified} with bounding function $\Pi : \RR^{d\times d} \to \RR^{d\times d}$ s.t. $\Pi(V) = V$ for $\| V \| \leq R_T+1$.
    Similarly, denote by $\Tilde{\Phi}$ the associated mean-field Attention in~\cref{eq:attention_meanfield}, by $\Tilde{\Ll}$ the associated risk in~\cref{eq:risk} and $\nabla \Tilde{\Ll}$ the associated gradient field of the risk in~\cref{eq:gradient_field}.
    Then by construction, we have for every $t \in [0,T]$ and a.e. $s \in [0,1]$ that $\Tilde{\Phi}_{\rho_t(.|s)} = \Phi_{\rho_t(.|s)}$ and hence $\Tilde{\Ll}(\rho_t) = \Ll(\rho_t)$. Also $\nabla \Tilde{\Ll}[\rho_t] = \nabla \Ll[\rho_t]$ on $[0,1] \times \left\{ (V, Q, q) \in \Theta, \| V \| \leq R_T \right\}$, hence $(\rho_t)_{t \in [0,T]}$ is a solution to the continuity equation:
    \begin{align*}
        \partial_t \rho_t - \nabla_\theta \cdot \left( \rho_t \nabla \Tilde{\Ll}(\rho_t) \right) = 0, \quad \text{on $[0,T] \times [0,1] \times \Theta$},
    \end{align*}
    i.e., the curve $(\rho_t)_{t \in [0,T]}$ is a gradient flow of $\Tilde{\Ll}$ in the sense of~\cref{def:gradient_flow}.
    The same arguments show that it is also the case for $(\rho_t')_{t \in [0,T]}$.
    Hence, by~\cref{thm:gradient_flow_uniqueness} we obtain $\rho_t = \rho_t'$ for every $t \in [0,T]$.

    \proofpart{Existence}
    Using the existence result in the case of the modified Attention (\cref{thm:gradient_flow_existence}), we have that there exists a solution $(\Tilde{\rho}_t)_{t \in [0, +\infty)}$ to the gradient flow equation for the modified risk $\Tilde{\Ll}$ starting from $\rho_0$.

    For $T' \geq 0$, let $R_{T'}$ be defined by~\cref{eq:support_estimate}.
    In particular, by~\cref{lem:energy_support_estimate}, we have that $\Supp(\Tilde{\rho}_t) \subset [0,1] \times \left\{ (V, Q, q) \in \Theta, \| V \| \leq R_t \right\}$ for every $t \geq 0$.
    Using the continuity of the map $T' \mapsto R_{T'}$, let us fix some $T' > T$ s.t. $R_{T'} < R_T+1$.
    Then by a similar argument as before, we have for $t \in [0, T']$ that $\nabla \Tilde{\Ll}[\Tilde{\rho}_t] = \nabla \Ll[\Tilde{\rho}_t]$ on $\Supp(\Tilde{\rho}_t)$.
    Hence $\Tilde{\rho}_t$ is a solution to the continuity equation:
    $$
    \partial_t \Tilde{\rho}_t - \nabla_\theta \cdot \left( \Tilde{\rho}_t \nabla \Tilde{\Ll}(\Tilde{\rho}_t) \right) = 0, \quad \text{on $[0,T'] \times [0,1] \times \Theta$} ,
    $$
    and is hence a gradient flow for the risk $\Ll$ up to time $T'$.
\end{proof}

\clearpage

\section{Transport equations with non-local velocities}
\label{sec:nonlocal_transport}

We investigate the well-posedness of transport equations with non-local velocity fields.
Precisely, given an initial distribution $\mu_0 \in \Pp(\RR^d)$ and a time-dependent velocity field $V_s : \Pp(\RR^d) \times \RR^d \to \RR^d$, we investigate the existence and uniqueness of solutions to the transport equation:
\begin{align}
    \partial_s \mu(s) + \div (\mu(s) V_s[\mu(s)]) = 0, \quad \text{with $\mu(0) = \mu_0$.}
\end{align}
We will consider the following set of assumptions on the time-dependent velocity $V$:

\begin{assumption} \label{ass:velocity}
    \begin{enumerate}
        \item For every $\mu \in \Pp_c(\RR^d)$ and $x \in \RR^d$ the map $s \mapsto V_s[\mu](x) \in \RR^d$ is measurable.
        \item There exists a function $C_0 \in L^1([0,1])$ s.t. for every $s \in [0,1]$ it holds for every radius $R > 0$ and every $\mu \in \Pp_c(\RR^d)$ supported in $B(0,R)$:
        \begin{align*}
            \left\| V_s[\mu] \right\|_{\Cc^0} \leq C_0(s) (1+R).
        \end{align*}

        \item For every radius $R > 0$ there exist functions $L_R, M_R \in L^1([0,1])$ s.t. for every $\mu \in \Pp_c(\RR^d)$ supported in $B(0,R)$ it holds:
        \begin{align*}
            \sup_{x,y \in \RR^d} \| V_s[\mu](x) - V_s[\mu](y)  \| \leq L_R(s) \| x-y \|
        \end{align*}
        and for every $\mu, \nu \in \Pp_c(\RR^d)$ supported on $B(0,R)$:
        \begin{align*}
            \| V_s[\mu] - V_s[\nu]  \|_{\Cc^0(B(0,R))} \leq M_R(s) \Ww_1(\mu, \nu).
        \end{align*}
    \end{enumerate}
\end{assumption}

Under those assumptions, one can show the existence and uniqueness of the transport equation with velocity-field $V$.
The following result is a time-dependent generalization of the result in~\cite[Thm.~2.3]{piccoli2015control}.
Recall that $\Cc^0_\co([0,1], \Pp_c(\RR^d))$ denotes the set of continuous paths of co-compactly supported probability measures.

\begin{theorem} \label{thm:wellposed}
    Assume $V$ satisfies~\cref{ass:velocity} and $\mu_0 \in \Pp_c(\RR^d)$. Then there exists a unique $\mu \in \Cc^0_\co([0,1], \Pp_c(\RR^d))$ with $\mu(0) = \mu_0$ that is a solution (in the sense of distribution) to the transport equation:
    \begin{align} \label{eq:continuity_V}
        \partial_s \mu(s) + \div(\mu(s) V_s[\mu(s)]) = 0, \quad \text{over $(0,1) \times \RR^d$}.
    \end{align}
    Moreover, if $\Lambda : [0,1] \times \RR^d \to \RR^d$ is the flow-map solution of the Cauchy problem:
    \begin{align} \label{eq:flow_map}
        \Lambda(s,x) = x + \int_0^s V_r[\mu(r)](\Lambda(r,x)) \d r, \quad \forall s \in [0,1], x \in \RR^d
    \end{align}
    then we have for every $s \in [0,1]$, $\mu(s) = \Lambda(s)_\# \mu_0$.
\end{theorem}

\begin{proof}
    The proof can be done following the lines of~\cite[Thm.~2.3]{piccoli2015control} or~\cite[Thm.~6.6]{geshkovski2023emergence} and adapting the arguments to handle time-dependent velocity fields.
    We propose here a different argument based on the well-posedness of the flow-map equation.
    In the rest of the proof, we consider a radius $R > 0$ s.t. $\Supp(\mu_0) \subset B(0,R)$.

    \proofpart{Existence and uniqueness of the flow-map}
    
    For $s \in [0,1]$, consider the time-dependent mapping $\Aa_s : \Cc^0(B(0,R), \RR^d) \to \Cc^0(B(0,R), \RR^d)$ defined for $\Gamma \in \Cc^0(B(0,R), \RR^d)$ by:
    \begin{align*}
        \Aa_s(\Gamma)(x) \eqdef V_s[\Gamma_\# \mu_0] (\Gamma(x)).
    \end{align*}
    We first show existence and uniqueness of a solution $\Gamma \in \Cc^0([0,1] \times B(0,R), \RR^d)$ to the Cauchy problem:
    \begin{align} \label{eq:flow_map_Lambda}
        \forall s \in [0,1], \quad \Gamma(s) = \Id + \int_0^s \Aa_r(\Gamma(r)) \d r .
    \end{align}
    Indeed, for $s \in [0,1]$, the mapping $\Aa_s : \Cc^0(B(0,R), \RR^d) \to \Cc^0(B(0,R), \RR^d)$ is well-defined since, for $\Gamma \in \Cc^0(B(0,R), \RR^d)$, $\Gamma_\# \mu_0$ has compact support and hence $\Aa_s(\Gamma) \in \Cc^0(B(0,R), \RR^d)$ by~\cref{ass:velocity}.
    Also $\Aa_s$ is locally Lipschitz, as if $\Gamma, \Gamma' \in \Cc^0(B(0,R), \RR^d)$ are such that $\| \Gamma \|_{\Cc^0}, \| \Gamma' \|_{\Cc^0} \leq N$ then $\Supp(\Gamma_\# \mu_0) \subset B(0, N)$ and for $x \in B(0,R)$:
    \begin{align*}
        \left\| \Aa_s(\Gamma)(x) - \Aa_s(\Gamma')(x) \right\|
         = & \left\| V_s [ \Gamma_\# \mu_0] (\Gamma(x)) - V_s [ \Gamma'_\# \mu_0] (\Gamma'(x)) \right\| \\
        \leq & \left\| V_s [ \Gamma_\# \mu_0] (\Gamma(x)) - V_s [ \Gamma_\# \mu_0] (\Gamma'(x)) \right\| + \left\| V_s [ \Gamma_\# \mu_0] (\Gamma'(x)) - V_s [ \Gamma'_\# \mu_0] (\Gamma'(x)) \right\| \\
        \leq &  L_N(s) \| \Gamma-\Gamma' \|_{\Cc^0} + M_N(s) \Ww_1(\Gamma_\# \mu_0, \Gamma'_\# \mu_0) \\
        \leq & \left( L_N(s) + M_N(s) \right) \| \Gamma-\Gamma' \|_{\Cc^0},
    \end{align*}
    where we used~\cref{ass:velocity}.
    The local Lipschitz constant $L_N$ and $M_N$ being integrable w.r.t. $s \in [0,1]$, we have that $\Aa$ satisfies the Carathéodory and the $L^1$-Lipschitz conditions in~\cref{ass:caratheodory}.
    Moreover, if $\Gamma \in \Cc^0(B(0,R), \RR^d)$, then $\Supp(\Gamma(s)_\# \mu_0) \subset B(0, \| \Gamma(s) \|_{\Cc^0})$ and hence using~\cref{ass:velocity}:
    $$
    \left\| \Aa_s(\Gamma) \right\|_{\Cc^0} \leq C_0(s) \left( 1+ \| \Gamma(s) \|_{\Cc^0} \right) .
    $$
    Hence $\Aa$ also satisfies the $L^1$-linear growth condition of~\cref{ass:caratheodory}.
    Thus, by~\cref{thm:caratheodory}, there exists a unique solution $\Gamma \in \Cc^0([0,1] \times B(0,R), \RR^d)$ to the Cauchy problem~\cref{eq:flow_map_Lambda}.
    Moreover, an application of Grönwall's lemma gives for every $s \in [0,1]$:
    \begin{align*}
        \| \Gamma(s) \|_{\Cc^0} \leq R + (1+R) (e^{\int_0^s C_0(r) \d r}-1) .
    \end{align*}

    \proofpart{Existence and uniqueness of solutions}
    
    Observe that the above defined flow-map $\Gamma \in \Cc^0(B(0,R), \RR^d)$ gives a solution of the transport equation~\cref{eq:continuity_V}. Indeed, for $s \in [0,1]$, defining $\mu(s) \eqdef  \Gamma(s)_\# \mu_0$, we have $\mu \in \Cc^0_\co([0,1], \Pp_c(\RR^d))$ and for any test function $\varphi \in \Cc^\infty_c([0,1] \times \RR^d)$ and any $x \in \RR^d$:
    \begin{align*}
        \varphi(1,\Gamma(1,x)) - \varphi(0,x) & = \int_0^1 \left( \partial_s  \varphi(s, \Gamma(s,x)) + \nabla \varphi(s, \Gamma(s,x)) \cdot V_s[\mu(s)](\Gamma(s,x)) \right) \d s
    \end{align*}
    and integrating w.r.t. $\mu_0$:
    \begin{align*}
        \int_{\RR^d} \varphi(1) \d \mu(1) - \int_{\RR^d} \varphi(0) \d \mu(0) = \int_0^1 \int_{\RR^d} \left( \partial_s  \varphi(s) + \nabla \varphi(s) \cdot V_s[\mu(s)] \right) \d \mu(s) \d s
    \end{align*}
    which is the definition of (weak) solutions to~\cref{eq:continuity_V}.

    Conversely, if $\mu \in \Cc^0_\co([0,1], \Pp_c(\RR^d))$ is a solution to~\cref{eq:continuity_V}, then the time-dependent velocity field $v : (s, x) \in [0,1] \times \RR^d \to V_s[\mu(s)](x)$ is Lipschitz with an integrable Lipschitz constant. Hence by Cauchy-Lipschitz theorem, the flow-map $\Lambda : [0,1] \times \RR^d \to \RR^d$ is uniquely defined and continuous and since $\mu$ is a solution to the continuity equation:
    \begin{align*}
        \partial_s \mu(s) + \div (\mu(s) v(s)) = 0.
    \end{align*}
    It follows from classical results in the representation of solutions to the transport equation that $\mu(s) = \Lambda(s)_\# \mu_0$ (see~\cite[Prop.~8.18]{ambrosio2008gradient}).
    Hence, when restricted to $B(0,R)$, $\Lambda \in \Cc^0(B(0,R), \RR^d)$ and by definition for every $x \in B(0,R)$ and every $s \in [0,1]$:
    \begin{align*}
        \Lambda(s,x) = x + \int_0^s V_r[\Lambda(r)_\# \mu_0](\Lambda(r,x)) \d r
    \end{align*}
    i.e., $\Lambda$ is a solution to~\cref{eq:flow_map_Lambda}, which by uniqueness implies $\Lambda = \Gamma$ on $B(0,R)$.
\end{proof}

\clearpage

\section{Theory of ODEs in Banach spaces} \label{sec:Banach_ODE}

We work extensively in this paper with weak solutions to \emph{Ordinary Differential Equations (ODEs)} in separable Banach spaces.
The theory of such \emph{Carathéodory} differential equations can be found in~\cite{aulbach1996integral,younes_shapes_2010,deimling2006ordinary}.
The following conditions and well-posedness results are taken from~\cite{aulbach1996integral}.

\begin{assumption} \label{ass:caratheodory}
    Consider a separable Banach space $\Xx$, an interval $I \subset \RR$ and a map $f : I \times \Xx \to \Xx$. 
    \begin{itemize}
        \item We say $f$ satisfies the \emph{Carathéodory conditions} if
        \begin{enumerate}
            \item for every $x \in \Xx$ the map $t \in I \mapsto f(t,x) \in \Xx$ is measurable,
            \footnote{Since we assume $\Xx$ is separable, we do not need to distinguish between weak measurability and strong measurability of Banach-space-valued functions. This is a consequence of Pettis' measurability theorem~\cite[Chap.~II]{diestel1977measures}}
            \item for a.e. $t \in I$, the map $x \in \Xx \mapsto f(t,x) \in \Xx$ is continuous.
        \end{enumerate} 
        
        \item We say $f$ is \emph{locally $L^1$-Lipschitz} if for every bounded subset $\Vv \subset \Xx$ there exists a function $L_\Vv \in L^1_\loc(I)$ such that for a.e.\@ $t \in I$ it holds:
        \begin{align*}
            \forall x,y \in \Vv, \quad \left\| f(t,x) - f(t,y) \right\| \leq L_\Vv(t) \| x-y \| .
        \end{align*}
        
        \item We say $f$ has $L^1$-linear growth if there exists a function $C \in L^1_\loc(I)$ such that for a.e.\@ $t \in I$ it holds:
        \begin{align*}
            \forall x \in \Xx, \quad \| f(t,x) \| \leq C(t)(1+ \| x \|).
        \end{align*}
    \end{itemize}
\end{assumption} 

\begin{theorem}[Carathéodory theorem {\cite[Thm.~2.4]{aulbach1996integral}}] \label{thm:caratheodory}
    Let $\Xx$ be a separable Banach space and $I \subset \RR$ be an interval.
    Assume $f : I \times \Xx \to \Xx$ satisfies the Carathéodory conditions, is locally $L^1$-Lipschitz and has $L^1$-linear growth (\cref{ass:caratheodory}).
    Then for any $x_0 \in \Xx$, there exists a unique solution $x \in \Cc^0_\loc(I, \Xx)$ to the Cauchy problem:
    \begin{align} \label{eq:cauchy_problem}
        \forall t \in I, \quad x(t) = x_0 + \int_0^t f(s, x(s)) \d s.
    \end{align}
    In particular, the map $I \ni t \mapsto x(t) \in \Xx$ is locally absolutely continuous.
\end{theorem}

We also extensively use Grönwall's lemma to provide a priori estimates on solutions.

\begin{lemma}[Grönwall's lemma] \label{lem:gronwall}
    Let $T \geq 0$ and let $u \in \Cc^0_\loc([0,T), \RR_+)$ be such that it holds:
    \begin{align*}
        \forall t \in [0,T), \quad u(t) \leq u(0) + \int_0^t (a(s) u(s) + b(s)) \d s 
    \end{align*}
    for some $a, b \in L^1_\loc([0,T), \RR_+)$ with $a$ non-negative.
    Then for all $t \in [0,T)$ it holds:
    \begin{align*}
        u(t) \leq u(0) e^{\int_0^t a(s) \d s} + \int_0^t b(s) e^{\int_s^t a(r) \d r} \d s.
    \end{align*}
\end{lemma}

\end{document}